\documentclass[11pt,a4paper,titlepage,twoside
]{book}

\usepackage[utf8x]{inputenc}
\usepackage[T1]{fontenc}
\usepackage{kpfonts}

\usepackage[a4paper,includeheadfoot,pdftex,textwidth=16cm,textheight=24cm,
bottom=3.6cm]{geometry}
\usepackage[svgnames]{xcolor}%https://www.latextemplates.com/svgnames-colors
\usepackage{graphicx}

\usepackage[bookmarks=true,
pdfborder={0 0 1},%citebordercolor=violet, 
colorlinks=true,urlcolor=blue,citecolor=Purple, 
linkcolor=NavyBlue,hypertexnames=false]{hyperref}

\usepackage{enumitem}
\setlist{parsep=0pt} % decrease enumerate/itemize separation
\setlist[itemize,enumerate]{nolistsep,itemsep=1pt,topsep=1pt} 
\setlist{leftmargin=5mm}

\usepackage{fancybox}
\usepackage[Lenny]{fncychap}
\usepackage{fancyhdr}
\usepackage{amsmath}
\usepackage{amsfonts}
\usepackage{amssymb}
\usepackage{amsthm}
\usepackage{ upgreek }

\usepackage{bbm}

\usepackage{mathtools}%,xparse}
\usepackage{mdframed}

\usepackage{pgf,tikz}
\usetikzlibrary{matrix,automata,arrows,positioning,calc,patterns,fit,calc,
decorations.pathreplacing,decorations.pathmorphing,decorations.markings,matrix,
tikzmark,shadows.blur}
\usetikzlibrary{cd}
 \usepgflibrary{fpu}

\tikzset{
    mynode/.style={circle,fill=black,inner sep=1.5pt},
}

\usepackage[margin=10pt,font=small,labelfont=bf,
labelsep=endash]{caption}

\definecolor{ThmColor}{rgb}{0.9,0.9,0.995}
\definecolor{DefColor}{rgb}{0.85,0.94,0.98}
\definecolor{RemColor}{rgb}{0.95,0.9,0.95}
\definecolor{ExoColor}{rgb}{0.89,0.993,0.89}

\mdfdefinestyle{thmstyle}{backgroundcolor=ThmColor,nobreak}
\mdfdefinestyle{defstyle}{backgroundcolor=DefColor,nobreak} 
\mdfdefinestyle{remstyle}{backgroundcolor=RemColor} 
\mdfdefinestyle{exostyle}{backgroundcolor=ExoColor} 

\mdtheorem[style=thmstyle]{theorem}%
{Theorem}[section]

\mdtheorem[style=thmstyle]{proposition}[theorem]{Proposition}[section]
\mdtheorem[ntheorem,style=thmstyle]{corollary}[theorem]{Corollary}[section]
\mdtheorem[ntheorem,style=thmstyle]{lemma}[theorem]{Lemma}[section]
\mdtheorem[ntheorem,style=defstyle]{definition}[theorem]{Definition}[section]
\mdtheorem[ntheorem,style=defstyle]{notation}[theorem]{Notation}[section]
\mdtheorem[ntheorem,style=defstyle]{assumption}[theorem]{Assumption}[section]

\mdtheorem[ntheorem,style=remstyle]{example}[theorem]{Example}[section]
\mdtheorem[ntheorem,style=remstyle]{remark}[theorem]{Remark}[section]

\mdtheorem[ntheorem,style=exostyle]{exercise}[theorem]{Exercise}[section]

\usepackage{cleveref}
\crefname{exercise}{exercise}{exercises}
\usepackage{autonum} %numerotation intelligente

\input{sarajevo.sty}
\input{trees_Hopf.sty}
\input{graphs_gibbs_phi4.sty}
\begin{document}

\pagestyle{empty}
\newgeometry{margin=1in}

\hypersetup{pageanchor=false}

\thispagestyle{empty}

\vspace*{1cm}
\begin{center}

{\Huge\bfseries\scshape
Topics in Gaussian \\[1mm]
Wiener chaos expansion \\[1mm]
}%\\[8mm]

\vspace*{12mm}
{\large Nils Berglund}\\[2mm]
{\large Institut Denis Poisson -- UMR 7013}\\[2mm]
{\large Universit\'e d'Orl\'eans, Universit\'e de Tours, CNRS}

\vspace*{12mm}
{\Large Lecture notes\\[4mm]
44th Finnish Summer School on Probability and Statistics}\\[2mm]
{\large Lammi, May 2026}\\
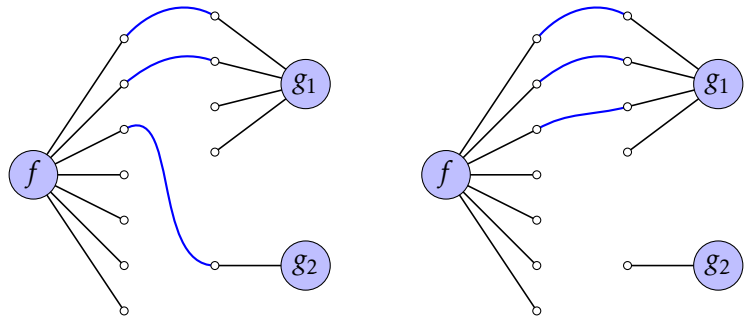
\vspace*{12mm}

\begin{figure}[h]
\begin{center}
\begin{tikzpicture}[>=stealth',main node/.style={circle,minimum
size=0.4cm,inner sep=2pt,fill=blue!25,draw},small node/.style={draw,circle,fill=white,minimum
size=3pt,inner sep=0pt},scale=1.2]
\draw[semithick] (0,0) -- (1,1.5);
\draw[semithick] (0,0) -- (1,1.0);
\draw[semithick] (0,0) -- (1,0.5);
\draw[semithick] (0,0) -- (1,0.0);
\draw[semithick] (0,0) -- (1,-0.5);
\draw[semithick] (0,0) -- (1,-1.0);
\draw[semithick] (0,0) -- (1,-1.5);

\draw[semithick] (3,1) -- (2,1.75);
\draw[semithick] (3,1) -- (2,1.25);
\draw[semithick] (3,1) -- (2,0.75);
\draw[semithick] (3,1) -- (2,0.25);

\draw[semithick] (3,-1) -- (2,-1);

\draw[thick,blue] (1,1.5) to[out=50,in=150] (2,1.75);
\draw[thick,blue] (1,1.0) to[out=40,in=160] (2,1.25);
\draw[thick,blue] (1,0.5) to[out=30,in=180] (2,-1);

\node[main node] (f) at (0,0) {$f$};

\node[main node] (g1) at (3,1) {$g_1$};
\node[main node] (g2) at (3,-1) {$g_2$};

\node[small node] (1) at (1,1.5) {};
\node[small node] (2) at (1,1.0) {};
\node[small node] (3) at (1,0.5) {};
\node[small node] (4) at (1,0.0) {};
\node[small node] (5) at (1,-0.5) {};
\node[small node] (6) at (1,-1.0) {};
\node[small node] (7) at (1,-1.5) {};

\node[small node] (8) at (2,1.75) {};
\node[small node] (9) at (2,1.25) {};
\node[small node] (10) at (2,0.75) {};
\node[small node] (11) at (2,0.25) {};

\node[small node] (12) at (2,-1) {};
\end{tikzpicture}
\hspace{10mm}
\begin{tikzpicture}[>=stealth',main node/.style={circle,minimum
size=0.4cm,inner sep=2pt,fill=blue!25,draw},small node/.style={draw,circle,fill=white,minimum
size=3pt,inner sep=0pt},scale=1.2]
\draw[semithick] (0,0) -- (1,1.5);
\draw[semithick] (0,0) -- (1,1.0);
\draw[semithick] (0,0) -- (1,0.5);
\draw[semithick] (0,0) -- (1,0.0);
\draw[semithick] (0,0) -- (1,-0.5);
\draw[semithick] (0,0) -- (1,-1.0);
\draw[semithick] (0,0) -- (1,-1.5);

\draw[semithick] (3,1) -- (2,1.75);
\draw[semithick] (3,1) -- (2,1.25);
\draw[semithick] (3,1) -- (2,0.75);
\draw[semithick] (3,1) -- (2,0.25);

\draw[semithick] (3,-1) -- (2,-1);

\draw[thick,blue] (1,1.5) to[out=50,in=150] (2,1.75);
\draw[thick,blue] (1,1.0) to[out=40,in=160] (2,1.25);
\draw[thick,blue] (1,0.5) to[out=30,in=195] (2,0.75);

\node[main node] (f) at (0,0) {$f$};

\node[main node] (g1) at (3,1) {$g_1$};
\node[main node] (g2) at (3,-1) {$g_2$};

\node[small node] (1) at (1,1.5) {};
\node[small node] (2) at (1,1.0) {};
\node[small node] (3) at (1,0.5) {};
\node[small node] (4) at (1,0.0) {};
\node[small node] (5) at (1,-0.5) {};
\node[small node] (6) at (1,-1.0) {};
\node[small node] (7) at (1,-1.5) {};

\node[small node] (8) at (2,1.75) {};
\node[small node] (9) at (2,1.25) {};
\node[small node] (10) at (2,0.75) {};
\node[small node] (11) at (2,0.25) {};

\node[small node] (12) at (2,-1) {};
\end{tikzpicture}
\end{center}
\end{figure}

\vspace*{27mm}
--- Version of July 30, 2026 ---\\[2mm]
% --- Version of \today\ ---\\[2mm]

\end{center}
\hypersetup{pageanchor=true}

\cleardoublepage

\pagestyle{fancy}
\fancyhead[RO,LE]{\thepage}
\fancyhead[LO]{\nouppercase{\rightmark}}
\fancyhead[RE]{\nouppercase{\leftmark}}
\cfoot{}
\setcounter{page}{1}
\pagenumbering{roman}
\restoregeometry

\tableofcontents

\cleardoublepage

%%%%%%%%%%%%%%%%%%%%%%%%%%%%%%%%%%%%%%%%%%%%%%%%%%%%%%%%%%%%%%%%%%%%%%%%%%%%%%%%

\chapter*{Preface}

These notes have been written for a series of lectures given at the 44th 
Finnish Summer School on Probability and Statistics in Lammi, Finland, from 25th to 
29th May, 2026. They contain an introduction to Wiener chaos decomposition in finite 
dimension, a construction of Gaussian fields on the torus, including white noise 
and the Gaussian free field, and applications to the $\Phi^4$ model. They do 
\emph{not} cover other important aspects of the topic, such as stochastic 
integration, stochastic PDEs and Malliavin calculus. 
Sections with a $*$ are more technical, and can safely be skipped in a 
first reading.

The material included in these notes is mostly based on the 
monograph~\cite{nualart2006malliavin}, the lecture 
notes~\cite{Hairer_Malliavin26,Tubaro_Zanella25}, and the monograph~\cite{NB_SPDE_book}.
Other useful resources on the topic  
include~\cite{Peccati_Taqqu_book,Janson_book_08,Sanz_Sole_book}.
This is a preliminary version of the notes, that may contain mistakes 
and typos. Feel free to let me know if you find any. 

Thanks are due to Dario Gasbarra for organising the Summer School and 
inviting me to give these lectures, thereby providing the motivation 
to compile these notes, as well as supporting 
institutions. 

\vfill

\cleardoublepage
%\newpage
\setcounter{page}{1}
\pagenumbering{arabic}

%%%%%%%%%%%%%%%%%%%%%%%%%%%%%%%%%%%%%%%%%%%%%%%%%%%%%%%%%%%%%%%%%%%%%%%%%%%%%%%%

\chapter{The one-dimensional case}
\label{chap:1d} 

We start this exposition with the very simple situation of a one-dimensional 
Gaussian random variable, since this allows us to introduce many objects that will 
become important in higher dimension in a relatively simple setting. 

%%%%%%%%%%%%%%%%%%%%%%%%%%%%%%%%%%%%%%%%%%%%%%%%%%%%%%%%%%%%%%%%%%%%%%%%%%%%%%%%

\section{Gaussian random variables}
\label{sec:Gaussian} 

Our fundamental objects are Gaussian random variables, whose definition we recall 
here. 

\begin{definition}[Gaussian random variable]
Let $\R$ be the real line, equipped with the $\sigma$-algebra $\cB$ 
of Borel sets and Lebesgue measure $\6x$. A random variable $X:\R\to\R$ is a 
(one-dimensional) Gaussian random variable with mean $m$ and variance $\sigma^2$ if 
its law is  
\begin{equation}
 \mu(\6x) = \frac{1}{\sqrt{2\pi\sigma^2}} \e^{-(x-m)^2/(2\sigma^2)} \6x\;.
\end{equation}
In that case, we write $X\sim\cN(m,\sigma^2)$.
\end{definition}

We summarise some fundamental properties of Gaussian random variables as follows. 

\begin{proposition}[Basic properties of normal random variables] 
\label{prop:Gaussian_basic} 
\begin{enumerate}
\item   One has $X\sim\cN(m,\sigma^2)$ if, and only if, $X = m + \sigma Y$ 
with $Y \sim \cN(0,1)$. 

\item   Assume $X\sim\cN(m_1,\sigma_1^2)$ and $Y\sim\cN(m_2,\sigma_2^2)$ 
are defined on a common probability space, and let $Z = X + Y$. Then 
$Z$ is Gaussian, with parameters 
\begin{equation}
 Z \sim \cN \bigpar{m_1+m_2, \sigma_1^2 + \sigma_2^2 + 2\cov(X,Y)}\;.
\end{equation} 

\item   Two Gaussian variables $X$ and $Y$ are independent if, and only if, 
they are uncorrelated, that is, $\cov(X,Y) = \expec{XY} - \expec{X}\expec{Y} = 0$. 
\end{enumerate}
\end{proposition}

The first property states that all one-dimensional Gaussian random variables are 
equivalent by an affine transformation. The second one states that Gaussian random 
variables are stable, and is at the core of the central limit theorem. The third 
property is only true for very special random variables: while independence 
always implies non-correlation, the converse is false in general.

Because of the first property, we will almost always focus on the case $\mu = 0$,
that is, when $X$ is \emph{centred}. Whenever possible, we will assume $\sigma^2 = 1$, 
although it will sometimes be useful to allow for different variances. 

Our main interest will be expectations of functions of Gaussian random variables. 
Assume $X\sim\cN(0,1)$, and let $f:\R\to\R$. Then 
\begin{equation}
 \expec{f(X)} = \int_{-\infty}^\infty f(x) \mu(\6x)\;,
\end{equation} 
provided the integral is absolutely convergent. 

\begin{example}[Laplace transform]
\label{example:Laplace} 
Let $f(x) = \e^{t x}$ for $t\in\R$. Then, using completion of squares 
to write 
\begin{equation}
 t x - \frac{x^2}{2} 
 = -\frac12 (x-t)^2 + \frac{t^2}{2}\;,
\end{equation} 
we find 
\begin{equation}
\label{eq:exp_moment_N} 
 \expec{\e^{t X}}
 = \int_{-\infty}^\infty 
 \e^{t x-x^2/2}
 \frac{\6x}{\sqrt{2\pi}} 
 = \e^{t^2/2} \int_{-\infty}^\infty 
 \e^{-(x-t)^2/2}
 \frac{\6x}{\sqrt{2\pi}} 
 = \e^{t^2/2}\;.
\end{equation} 
\end{example}

For general functions $f$, explicit expressions of $\expec{f(X)}$ are not available. 
One possible strategy to compute expectations efficiently is to compute expectations
of a set of appropriate basis functions. One choice of basis functions is given by 
monomials. 

\begin{proposition}[Moments of Gaussian random variables]
\label{prop:gaussian_moments} 
Let $X\sim\cN(0,1)$. 
For any $n\in\N$, one has 
\begin{equation}
 \expec{X^n} = 
 \begin{cases}
  (n-1)!! & \text{if $n$ is even\;,} \\
  0 & \text{if $n$ is odd\;,}
 \end{cases}
\label{eq:moment_Gaussian} 
\end{equation} 
where
\begin{equation}
 (n-1)!! = \prod_{k=0}^{n/2-1} (2k+1)
 = 1\cdot3\cdot5 \dots (n-3)(n-1)
\end{equation} 
is the double factorial. 
\end{proposition}

\begin{exercise}
Prove Proposition~\ref{prop:gaussian_moments} in two different ways:
\begin{enumerate}
\item   by using~\eqref{eq:exp_moment_N};
\item   by showing, using integration by parts, that 
\begin{equation}
 \expec{X^{n+1}} = n\expec{X^{n-1}}
 \qquad 
 \text{for all $n\geqs1$\;.}
\end{equation} 
\end{enumerate}
\end{exercise}

If $f$ admits an entire series expansion 
\begin{equation}
 f(x) = \sum_{n=0}^\infty a_n x^n 
\end{equation} 
with positive radius of convergence $R$, then Proposition~\ref{prop:gaussian_moments} 
allows to compute 
\begin{equation}
\label{eq:expec_series} 
  \expec{f(X)} = \sum_{n\geqs0} a_n \expec{X^n}\;.
\end{equation}

%%%%%%%%%%%%%%%%%%%%%%%%%%%%%%%%%%%%%%%%%%%%%%%%%%%%%%%%%%%%%%%%%%%%%%%%%%%%%%%%

\section{Hermite polynomials}
\label{sec:Hermite} 

Our main workhorse will be Hermite polynomials. In this section, we review 
several of their definitions, and how they are related to geometry/linear 
algebra, probability theory, analysis, algebra, and combinatorics. 

%%%%%%%%%%%%%%%%%%%%%%%%%%%%%%%%%%%%%%%%%%%%%%%%%%%%%%%%%%%%%%%%%%%%%%%%%%%%%%%%

\subsection{Gram--Schmidt orthogonalisation}
\label{ssec:Gram-Schmidt} 

One drawback of using moments to compute expectations is that the basis $(X^n)_{n\geqs0}$ 
is not orthogonal. Here orthogonality is defined with respect to the inner product 
\begin{equation}
  \pscal{f}{g} = \expec{fg}\;,
\end{equation} 
meaning in particular that independent random variables are orthogonal. Indeed,
Proposition~\ref{prop:gaussian_moments} shows that $X^n$ and $X^m$ are orthogonal, 
according to this definition, if and only if $n+m$ is odd. 

However, a basis can always be turned into an orthogonal basis by the 
\emph{Gram--Schmidt procedure}. Let $(v_n)_{n\geqs0}$ be an arbitrary basis of a vector 
space. Then the Gram--Schmidt procedure provides an orthogonal basis $(u_n)_{n\geqs0}$
defined inductively by $u_0 = v_0$ and
\begin{equation}
 u_n = v_n - \sum_{k=0}^{n-1} \frac{\pscal{v_n}{u_k}}{\pscal{u_k}{u_k}} u_k\;,
 \qquad n \geqs 1\;.
\end{equation} 
This means that $u_n$ is obtained by subtracting from $v_n$ its projection 
on the plane spanned by the $n-1$ first $u_k$. It is easy to show by 
induction that $u_n$ is orthogonal to all vectors $u_0, \dots, u_{n-1}$. 
Indeed, the base case $n = 0$ is trivially true, while for $n\geqs1$, 
one has, for any $\ell \in \set{0,\dots,n-1}$,  
\begin{equation}
 \pscal{u_n}{u_\ell} = \pscal{v_n}{v_\ell} 
 - \frac{\pscal{v_n}{u_\ell}}{\pscal{u_\ell}{u_\ell}} \pscal{u_\ell}{u_\ell}
 = 0\;.
\end{equation} 
Let us apply this procedure to the basis $(X^n)_{n\geqs0}$, denoting the 
resulting orthogonal basis by $(H_n(X))_{n\geqs0}$. The first steps are 
\begin{align}
 H_0(X) &= X^0 = 1\;, \\
 H_1(X) &= X - \frac{\pscal{X}{1}}{\pscal{1}{1}} 1 = X\;, \\
 H_2(X) &= X^2 - \frac{\pscal{X^2}{X}}{\pscal{X}{X}} X 
 - \frac{\pscal{X^2}{1}}{\pscal{1}{1}} 1
 = X^2 - 0 - \frac{\expec{X^2}}{\expec{1}} 1 
 = X^2 - 1\;, \\ 
 H_3(X) &= X^3 - \frac{\pscal{X^3}{H_2(X)}}{\pscal{H_2(X)}{H_2(X)}} H_2(X) 
 - \frac{\pscal{X^3}{X}}{\pscal{X}{X}} X 
 - \frac{\pscal{X^3}{1}}{\pscal{1}{1}} 1
 = X^3 - 0 - \frac{\expec{X^4}}{\expec{1}} X 
 = X^3 - 3X\;.
\end{align}

\begin{exercise}
Show by this method that 
\begin{equation}
H_4(X) = X^4 - 6X^2 + 3\;. 
\end{equation} 
Check that the random variables $(H_n(X))_{0\leqs n\leqs 4}$ are mutually 
independent. 
\end{exercise}

Table~\ref{tab:Hermite} shows the first $11$ Hermite polynomials. 
Clearly, while the Gram--Schmidt procedure does produce an orthogonal basis, 
the computations are not efficient. We are thus going to look for more 
efficient methods. One of them, which we discuss in the next section, uses the notion of \emph{cumulant expansion}. 

\begin{table}
\begin{center}
\begin{tabular}{|r|l|}
\hline 
$n$ & $H_n(x)$ \\
\hline 
$0$ & $1$ \\
$1$ & $x$ \\
$2$ & $x^2 - 1$ \\
$3$ & $x^3 - 3x$ \\
$4$ & $x^4 - 6x^2 + 3$ \\
$5$ & $x^5 - 10x^3 + 15x$ \\
$6$ & $x^6 - 15x^4 + 45x^2 - 15$ \\
$7$ & $x^7 - 21x^5 + 105x^3 - 105x$ \\
$8$ & $x^8 - 28x^6 + 210x^4 -420x^2 + 105$ \\
$9$ & $x^9 - 36x^7 + 378x^5 - 1260x^3 + 945$ \\
$10$ & $x^{10} -45x^8 + 630x^6 -3150x^4 + 472x^2 - 945$\\
\hline
\end{tabular}
\end{center}
\caption[]{List of the first Hermite polynomials.}
\label{tab:Hermite} 
\end{table}

\begin{remark}[Scaling conventions]
\label{rem:Hermite_conventions} 
One finds several conventions for Hermite polynomials in the literature. 
What we use here are the \lq\lq probabilists' Hermite polynomials\rq\rq. 
Another convention, called the \lq\lq physicists' Hermite polynomials\rq\rq, 
uses the scaled version  
\begin{equation}
\widetilde H_n(x) 
 = 2^{n/2} H_n(\sqrt{2}x)\;.
\end{equation} 
Yet another convention, used in~\cite{nualart2006malliavin}, 
is to multiply $H_n(x)$ by $1/n!$. 
\end{remark}

%%%%%%%%%%%%%%%%%%%%%%%%%%%%%%%%%%%%%%%%%%%%%%%%%%%%%%%%%%%%%%%%%%%%%%%%%%%%%%%%

\subsection{Hermite polynomials and cumulants}
\label{ssec:cumulants} 

\begin{definition}[Cumulant expansion]
Let $X$ be a random variable such that $\expec{\e^{tX}}$ exists for all $t$ 
in an open interval $(-\delta,\delta)$, and write 
\begin{equation}
\label{eq:X_moment_gen} 
 \expec{\e^{tX}} = \sum_{n\geqs0} \mu_n\frac{t^n}{n!}\;, 
 \qquad 
 \mu_n = \expec{X^n}\;.
\end{equation} 
Then the \emph{cumulant expansion} of $X$ is given by 
\begin{equation}
\label{eq:X_cumulant} 
 K_X(t) = \log\expec{\e^{tX}}
 = \sum_{n\geqs0} \kappa_n \frac{t^n}{n!}\;.
\end{equation} 
The coefficients $\mu_n$ are called \emph{moments} of $X$, while the 
$\kappa_n$ are called \emph{cumulants}. 
\end{definition}

In the case of a Gaussian $X\sim\cN(0,1)$, the situation is particularly simple.
Indeed, we have by~\eqref{eq:exp_moment_N}
\begin{equation}
 K_X(t) = \log(\e^{t^2/2}) = \frac{t^2}{2}\;,
\end{equation} 
so that the cumulants are given by 
\begin{equation}
\label{eq:cumulant_gaussian} 
 \kappa_n = 
 \begin{cases}
  1 & \text{if $n = 2$\;,} \\
  0 & \text{otherwise\;.}
 \end{cases}
\end{equation} 
Consider now the function 
\begin{equation}
\label{eq:generating_function} 
 G(t,x) 
 = \frac{\e^{tx}}{\expec{\e^{tX}}}
 = \e^{tx - t^2/2}\;.
\end{equation} 
where $t\geqs0$ and $x\in\R$.

\begin{proposition}[Generating function]
\label{prop:Hermite_generating}
$G$ is the generating function of Hermite polynomials, that is 
\begin{equation}
\label{eq:G_Hermite} 
 G(t,x) = \sum_{n\geqs0} \frac{t^n}{n!} H_n(x)\;.
\end{equation} 
\end{proposition}

We will proceed in several steps to prove this result. First, we 
make an easy observation: if $X\sim\cN(0,1)$ we have 
\begin{equation}
 1 = \expec{G(t,X)}
 = \sum_{n\geqs0} \frac{t^n}{n!} \expec{H_n(X)}\;.
\end{equation} 
Since this is valid for all $t$ in $(-\delta,\delta)$, uniqueness of coefficients 
of power series shows that
\begin{equation}
 \expec{H_n(X)} = \delta_{n0} = 
 \begin{cases}
  1 & \text{if $n=0$\;,}\\
  0 & \text{otherwise\;.}
 \end{cases}
\end{equation} 
In other terms, all $H_n(X)$ with $n\geqs1$ are centred, and therefore 
orthogonal to $H_0(X) = 1$. The following result is a generalisation of this
observation.

\begin{proposition}[Orthogonality of the $H_n(X)$]
\label{prop:Hermite_orthogonal}
For any $n, m\in\N_0$, one has 
\begin{equation}
\label{eq:Hermite_orthogonal} 
 \expec{H_n(X)H_m(X)} =
 n!\delta_{nm} = 
 \begin{cases}
  n! & \text{if $n = m$\;,} \\
  0 & \text{otherwise\;.}
 \end{cases}
\end{equation} 
\end{proposition}
\begin{proof}
We compute the expectation of $G(t,X)G(s,X)$ in two different ways. 
The first way is 
\begin{align}
 \expec{G(t,X)G(s,X)}
 &= \expec{\e^{tX-t^2/2}\e^{sX-s^2/2}} \\
 &= \e^{-(t^2+s^2)/2} \expec{\e^{(t+s)X}} \\
 &= \e^{-(t^2+s^2)/2} \e^{(t+s)^2/2} \\
 &= \e^{ts} \\
 &= \sum_{n\geqs0} \frac{t^ns^n}{n!}\;.
\end{align}
The second way to perform the computation is, using~\eqref{eq:G_Hermite}, 
\begin{equation}
 \expec{G(t,X)G(s,X)}
 = \sum_{n\geqs0} \sum_{m\geqs0}
 \frac{t^n}{n!} \frac{s^m}{m!} \expec{H_n(X)H_m(X)}\;.
\end{equation} 
Comparing the two obtained power series yields the result, by uniqueness 
of the coefficients of a series. 
\end{proof}

This result shows that the $H_n(X)$ defined via~\eqref{eq:G_Hermite} 
form \textit{an} orthogonal basis. It remains to show that this basis 
is identical with the one obtained by the Gram--Schmidt procedure. 
We do this with the following result. 

\begin{lemma}[Recursive relation between Hermite polynomials]
\label{lem:Hermite_recursive} 
For any $n\geqs0$, one has 
\begin{equation}
\label{eq:Hermite_recursive} 
 H_{n+1}(x) = x H_n(x) - H_n'(x)\;.
\end{equation} 
\end{lemma}
\begin{proof}
This follows from the relation 
\begin{equation}
 \frac{\partial}{\partial t} G(t,x) 
 = (x-t) G(t,x) 
 = xG(t,x) - \frac{\partial}{\partial x} G(t,x)\;.
\end{equation}
Indeed, we have 
\begin{equation}
 \frac{\partial}{\partial t} G(t,x) 
 = \sum_{n\geqs1} \frac{t^{n-1}}{(n-1)!} H_n(x) 
 = \sum_{n\geqs0} \frac{t^n}{n!} H_{n+1}(x)\;, 
\end{equation} 
while
\begin{equation}
 xG(t,x) - \frac{\partial}{\partial x} G(t,x)
 = \sum_{n\geqs0} \frac{t^n}{n!} \bigpar{xH_n(x) - H_n'(x)}\;.
\end{equation} 
Comparing the coefficients of the last two power series yields the 
result. 
\end{proof}

The recursive relation~\eqref{eq:Hermite_recursive} provides a quicker 
way to compute Hermite polynomials than the Gram--Schmidt procedure, 
starting with $H_0(x) = 1$. It also allows us to complete the proof 
of Proposition~\ref{prop:Hermite_generating}.

\begin{proof}[{\sc Proof of Proposition~\ref{prop:Hermite_generating}}]
We have already shown that the $(H_n)_{n\geqs0}$ form an orthogonal 
family. It remains to show that they co\"incide with the polynomials 
constructed by the Gram--Schmidt procedure. 
Evaluating~\eqref{eq:G_Hermite} at $t = 0$ yields 
$1 = G(0,x) = H_0(x)$. From~\eqref{eq:Hermite_recursive}, it follows 
by induction on $n$ that $H_n(x)$ has degree $n$, with $x^n$ having 
coefficient $1$, which is also the case for the $H_n$ obtained 
via the Gram--Schmidt procedure. 
\end{proof}

Another useful consequence of the expression~\eqref{eq:G_Hermite} of
the generating function is the following generalisation of Proposition~\ref{prop:Hermite_orthogonal}, which allows to 
transform products of Hermite functions into sums of such functions.
One can think of it as an analogue of product-sum formulas in trigonometry, 
which are useful to compute Fourier series. 

\begin{proposition}[Product--sum formula]
\label{prop:Hermite_product_sum}
For any $n, m\geqs0$, one has 
\begin{equation}
\label{eq:Hermite_product_sum} 
 H_n(x) H_m(x) = \sum_{p=0}^{n\wedge m} 
 p! \binom{n}{p} \binom{m}{p} H_{n+m-2p}(x)\;,
\end{equation} 
where $n\wedge m$ denotes the minimum of $n$ and $m$. 
\end{proposition}
\begin{proof}
We start by observing that 
\begin{align}
 G(t,x) G(s,x) 
 &= \e^{(t+s)x-(t^2+s^2)/2} \\
 &= \e^{ts} \e^{(t+s)x-(t+s)^2/2} \\
 &= \e^{ts} G(t+s,x)\;.
\end{align}
Expanding the exponential and $G(t+s,x)$, and using Newton's binomial 
formula yields 
\begin{align}
 G(t,x) G(s,x)
 &= \sum_{p=0}^\infty \frac{(ts)^p}{p!} 
 \sum_{q=0}^\infty \frac{(t+s)^q}{q!} H_q(x) \\
 &= \sum_{p=0}^\infty\sum_{q=0}^\infty \sum_{r=0}^q 
 \frac{t^{p+r}s^{p+q-r}}{p!r!(q-r)!} H_q(x)\;.
 \label{eq:H_product_sum1} 
\end{align}
On the other hand, we have 
\begin{equation}
 G(t,x) G(s,x)
 = \sum_{n\geqs0} \sum_{m\geqs0} \frac{t^n}{n!} \frac{s^m}{m!}
 H_n(x) H_m(x)\;.
\end{equation} 
The result follows by determining the coefficient of $t^ns^m$ 
in~\eqref{eq:H_product_sum1}, which is obtained by summing over 
all triples $(p,q,r)$ such that $p+r = n$ and $p+q-r = m$. 
\end{proof}

Note that when taking the expectation on both sides 
of~\eqref{eq:Hermite_product_sum} when $x = X$, we 
recover the orthogonalilty relation~\eqref{eq:Hermite_orthogonal}. 

\begin{exercise}
Use~\eqref{eq:Hermite_product_sum} to write $H_4(X)^2$ as a sum 
of Hermite polynomials, and compute $\expec{H_4(X)^3}$. 
\end{exercise}

%%%%%%%%%%%%%%%%%%%%%%%%%%%%%%%%%%%%%%%%%%%%%%%%%%%%%%%%%%%%%%%%%%%%%%%%%%%%%%%%

\subsection{Hermite polynomials and differential operators}
\label{ssec:differential} 

Lemma~\ref{lem:Hermite_recursive} has revealed a link between Hermite 
polynomials and differential operators. To make this connection more 
precise, we introduce the linear operators 
\begin{equation}
\label{eq:aadagger} 
 a = \frac{\6}{\6x}\;, \qquad 
 a^\dagger = x - \frac{\6}{\6x}\;, \qquad 
 \cL = -a^\dagger a
 = \frac{\6^2}{\6x^2} - x \frac{\6}{\6x}
\end{equation} 
acting on $\cC^\infty$ functions in the Hilbert 
space $\cH = L^2(\R,\mu(\6x))$, where $\mu(\6x)$ is the Gaussian 
measure 
\begin{equation}
\frac{1}{\sqrt{2\pi}} \e^{-x^2/2} \6x\;. 
\end{equation} 
The notation $a^\dagger$ is motivated by the following result.

\begin{lemma}
\label{lem:aadagger}
The operators $a$ and $a^\dagger$ are mutually adjoint, while $\cL$ 
is self-adjoint in $\cH$. Furthermore,
\begin{equation}
\label{eq:a_commutator} 
 a a^\dagger - a^\dagger a = \id\;.
\end{equation} 
\end{lemma}
\begin{proof}
An elegant proof of the first claim consists in rewriting $a$ and $a^\dagger$ as 
\begin{equation}
 (af)(x) = \e^{x^2/2} \biggpar{x + \frac{\6}{\6x}}\bigpar{\e^{-x^2/2}f(x)}\;, 
 \qquad 
 (a^\dagger f)(x) = -\e^{x^2/2} \frac{\6}{\6x} \bigpar{\e^{-x^2/2}f(x)}\;.
\label{eq:aadagger_inner} 
\end{equation} 
One indeed checks that this is equivalent to~\eqref{eq:aadagger} by applying 
Leibniz' rule. Now for any $f, g\in\cH$, integration by parts gives 
\begin{align}
 \pscal{a^\dagger f}{g}
 &= -\int_{-\infty}^\infty \frac{\6}{\6x} 
 \bigpar{\e^{-x^2/2}f(x)}g(x) \frac{\6x}{\sqrt{2\pi}} \\
 &= \e^{-x^2/2}f(x)g(x) \biggr\vert_{-\infty}^\infty
 + \int_{-\infty}^\infty \e^{-x^2/2}f(x)g'(x) \frac{\6x}{\sqrt{2\pi}} \\
 &= \pscal{f}{ag}\;,
\end{align} 
since the boundary term vanishes because $\pscal{f}{g} < \infty$ by 
the Cauchy--Schwarz inequality. As a consequence, we also have 
\begin{equation}
 -\pscal{f}{\cL g} 
 = \pscal{f}{a^\dagger a g} 
 = \pscal{af}{ag} 
 = \pscal{a^\dagger a f}{g}
 = -\pscal{\cL f}{g}\;,
\end{equation} 
which implies that $\cL$ is self-adjoint. Finally, 
\begin{equation}
 (aa^\dagger f)(x)
 = \frac{\6}{\6x} \bigpar{xf(x) - f'(x)}
 = f(x) + xf'(x) - f''(x) 
 = (a^\dagger a f)(x) + f(x)\;,
\end{equation} 
which proves~\eqref{eq:a_commutator}.
\end{proof}

The link with Hermite polynomials is as follows.

\begin{corollary}[Hermite polynomials and differential operators]
\label{cor:Hermite_differential}
The Hermite polynomials are eigenfunctions of $\cL$. More precisely,
\begin{equation}
\label{eq:Hermite_eigenvalue} 
 (\cL H_n)(x) = -n H_n(x) 
 \qquad 
 \forall n\geqs 0\;.
\end{equation} 
Furthermore, one has 
\begin{equation}
\label{eq:Hermite_aadagger} 
 a^\dagger H_{n-1} = H_n\;, 
 \qquad 
 a H_n = n H_{n-1}
 \qquad 
 \forall n\geqs 1\;.
\end{equation} 
% for all $n\geqs 1$. 
\end{corollary}
\begin{proof}
First note that the first relation in~\eqref{eq:Hermite_aadagger} is 
just a rewriting of~\eqref{eq:Hermite_recursive}. 
This allows us to prove~\eqref{eq:Hermite_eigenvalue} by induction on $n$. 
For $n = 0$, we clearly have $\cL H_0(x) = 0$, proving the base case. 
Assuming~\eqref{eq:Hermite_eigenvalue} holds for some $n\geqs 0$, 
\eqref{eq:a_commutator} yields
\begin{equation}
 -\cL H_{n+1} = a^\dagger a a^\dagger H_n 
 = a^\dagger (a^\dagger a + \id) H_n 
 = (n + 1) a^\dagger H_n 
 = (n+1) H_{n+1}\;.
\end{equation} 
Finally, we also have 
\begin{equation}
 aH_{n+1} = aa^\dagger H_n 
 = (a^\dagger a + \id) H_n 
 = (n+1) H_n\;,
\end{equation} 
which proves the second relation in~\eqref{eq:Hermite_aadagger}.
\end{proof}

Note that this result provides another proof of orthogonality of Hermite 
polynomials, since eigenfunctions of a self-adjoint operator 
corresponding to different eigenvalues are known to be orthogonal.
The second relation in~\eqref{eq:Hermite_aadagger} can also be written 
\begin{equation}
 H_n'(x) = n H_{n-1}(x)\;.
\end{equation} 
Together with the recursive relation~\eqref{eq:Hermite_recursive}, this yields 
\begin{equation}
\label{eq:Hn_recurrence2} 
 H_{n+1}(x) = x H_n(x) - n H_{n-1}(x)
 \qquad \forall n\geq1\;.
\end{equation} 
Yet another relation, following from the 
representation~\eqref{eq:aadagger_inner} of $a^\dagger$ is 
\begin{equation}
 H_n(x) = ((a^\dagger)^n H_0)(x)
 = (-1)^n \e^{x^2/2} \frac{\6^n}{\6x^n} (\e^{-x^2/2})\;.
\end{equation} 
The operator $\cL$ defined in~\eqref{eq:aadagger} occurs in several applications. 
In particular, define the \emph{Ornstein--Uhlenbeck process} as the solution of the 
stochastic differential equation 
\begin{equation}
 \6x_t = -x_t \6t + \sqrt{2}\6W_t\;,
\end{equation} 
where $(W_t)_{t\geqs0}$ is a standard Brownian motion, 
which can be written in terms of an It\^o integral as 
\begin{equation}
 x_t = x_0\e^{-t} + \sqrt{2} \int_0^t \e^{-(t-s)} \6W_s\;.
\end{equation} 
Then $\cL$ is the infinitesimal generator of the process, meaning that 
for any sufficiently regular test function $f$, one has 
\begin{equation}
 \frac{\6}{\6t} \expecin{x_0}{f(x_t)} \Bigr\vert_{t = 0} 
 = (\cL f)(x_0)\;.
\end{equation} 
This is a consequence of It\^o's formula (see als Section~\ref{ssec:OU} below). 

There also is a connection with quantum physics. Indeed, 
one finds that the conjugated operator 
\begin{equation}
 H = \e^{-x^2/4}\cL \e^{x^2/4}
\end{equation} 
has the expression 
\begin{equation}
 (Hf)(x) = \biggpar{\frac12 - \frac{x^2}{4}} f(x) + f''(x)\;, 
\end{equation} 
which is equivalent, up to a scaling, to the Hamiltonian of the 
quantum harmonic oscillator. The operator $H$ is self-adjoint in 
$L^2(\R,\6x)$, and its eigenfunctions are conjugated to the Hermite 
polynomials. As the operators $a^\dagger$ and $a$ allow to move 
between eigenfunctions, and these eigenfunctions are interpreted 
as $n$-particle states in quantum field theory, they as known as 
\emph{creation operator} and \emph{annihilation operator}. 

%%%%%%%%%%%%%%%%%%%%%%%%%%%%%%%%%%%%%%%%%%%%%%%%%%%%%%%%%%%%%%%%%%%%%%%%%%%%%%%%

\subsection{Convolution algebra*}
\label{ssec:convolution} 

Some of the above computations required to perform operations on 
power series, such as multiplication, division, and taking the logarithm. 
There exists an algebraic framework that makes these computations 
particularly easy.

Let $\R[x]$ denote the vector space of polynomials in one variable $x$, 
with the canonical basis $\set{x^n}_{n\geqs0}$. This is also an algebra 
for the usual product 
\begin{equation}
 x^n \cdot x^m = x^{n+m}\;.
\end{equation} 
Let $\R[[t]]$ be the space of formal power series, that is, expressions
of the form 
\begin{equation}
 \sum_{n\geqs0} \ph_n \frac{t^n}{n!}\;, 
 \qquad 
 \ph_n \in \R\;,
\end{equation} 
endowed with pointwise multiplication. 
By formal we mean that at this point, we are not concerned about 
convergence of the series. 

We can view the coefficients $\ph_n$ as the images of the $x^n$ by 
a linear map $\ph:\R[x]\to\R$. Denoting by $\cL(\R[x],\R)$ the 
space of all these maps, we can define a linear map 
\begin{align}
\label{eq:def_Lambda} 
 \Lambda \colon \cL(\R[x],\R) &\longrightarrow \R[[t]] \\
 \ph &\longmapsto \sum_{n\geqs0} \ph(x^n) \frac{t^n}{n!}\;.
\end{align} 
This map associates a power series with the map $\ph$ defining its coefficients. 
The interest of this construction is that due to the Cauchy product formula, 
multiplication of power series is equivalent to a convolution operation 
of maps. More precisely, for two maps $\ph,\psi \in \cL(\R[x],\R)$, 
define a map $\ph * \psi\in \cL(\R[x],\R)$ by 
\begin{equation}
\label{eq:convolution} 
 (\ph*\psi)(x^n) 
 = \sum_{k=0}^n \binom{n}{k} \ph(x^k)\psi(x^{n-k})\;.
\end{equation} 
We then have the following result.

\begin{theorem}[Isomorphism between convolution algebra and algebra of 
power series]
\label{thm:convolution} 
The map $\Lambda$ is an isomorphism between $\cL(\R[x],\R)$ and $\R[[t]]$.
\end{theorem}
\begin{proof}
By the Cauchy product formula, for any $\ph,\psi \in \cL(\R[x],\R)$, 
\begin{align}
\Lambda(\ph)(t)\Lambda(\psi)(t)
&= \biggpar{\sum_{n\geqs0} \ph(x^n) \frac{t^n}{n!}}
\biggpar{\sum_{m\geqs0} \psi(x^m) \frac{t^m}{m!}} \\
&= \sum_{p\geqs0} \biggpar{\sum_{k=0}^p 
\frac{\ph(x^k)}{k!} \frac{\psi(x^{p-k})}{(p-k)!}}t^p \\
&= \sum_{p\geqs0} (\ph*\psi)(x^p) \frac{t^p}{p!} \\
&= \Lambda(\ph*\psi)(t)\;.
\end{align}
This shows that $\Lambda$ is indeed an algebra morphism. Bijectivity follows 
from uniqueness of the coefficients of power series.
\end{proof}

This result allows us to work with convolution of linear maps instead 
of multiplication of power series. 
More generally, we define the $p$-fold convolution by 
\begin{equation}
 \ph^{\ast p}(x^n) 
 = \sum_{\substack{n_1, \dots, n_p\geqs0\\n_1+\dots+n_p=n}} 
 \frac{n!}{n_1!\dots n_p!}  \ph(x^{n_1}) \dots \ph(x^{n_p})\;.
\end{equation} 
It till turn out to be useful to work with linear instead of multilinear maps.
This is achieved by introducing a linear map $\Delta:\R[x] \to \R[x]\otimes\R[x]$,
given by 
\begin{equation}
\label{eq:def_Delta} 
 \Delta(x^n) = \sum_{k=0}^n \binom{n}{k} x^k\otimes x^{n-k}\;.
\end{equation} 
Here the tensor product $\R[x]\otimes\R[x]$ is the vector space spanned 
by all $x^n\otimes x^m$ with $n, m\geqs0$. Indeed~\eqref{eq:def_Delta} allows us 
to rewrite the convolution product~\eqref{eq:convolution} as 
\begin{equation}
 \ph \ast \psi = \cM(\ph\otimes\psi)\Delta\;,
\end{equation} 
where $\cM$ denotes the multiplication map, $\cM(a\otimes b) = ab$. 
More generally, for $p\geqs 3$ we set 
\begin{equation}
 \Delta^{(p-1)}(x^n) 
 = \sum_{\substack{n_1, \dots, n_p\geqs0\\n_1+\dots+n_p=n}} 
 \frac{n!}{n_1!\dots n_p!} x^{n_1}\otimes\dots\otimes x^{n_p}\;.
\end{equation} 
Writing $\cM_p(a_1\otimes\dots\otimes a_p) = a_1\dots a_p$,  
the $n$-fold convolution product becomes
\begin{equation}
 \ph^{\ast p}
 = \cM_p(\ph^{\otimes p})\Delta^{(p-1)}\;.
\end{equation}

\begin{remark}[Hopf algebra]
\label{rem:Hopf}
The space $\R[x]$ endowed with the maps $\cdot$, $\Delta$, and a 
\emph{counit} $\unit^\star:\R[x]\to\R$ defined by $\unit^\star(x^n) = \delta_{n0}$
is a so-called \emph{bi-algebra}. When adding a linear map $\cA$
defined by $\cA(x^n) = (-1)^n x^n$ and called \emph{antipode}, it 
becomes a \emph{Hopf algebra}. The map $\Delta$ is called a 
\emph{co-product}, because it enjoys a property called 
\emph{co-associativity}, saying that applying 
$\Delta$ to the left or to the right of the tensor product 
in $\Delta(x^n)$ yields the same result, namely 
\begin{equation}
 \Delta^{(2)}(x^n)
 = (\Delta\otimes\id)\Delta(x^n) 
 = (\id\otimes\Delta)\Delta(x^n) 
 = \sum_{\substack{n_1, n_2, n_3\geqs0\\n_1+n_2+n_3 = n}} \frac{n!}{n_1!n_2!n_3!}
 x^{n_1}\otimes x^{n_2}\otimes x^{n_3}
\end{equation}
and similarly for higher powers. 
\end{remark}

Let us now introduce two special subsets of $\cL(\R[x],\R)$, given by 
\begin{align}
 \cL_1 &= \bigsetsuch{\ph\in\cL(\R[x],\R)}{\ph(1) = 1}\;, \\
 \cL_0 &= \bigsetsuch{\ph\in\cL(\R[x],\R)}{\ph(1) = 0}\;. 
\end{align}
Elements of $\cL_1$ can be inverted, via the Neumann series 
\begin{equation}
 \ph^{-1} = \sum_{k=0}^\infty(\unit^\star - \ph)^{\ast k}\;.
\end{equation} 
One has explicitly 
\begin{equation}
\label{eq:phi_inverse} 
 \ph^{-1}(x^n) = \sum_{k=1}^n (-1)^k \sum_{\substack{n_1,\dots,n_k\geqs1\\n_1+\dots+n_k=n}}
 \frac{n!}{n_1!\dots n_k!} \ph(x^{n_1})\dots\ph(x^{n_k})\;.
\end{equation} 

\begin{exercise}
Check that one has indeed 
$(\ph \ast \ph^{-1})(x^n) = \delta_{n0}$ for all $n\geqs0$.

\noindent
\textbf{Hints:} Recall that $\ph(1) = 1$, and therefore $\ph^{-1}(1) = 1$. 
Check the cases $n \in \set{0,1,2}$ first.
For the general case, the term $k = 0$ in~\eqref{eq:convolution} 
should be treated separately. 
\end{exercise}

We can now define an exponential map $\exp_\ast:\cL_0\to\cL_1$ and its inverse $\log_\ast:\cL_1\to\cL_0$ by 
\begin{equation}
 \exp_\ast(\ph) = \sum_{k\geqs0} \frac{1}{k!} \ph^{\ast k}\;, \qquad 
 \log_\ast(\ph) = \sum_{k\geqs1} \frac{(-1)^k}{k} (\ph - \unit^*)^{\ast k}\;.
\end{equation} 
There is no issue of convergence, since the sums are always finite when evaluated on a basis element. In fact,
\begin{align}
\label{eq:expstar}
 \exp_\ast(\ph)(x^n) &= 
 \sum_{k=0}^n \frac{1}{k!}
 \sum_{\substack{n_1, \dots, n_k\geqs1\\n_1+\dots+n_k=n}}
 \frac{n!}{n_1!\dots n_k!}\ph(x^{n_1})\dots\ph(x^{n_k})\;, \\
 \log_\ast(\ph)(x^n) &= 
 \sum_{k=1}^n \frac{(-1)^{k+1}}{k}
 \sum_{\substack{n_1, \dots, n_k\geqs1\\n_1+\dots+n_k=n}}
 \frac{n!}{n_1!\dots n_k!}\ph(x^{n_1})\dots\ph(x^{n_k})\;.
\label{eq:logstar} 
\end{align} 
The following corollary of Theorem~\ref{thm:convolution} 
clarifies the link between these objects and the usual inverse, 
exponential and logarithm. 

\begin{proposition}
\label{prop:Lambda} 
For any $\ph\in\cL_0$ and $\psi\in\cL_1$, one has the relations 
\begin{align}
 \Lambda(\psi^{-1})(t) &= \bigbrak{\Lambda(\psi)(t)}^{-1}\;, \\
 \Lambda(\exp_\ast\ph)(t) &= \exp(\Lambda(\ph)(t))\;, \\
 \Lambda(\log_\ast\psi)(t) &= \log(\Lambda(\psi)(t))\;.
 \label{eq:Lambda_logstar} 
\end{align}
\end{proposition}
\begin{proof}
We prove the second relation. Setting $\psi = \exp_\ast\ph$, we have
\begin{equation}
 \Lambda(\psi)(t) = \sum_{k\geqs0} \frac{1}{k!} \Lambda(\ph^{\ast k})(t) 
 = \sum_{k\geqs0} \frac{1}{k!} \Lambda(\ph)(t)^k 
 = \exp(\Lambda(\ph)(t))\;.
\end{equation} 
The other relations are proved in a similar way. 
\end{proof}

Returning to the topic of Hermite polynomials, we consider the 
special case where $X$ is a real-valued random variable, admitting 
exponential moments, and associate with it the linear map 
$\mu_X:\R[x]\to\R$ given by 
\begin{equation}
 \mu_X(x^n) = \expec{X^n}\;.
\end{equation} 
Note that $\mu_X\in\cL_1$, since $\expec{1} = 1$. The associated power series 
\begin{equation}
\label{eq:moment_generating} 
 \Lambda(\mu_X)(t) = \sum_{n\geqs0} \frac{t^n}{n!} \expec{X^n}
 = \expec{\e^{tX}}
\end{equation} 
is the \emph{moment generating function} of $X$ introduced in~\eqref{eq:X_moment_gen}.
Proposition~\ref{prop:Lambda} implies that the cumulant generating function
(cf.~\eqref{eq:X_cumulant}) can be written as 
\begin{equation}
 K_X(t)  = \log\expec{\e^{tX}} 
 = \Lambda(\log_\ast \mu_X)(t) 
 = \Lambda(\kappa_X)(t)\;,
\end{equation} 
where
\begin{equation}
 \kappa_X = \log_\ast \mu_X\;, \qquad 
 \mu_X = \exp_\ast \kappa_X\;.
\end{equation} 
Applying~\eqref{eq:expstar} and~\eqref{eq:logstar} implies
\begin{align}
\label{eq:LS1} 
 \mu_X(x^n) &= 
 \sum_{k=0}^n \frac{1}{k!}
 \sum_{\substack{n_1, \dots, n_k\geqs1\\n_1+\dots+n_k=n}}
 \frac{n!}{n_1!\dots n_k!}\kappa_X(x^{n_1})\dots\kappa_X(x^{n_k})\;, \\
 \kappa_X(x^n) 
& = \sum_{k=1}^n \frac{(-1)^{k+1}}{k}
  \sum_{\substack{n_1,\dots,n_k\geqs1\\n_1+\dots+n_k=n}}
 \frac{n!}{n_1!\dots n_k!} \mu_X(x^{n_1})\dots\mu_X(x^{n_k})\;.
\label{eq:LS2} 
\end{align} 
These are the so-called \emph{Leonov--Shiraev moment-cumulant relations}. 
We now define the \emph{Wick exponential} associated to $X$ as the 
linear map $\W:\R[x]\to\R[x]$ given by 
\begin{equation}
\label{eq:Wick_expo} 
 \W = (\mu_X^{-1} \otimes \id) \Delta 
 = (\exp_\ast(-\kappa_X) \otimes \id) \Delta\;. 
\end{equation} 
One easily checks that $\W(1) = 1$, while for $n\geqs1$, 
\eqref{eq:phi_inverse} and~\eqref{eq:expstar} imply
\begin{equation}
\label{eq:W(Xn)} 
 \W(x^n) 
 = \sum_{k=0}^n \sum_{j=1}^k \frac{(-1)^j}{j!}
 \sum_{\substack{n_1,\dots,n_j\geqs1\\n_1+\dots+n_j=k}}
 \frac{n!}{(n-k)!n_1!\dots n_j!} \kappa_X(x^{n_1})\dots\kappa_X(x^{n_j})x^{n-k}\;.
\end{equation} 
Moreover, we have an explicit expression for the inverse of $\W$.

\begin{lemma}
The inverse of $\W$ is the map 
\begin{equation}
\label{eq:W_inverse} 
 \W^{-1} = (\mu_X\otimes\id)\Delta\;.
\end{equation} 
\end{lemma}
\begin{proof}
For a linear map $g: \R[x]\to\R$, write 
\begin{equation}
 M^g(x^n) = (g\otimes\id) \Delta(x^n)\;.
\end{equation} 
Note that for any linear form $f: \R[x]\to\R$, we have 
\begin{equation}
 \pscal{f}{M^g(x^n)} = (g\otimes f) \Delta(x^n) 
 =: \pscal{g\circ f}{x^n}\;,
\end{equation} 
where we use an \lq\lq inner product\rq\rq\ notation for the action of 
$f$ for clarity. 
It follows that if $h: \R[x]\to\R$ is yet another linear form, then 
\begin{equation}
 \pscal{f}{M^hM^g(x^n)}
 = \pscal{h\circ f}{M^g(x^n)} 
 = \pscal{g\circ h\circ f}{x^n} 
 = \pscal{f}{M^{g\circ h}(x^n)}\;.
\end{equation} 
In short, $M^hM^g = M^{g\circ h}$. Now we observe that 
$\W = M^{\mu_X^{-1}}$ and $\W^{-1} = M^{\mu_X}$. Therefore 
\begin{equation}
 \W\W^{-1}(x^n)
 = M^{\mu_X\circ\mu_X^{-1}}(x^n)
  = \bigpar{(\mu_X\circ\mu_X^{-1})\otimes\id}\Delta(x^n)\;,
\end{equation}
where 
\begin{align}
 (\mu_X\circ\mu_X^{-1})(x^k) 
 &= (\mu_X\otimes\mu_X^{-1})\Delta(x^k) \\
 &= \sum_{\ell=0}^k \binom{k}{\ell} \mu_X(x^\ell)\mu_X^{-1}(x^{k-\ell}) \\
 &= (\mu_X\ast\mu_X^{-1})(x^k) \\
 &= \delta_{k0}\;.
\end{align} 
This implies $(\mu_X\circ\mu_X^{-1}\otimes\id)\Delta(x^n) = x^n$
for all $n$, proving that $\W^{-1}$ is indeed the inverse of $\W$. 
\end{proof}

The following result shows that $\W(t,x) := \Lambda(\W)(t)$ is nothing but the 
generating function that we encountered in~\eqref{eq:generating_function}. 

\begin{proposition}
\label{prop:Wick_map} 
One has the relation
\begin{equation}
 \W(t,x)
 = \frac{\e^{tx}}{\expec{\e^{tX}}}
 = \e^{tx - K_X(t)}\;.
\end{equation} 
\end{proposition}
\begin{proof}
Observe that 
\begin{equation}
 \W(x^n)
 = \sum_{k=0}^n \binom{n}{k} \mu_X^{-1}(x^k) x^{n-k} 
 = (\mu_X^{-1} \ast \id)(x^n)\;.
\end{equation} 
Therefore, Theorem~\eqref{thm:convolution} implies 
\begin{equation}
 \Lambda(\W)(t) 
 = \Lambda(\mu_X^{-1} \ast \id)(t)
 = \Lambda(\mu_X^{-1})(t)\Lambda(\id)(t)
 = \frac{\Lambda(\id)(t)}{\Lambda(\mu_X)(t)}
 = \frac{\e^{tx}}{\expec{\e^{tX}}}\;,
\end{equation} 
where we have used Proposition~\ref{prop:Lambda} and~\eqref{eq:moment_generating}. 
\end{proof}

So far, the construction works for any random variable $X$ with exponential 
moments. Let us now particularise to the case $X \sim \cN(0,1)$. Since 
$\kappa_X(x^n) = \delta_{n2}$, cf.~\eqref{eq:cumulant_gaussian}, all $n_i$ 
in the first Leonov--Shiraev relation~\eqref{eq:LS1} must have value $2$. 
This is only possible if $n$ is even and $k = n/2$. Therefore, we obtain 
\begin{equation}
\label{eq:expec_X2k} 
 \expec{X^{2k}} 
 = \mu_X(x^{2k})
 = \frac{(2k)!}{k!2^k}\;.
\end{equation} 
This is equal to $(k-1)!!$, as one checks by separating even and odd factors 
in $(2k)!$. We thus recover Proposition~\ref{prop:gaussian_moments}. 
In addition, we obtain the following explicit expressions for 
Hermite polynomials, as well as the inverse relations between monomials 
and Hermite polynomials.

\begin{proposition}[Explicit expression of Hermite polynomials]
\label{prop:Hermite_explicit} 
For any $n\in\N_0$, one has 
\begin{equation}
\label{eq:Hn_xn} 
 H_n(x) = n!\sum_{k=0}^{\lfloor n/2 \rfloor}
 \frac{(-1)^k}{2^kk!(n-2k)!}x^{n-2k}\;.
\end{equation} 
The inverse relation is given by 
\begin{equation}
\label{eq:xn_Hn} 
 x^n = n!\sum_{k=0}^{\lfloor n/2 \rfloor}
 \frac{1}{2^kk!(n-2k)!}H_{n-2k}(x)\;.
\end{equation} 
\end{proposition}
\begin{proof}
Since 
\begin{equation}
 \W(t,x) = \Lambda(\W)(t) = \sum_{n\geqs0} \W(x^n) \frac{t^n}{n!}
\end{equation} 
by the definition~\eqref{eq:def_Lambda} of $\Lambda$, we have $H_n(x) 
= \W(x^n)$. We can thus apply~\eqref{eq:W(Xn)} when all $n_i$ are 
equal to $2$. This is only possible if $2j = k$, and the result 
follows by the index shift $k \mapsto k/2$. 
The inverse relation~\eqref{eq:xn_Hn} follows in the same way 
from the expression~\eqref{eq:W_inverse} for $\W^{-1}$. 
\end{proof}

%%%%%%%%%%%%%%%%%%%%%%%%%%%%%%%%%%%%%%%%%%%%%%%%%%%%%%%%%%%%%%%%%%%%%%%%%%%%%%%%

\subsection{Hermite polynomials and combinatorics}
\label{ssec:combinatorics} 

The coefficients of Hermite polynomials also have a nice combinatorial 
interpretation. Given a finite set $E_n$ of cardinal $n$, say 
$E_n = \intset{1,n} := \set{1,2,\dots,n}$, we will call \emph{pairwise matching}, 
or \emph{pairing}, a partition of $E_n$ into sets of cardinality $1$ or $2$, 
the former being called \emph{singletons} and 
the latter being called \emph{pairs}.

\begin{proposition}[Combinatorial interpretation of Hermite polynomials]
Let $0\leqs 2k\leqs n$. 
The coefficient of $x^{n-2k}$ of $H_n(x)$ is equal to the number of 
pairwise matchings of $E_n$ with $k$ pairs. 
\end{proposition}
\begin{proof}
Given a pairwise matching of $E_n$ with $k$ pairs, we associate to each 
singleton the label $x$, and to each pair the label $-1$. The value 
of this matching is defined as the product of all labels, namely  
\begin{equation}
 (-1)^{k} x^{n-2k}\;,
\end{equation} 
see Figure~\ref{fig:pairings} for an example. 
We claim that $H_n(x)$ is equal to the sum of the values of all pairwise 
matchings of $E_n$. 

We proceed by induction, taking $n = 1$ as base case. Then the only 
pairwise matching is $\set{\set{1}}$, which has value $x = H_1(x)$. 

Assume now that $n \geqs 1$. The pairwise matchings of $E_{n+1}$ are of two 
types. The first type are those in which $\set{n+1}$ is a singleton. 
These have value $xH_n(x)$ by induction hypothesis. The second type of 
pairings are those in which $n+1$ belongs to a pair. There are $n$ choices 
for the partner of $n+1$, and for each of these choices, the value of 
the matching of the remaining elements is $H_{n-1}(x)$. This shows that 
\begin{equation}
 H_{n+1}(x) = xH_n(x) - n H_{n-1}(x)\;,
\end{equation} 
which is exactly the recurrence relation~\eqref{eq:Hn_recurrence2}. 
\end{proof}

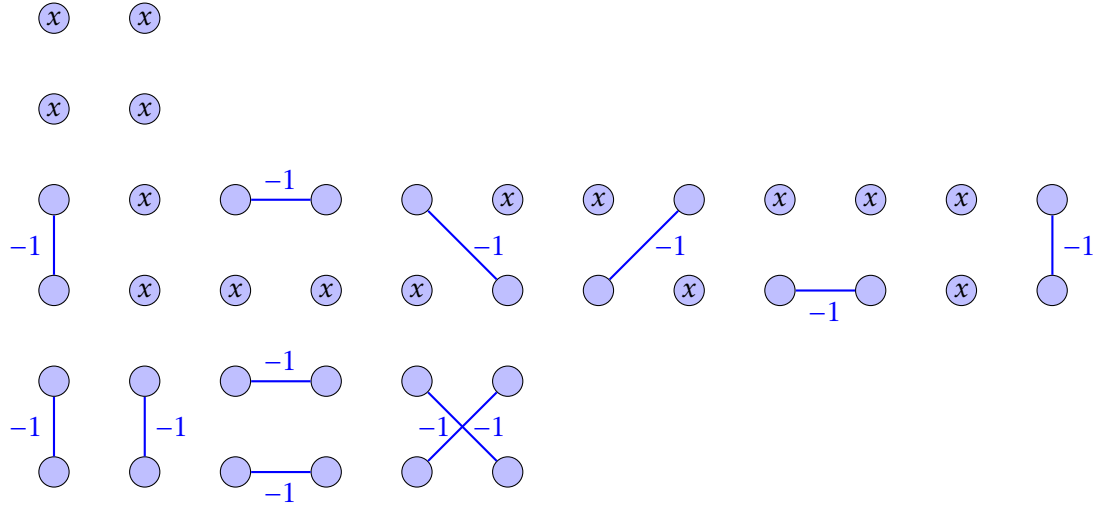
\begin{figure}
\begin{center}
 \begin{tikzpicture}[>=stealth',main node/.style={circle,minimum
size=0.4cm,inner sep=1pt,fill=blue!25,draw},scale=1.2]
\node[main node] (1) at (0,0) {$x$};
\node[main node] (2) at (0,-1) {$x$};
\node[main node] (3) at (1,0) {$x$};
\node[main node] (4) at (1,-1) {$x$};

\node[main node] (1) at (0,-2) {};
\node[main node] (2) at (0,-3) {};
\node[main node] (3) at (1,-2) {$x$};
\node[main node] (4) at (1,-3) {$x$};
\draw[thick,blue] (1) --node[left]{$-1$} (2);  

\node[main node] (1) at (2,-2) {};
\node[main node,] (2) at (2,-3) {$x$};
\node[main node] (3) at (3,-2) {};
\node[main node] (4) at (3,-3) {$x$};
\draw[thick,blue] (1) --node[above]{$-1$} (3);  

\node[main node] (1) at (4,-2) {};
\node[main node,] (2) at (4,-3) {$x$};
\node[main node] (3) at (5,-2) {$x$};
\node[main node] (4) at (5,-3) {};
\draw[thick,blue] (1) --node[right]{$-1$} (4);  

\node[main node] (1) at (6,-2) {$x$};
\node[main node,] (2) at (6,-3) {};
\node[main node] (3) at (7,-2) {};
\node[main node] (4) at (7,-3) {$x$};
\draw[thick,blue] (2) --node[right]{$-1$} (3);  

\node[main node] (1) at (8,-2) {$x$};
\node[main node,] (2) at (8,-3) {};
\node[main node] (3) at (9,-2) {$x$};
\node[main node] (4) at (9,-3) {};
\draw[thick,blue] (2) --node[below]{$-1$} (4);  

\node[main node] (1) at (10,-2) {$x$};
\node[main node,] (2) at (10,-3) {$x$};
\node[main node] (3) at (11,-2) {};
\node[main node] (4) at (11,-3) {};
\draw[thick,blue] (3) --node[right]{$-1$} (4);  

\node[main node] (1) at (0,-4) {};
\node[main node] (2) at (0,-5) {};
\node[main node] (3) at (1,-4) {};
\node[main node] (4) at (1,-5) {};
\draw[thick,blue] (1) --node[left]{$-1$} (2);  
\draw[thick,blue] (3) --node[right]{$-1$} (4);  

\node[main node] (1) at (2,-4) {};
\node[main node,] (2) at (2,-5) {};
\node[main node] (3) at (3,-4) {};
\node[main node] (4) at (3,-5) {};
\draw[thick,blue] (1) --node[above]{$-1$} (3);  
\draw[thick,blue] (2) --node[below]{$-1$} (4);  

\node[main node] (1) at (4,-4) {};
\node[main node,] (2) at (4,-5) {};
\node[main node] (3) at (5,-4) {};
\node[main node] (4) at (5,-5) {};
\draw[thick,blue] (1) --node[right]{$-1$} (4);  
\draw[thick,blue] (2) --node[left]{$-1$} (3);  
\end{tikzpicture}
\end{center}
\vspace{-4mm}
\caption[]{Pairwise matchings of $E_4 = \set{1,2,3,4}$.
There is one matching with no pairs, of value $x^4$.
There are $6$ matchings with one pair, of total value $-6x^2$, 
and $3$ matchings with two pairs, of total value $3$.
The sum of all values is $x^4 - 6x^2 +3 = H_4(x)$.}
\label{fig:pairings} 
\end{figure}

The number of pairwise matchings can also be computed exactly, 
which yields the following alternative proof of the first relation in
Proposition~\ref{prop:Hermite_explicit}.

\begin{proposition}[Explicit expression of Hermite polynomials]
For any $n\in\N_0$, one has 
\begin{equation}
 H_n(x) = n!\sum_{k=0}^{\lfloor n/2 \rfloor}
 \frac{(-1)^k}{2^kk!(n-2k)!}x^{n-2k}\;.
\end{equation} 
\end{proposition}
\begin{proof}
It is sufficient to count the number of matchings with $k$ pairs. 
There are $\binom{n}{n-2k}$ choices for the $n-2k$ singletons. 
The number of pairings of the remaining $2k$ elements is 
\begin{equation}
\label{eq:double_factorial} 
 (2k-1)!! = \frac{(2k)!}{k!2^k}\;.
\end{equation} 
This is because the first element has $2k-1$ possible partners. Having 
chosen the first part, it remains to pair $2k-2$ elements, so that the 
claim follows by induction. Multiplying the two numbers gives the 
claimed coefficient of $x^{n-2k}$.
\end{proof}

Pairings appear at other places in computations with Hermite polynomials.
For instance, Proposition~\ref{prop:Hermite_product_sum} yields 
\begin{equation}
 H_4(x)^2 
 = H_8(x) + 4^2 H_6(x) + 2! 6^2 H_4(x) + 3!4^2 H_2(x) + 4! H_0(x)\;.
\end{equation} 
The right-hand side can be represented graphically by 
\begin{equation}
 \FGfour \FGfour 
 + 4^2 \FGIlegVI
 + 2!6^2 \FGIIlegIV
 + 3!4^2 \FGIIIlegII 
 + 4! \FGIV \;, 
\end{equation} 
where the numerical coefficients count the number of ways to pair $2p$ legs 
from the two different diagrams having $4$ legs each.

%%%%%%%%%%%%%%%%%%%%%%%%%%%%%%%%%%%%%%%%%%%%%%%%%%%%%%%%%%%%%%%%%%%%%%%%%%%%%%%%

\section{Wiener chaos expansion}
\label{sec:Wiener_chaos} 

%%%%%%%%%%%%%%%%%%%%%%%%%%%%%%%%%%%%%%%%%%%%%%%%%%%%%%%%%%%%%%%%%%%%%%%%%%%%%%%%

Recall that we have introduced the Hilbert space 
\begin{equation}
 \cH = L^2(\R, \mu(\6x))\;, 
\end{equation} 
which consists in random variables of the form $f(X)$ which have a finite 
variance, with $X\sim\cN(0,1)$. We begin with a preparatory lemma, 
which is a special case of~\cite[Lemma 1.1.1]{nualart2006malliavin}.

\begin{lemma}
\label{lem:Wiener_chaos} 
The random variables $\setsuch{\e^{tX}}{t\in\R}$ form a total subset of $\cH$.  
\end{lemma}
\begin{proof}
Let $Z\in\cH$ be such that $\expec{Z\e^{tX}} = 0$ for all $t\in\R$. We have 
to show that $Z = 0$. Define a signed measure $\nu$ by 
\begin{equation}
 \nu(B) = \bigexpec{Z\indicator{B}(X)}
\end{equation} 
for any Borel set $B$ of $\R$. The fact that $\expec{Z\e^{tX}} = 0$
means that the Laplace transform of $\nu$  
is identically zero on $\R$. Therefore, this measure 
is zero, so that $\expec{Z\indicator{B}} = 0$ for any Borel set $B$. 
This implies that $Z$ is indeed equal to $0$. 
\end{proof}

\begin{definition}[Wiener chaos]
For any $n\geqs 1$, we denote by $\cH_n$ the one-dimensional subspace of $\cH$ 
spanned by the random variable $H_n(X)$. For $n = 0$, $\cH_0$ is the set
of constants, which is isomorphic to $\R$. Then $\cH_n$ is called the 
\emph{homogeneous Wiener chaos of order $n$}. 
The \emph{inhomogeneous Wiener chaos of order $n$} is defined as 
\begin{equation}
 \cH_{\leqs n} = \bigoplus_{k=0}^n \cH_k\;.
\end{equation} 
\end{definition}

Note that Proposition~\ref{prop:Hermite_orthogonal} implies that the subspaces 
$\cH_n$ are mutually orthogonal. 

The main theorem is then as follows (this is a particular case of~\cite[Theorem 1.1.1]{nualart2006malliavin}:

\begin{theorem}[Wiener chaos decomposition]
\label{thm:Wiener_chaos} 
The Hilbert space $\cH$ can be decomposed into the infinite orthogonal sum 
of the subspaces $\cH_n$:
\begin{equation}
 \cH = \bigoplus_{n=0}^\infty \cH_n\;.
\end{equation} 
\end{theorem}
\begin{proof}
Let $Z\in\cH$ be orthogonal to $\cH_n$ for all $n\geqs0$. This means that 
\begin{equation}
 \expec{Z H_n(X)} = 0
 \qquad \text{for all $n\geqs0$\;.}
\end{equation} 
By~\eqref{eq:xn_Hn}, $X^n$ can be expressed as a linear 
combination of $H_k(X)$, with $0\leqs k\leqs n$. Therefore, we also have 
\begin{equation}
 \expec{Z X^n} = 0
 \qquad \text{for all $n\geqs0$\;.}
\end{equation} 
This implies that $\expec{Z\e^{tX}} = 0$ for all $t\in\R$. By 
Lemma~\ref{lem:Wiener_chaos}, this means that $Z = 0$, which completes 
the proof. 
\end{proof}

In a sense, this theorem is quite remarkable, because $f$ being in 
$L^2(\R,\mu(\6x))$ is a rather weak requirement. In particular, $f$ 
needs not be continuous, nor does it need to be bounded. We started this
chapter by remarking that $\expec{f(X)}$ can be computed if $f$ admits 
a power series expansion, with strictly positive radius of convergence, 
see~\eqref{eq:expec_series}. Theorem~\ref{thm:Wiener_chaos} shows that 
this assumption of $f$ is not at all necessary.

\begin{remark}
If $f\in\cH$, we have 
\begin{equation}
\label{eq:f_Wiener_decomp} 
 f(X) = \sum_{n=0}^\infty \expec{f(X)H_n(X)} H_n(X)\;.
\end{equation} 
Taking expectations shows that 
\begin{equation}
 \expec{f(X)} = \expec{f(X)H_0(X)}\;,
\end{equation} 
which does not yield any new information. However, in applications $f(X)$ often 
admits a more or less explicit decomposition of the 
form~\eqref{eq:f_Wiener_decomp}, which makes it possible to compute
expectations. 
\end{remark}

\begin{exercise}
Consider the function 
\begin{equation}
 f(X) = \e^{-\lambda H_4(X)}
\end{equation} 
where $\lambda\in[0,\infty)$. Our aim is to obtain an asymptotic expansion 
\begin{equation}
 \expec{f(X)} \asymp \sum_{n\geqs 0} a_n \lambda^n\;,
\end{equation} 
which means that 
\begin{equation}
 \expec{f(X)} = \sum_{n = 0}^N a_n \lambda^n + \Order{\lambda^{N+1}}
\end{equation} 
for any $N\geqs1$. 
\begin{itemize}
\item   Compute explicitly $a_0$, $a_1$, $a_2$ and $a_3$.
\item   Give a combinatorial interpretation of $a_n$ for any $n$
(in terms of certain types of graphs).
\item   Find an upper bound $b_n$ for $\abs{a_n}$. What is the 
radius of convergence of the series $\sum_n b_n \lambda^n$?
\end{itemize}
\end{exercise}

\begin{exercise}
Consider the function 
\begin{equation}
 f(X) = \e^{-\lambda H_2(X)}
\end{equation} 
where $\lambda\in[0,\infty)$. 
\begin{itemize}
\item   Compute $\expec{f(X)}$ explicitly.
\item   What is the radius of convergence of the expansion of $\expec{f(X)}$
into powers of $\lambda$?
\item   Give a combinatorial interpretation of the $n$th coefficient 
of this series. Check this for the first values of $n$. 
\end{itemize}
\end{exercise}

%%%%%%%%%%%%%%%%%%%%%%%%%%%%%%%%%%%%%%%%%%%%%%%%%%%%%%%%%%%%%%%%%%%%%%%%%%%%%%%%

\chapter{The multi-dimensional case}
\label{chap:nd} 

In this chapter, we extend results from the previous chapter to $N$-dimensional
Gaussian random variables with $N$ finite. 

%%%%%%%%%%%%%%%%%%%%%%%%%%%%%%%%%%%%%%%%%%%%%%%%%%%%%%%%%%%%%%%%%%%%%%%%%%%%%%%%

\section{Wick calculus}
\label{sec:Wick_calculus}

%%%%%%%%%%%%%%%%%%%%%%%%%%%%%%%%%%%%%%%%%%%%%%%%%%%%%%%%%%%%%%%%%%%%%%%%%%%%%%%%

\subsection{Multivariate Gaussian random variables}
\label{ssec:Gaussian_multivariate} 

We start by a quick recapitulation of basic properties of $N$-dimensional
Gaussian random variables. 

\begin{definition}[Multivariate Gaussian]
For $N\geqs1$, let $\R^N$ be equipped with the $\sigma$-algebra $\cB$ of Borel
sets and Lebesgue measure $\6x$. Let $m\in\R^N$ and let $C\in\R^{N\times N}$ be 
a symmetric, positive definite matrix. 
A random variable $X:\R^N\to\R$ is a (multivariate) Gaussian random variable 
with mean $m$ and covariance matrix $C$ if its law is 
\begin{equation}
\label{eq:Gaussian_ndim} 
 \mu(\6x) = 
 \frac{1}{(2\pi)^{N/2}\det(C)^{1/2}} \e^{-\pscal{(x-m)}{C^{-1}(x-m)}/2} \6x\;.
\end{equation} 
In that case, we write $X \sim \cN(m, C)$. 
\end{definition}

The following result generalises~\eqref{eq:exp_moment_N}
to the multivariate case. 

\begin{proposition}[Laplace transform]
\label{prop:Laplace_Gaussian_ndim} 
Let $C$ be symmetric, positive definite. 
Then $X \sim \cN(0,C)$ if, and only if, for any $t\in\R^n$, one has 
\begin{equation}
\label{eq:Laplace_Gaussian_ndim} 
 \expec{\e^{\pscal{t}{X}}} = \e^{\pscal{t}{Ct}/2}\;.
\end{equation} 
\end{proposition}
\begin{proof}
We proceed by diagonalisation. 
Since $C$ is symmetric, positive definite, there exists an orthogonal 
matrix $U$ such that $U^\top C U = \Lambda$ is diagonal, with real, strictly 
positive diagonal elements $\lambda_1, \dots, \lambda_N$. Since $\det(\Lambda) 
= \det(C)\det(UU^\top) = \det(C)$, the fact that $X \sim \cN(0,C)$ is equivalent, by 
the change of variables formula (or transfer theorem), to $Y$ having law 
\begin{equation}
 \hat\mu(\6y) 
 = \frac{1}{(2\pi)^{N/2}\det(\Lambda)^{1/2}} \e^{-\pscal{U^\top y}{C^{-1}U^\top Y}/2} \6y
 = \frac{1}{(2\pi)^{N/2}\det(\Lambda)^{1/2}} \e^{-\pscal{y}{\Lambda^{-1}y}/2} \6y
 = \prod_{i=1}^N \frac{\e^{-\lambda_iy_i^2/2}}{\sqrt{2\pi\lambda_i}}\;.
\end{equation} 
This means that the components $Y_i$ of $Y$ are independent, with 
$Y_i\sim\cN(0,\lambda_i)$. By~\eqref{eq:exp_moment_N}, this implies that 
for any $s\in\R^N$, 
\begin{equation}
 \expec{\e^{\pscal{s}{Y}}} 
 = \prod_{i=1}^N \expec{\e^{\pscal{s_i}{Y_i}}}
 = \prod_{i=1}^N \e^{\lambda_i s_i^2/2}
 = \expec{\e^{\pscal{s}{\Lambda s}}}\;.
\end{equation} 
The converse is actually true, because $\expec{\e^{\pscal{s_i}{Y_i}}} = 1$
if $s_i = 0$. Setting $s = tU^\top$, this is equivalent to 
\begin{equation}
 \expec{\e^{\pscal{t}{X}}}
 = \expec{\e^{\pscal{t}{UY}}}
 = \expec{\e^{\pscal{U^\top t}{Y}}}
 = \e^{\pscal{U^\top t}{\Lambda U^\top t}/2}
 = \e^{\pscal{t}{U\Lambda U^\top t}/2}
 = \e^{\pscal{t}{Ct}/2}\;,
\end{equation} 
which completes the proof. 
\end{proof}

One consequence of this result is that one can easily justify the name 
\lq\lq covariance matrix\rq\rq\ (though there are other ways to verify this).

\begin{corollary}
If  $X \sim \cN(0,C)$, then $\expec{X_iX_j} = C_{ij}$ for all 
$i,j\in\intset{1,N}$. 
\end{corollary}
\begin{proof}
This follows by taking the derivative of~\eqref{eq:Laplace_Gaussian_ndim} 
with respect to $t_i$ and $t_j$, and evaluating at $t = 0$. 
\end{proof}

%%%%%%%%%%%%%%%%%%%%%%%%%%%%%%%%%%%%%%%%%%%%%%%%%%%%%%%%%%%%%%%%%%%%%%%%%%%%%%%%

\subsection{Isserlis' theorem}
\label{ssec:Isserlis} 

Isserlis' theorem (also known as Wick's theorem in physics) is a generalisation 
of the expression~\eqref{eq:moment_Gaussian} for the moments of a Gaussian 
random variable to the multivariate case. We first show a preparatory lemma,
which is a simple instance of the Schwinger--Dyson equations in 
quantum field theory. 

\begin{lemma}[Integration by parts]
\label{lem:ibp} 
Assume $X \sim \cN(0,C)$. 
For any $i\in\intset{1,N}$, the equality 
\begin{equation}
\label{eq:ibp} 
 \bigexpec{X_i f(X)} = \sum_{j=1}^N C_{ij} \bigexpec{\partial_j f(X)}
\end{equation} 
holds for all differentiable $f:\R^N\to\R$ such that both sides of the equality are 
well-defined. 
\end{lemma}
\begin{proof}
Leibniz' rule yields 
\begin{equation}
 \dpar{}{x_j} \Bigpar{f(x)\e^{-\pscal{x}{C^{-1}x}/2}}
 = \biggbrak{\dpar{f}{x_j}(x)
 -\frac12 f(x)\dpar{}{x_j} \pscal{x}{C^{-1}x} }
 \e^{-\pscal{x}{C^{-1}x}/2} \;.
\end{equation} 
Since we have 
\begin{equation}
 \frac12 \sum_{j=1}^N C_{ij} \dpar{}{x_j} \pscal{x}{C^{-1}x} 
 = \sum_{j,k=1}^N C_{ij}C^{-1}_{jk} x_k 
 = \sum_{k=1}^N \delta_{ik} x_k = x_i\;,
\end{equation} 
it follows that 
\begin{equation}
 x_i f(x)\e^{-\pscal{x}{C^{-1}x}/2} 
 = \sum_{j=1}^N C_{ij} 
 \biggbrak{-\dpar{}{x_j} \Bigpar{f(x)\e^{-\pscal{x}{C^{-1}x}/2}} + 
 \dpar{f}{x_j} \e^{-\pscal{x}{C^{-1}x}/2}}\;.  
\end{equation} 
Integrating over the whole space, we see that the boundary terms vanish, proving~\eqref{eq:ibp}.
\end{proof}

An immediate consequence of this result is the following theorem, due to 
Leon Isserlis. 

\begin{theorem}[Isserlis]
\label{thm:Isserlis} 
 For any $1\leqs k\leqs \frac{N}2$, we have 
 \begin{align}
  \bigexpec{X_1\dots X_{2k}} &= 
  \sum_{\cP}
  \;\prod_{\set{i,j}\in\cP} \bigexpec{X_iX_j}\;, \\
  \bigexpec{X_1\dots X_{2k-1}} &= 0\;,
 \end{align} 
where the sum runs over all \emph{perfect matchings} $\cP$ 
of $\intset{1,2k}$, meaning that each $\cP$ is a partition 
$\cP=\set{\set{i_1,j_1},\dots,\set{i_k,j_k}}$ of this set into disjoint subsets 
of two elements.
\end{theorem}
\begin{proof}
The proof proceeds by induction on the number of factors, 
applying~\eqref{eq:ibp} with $i=1$ and $f(X)$ of the form $X_2\dots X_m$. 
\end{proof}

\begin{example}
In the case $2k=4$, we obtain 
\begin{equation}
 \bigexpec{X_1X_2X_3X_4} 
 = \bigexpec{X_1X_2}\bigexpec{X_3X_4} + \bigexpec{X_1X_3}\bigexpec{X_2X_4} + 
\bigexpec{X_1X_4}\bigexpec{X_2X_3}\;.
\label{eq:Isserlis_X1X2X3X4} 
\end{equation} 
A convenient graphical way of representing this relation is the following:
\begin{equation}
 \bigexpec{X_1X_2X_3X_4} 
 =
\raisebox{-22pt}{\tikz{
 \path[use as bounding box] (-0.25, -1.25) rectangle (1.25,0.25);
 \node[mynode,label=above:{\scriptsize 1}] (1) at (0,0) {};
 \node[mynode,label=below:{\scriptsize 2}] (2) at (0,-0.75) {};
 \node[mynode,label=above:{\scriptsize 3}] (3) at (0.75,0) {};
 \node[mynode,label=below:{\scriptsize 4}] (4) at (0.75,-0.75) {};
 \draw[thick] (1) -- (2);  
 \draw[thick] (3) -- (4);  
 \drawbox;
 }}
 +
\raisebox{-22pt}{\tikz{
 \path[use as bounding box] (-0.25, -1.25) rectangle (1.25,0.25);
 \node[mynode,label=above:{\scriptsize 1}] (1) at (0,0) {};
 \node[mynode,label=below:{\scriptsize 2}] (2) at (0,-0.75) {};
 \node[mynode,label=above:{\scriptsize 3}] (3) at (0.75,0) {};
 \node[mynode,label=below:{\scriptsize 4}] (4) at (0.75,-0.75) {};
 \draw[thick] (1) -- (3);  
 \draw[thick] (2) -- (4);  
 \drawbox;
 }}
+
\raisebox{-22pt}{\tikz{
 \path[use as bounding box] (-0.25, -1.25) rectangle (1.25,0.25);
 \node[mynode,label=above:{\scriptsize 1}] (1) at (0,0) {};
 \node[mynode,label=below:{\scriptsize 2}] (2) at (0,-0.75) {};
 \node[mynode,label=above:{\scriptsize 3}] (3) at (0.75,0) {};
 \node[mynode,label=below:{\scriptsize 4}] (4) at (0.75,-0.75) {};
 \draw[thick] (1) -- (4);  
 \draw[thick] (3) -- (2);  
 \drawbox;
 }}
\end{equation}
\end{example}

\begin{remark}
Theorem~\ref{thm:Isserlis} remains true if the covariance matrix $C$ 
is only semi-definite positive. We may thus allow for the case 
$X = (X_1,X_1,\dots,X_1)$, in which case it yields 
\begin{equation}
 \expec{X_1^{2k}} = (2k-1)!!
\end{equation} 
since this is the number of perfect matchings of $\intset{1,2k}$.
We thus recover Proposition~\ref{prop:gaussian_moments}. 
\end{remark}

\begin{exercise}
Provide an alternative proof of Isserlis' theorem, using the expression~\eqref{eq:Laplace_Gaussian_ndim} of the Laplace transform 
of $X = (X_1,\dots,X_n)$. 
\end{exercise}

%%%%%%%%%%%%%%%%%%%%%%%%%%%%%%%%%%%%%%%%%%%%%%%%%%%%%%%%%%%%%%%%%%%%%%%%%%%%%%%%

\section{Hermite polynomials for multivariate Gaussians}
\label{sec:Hermite_ndim} 

In this section, we derive some additional properties of Hermite polynomials, 
when they are evaluated in linear combinations of Gaussian random variables.

%%%%%%%%%%%%%%%%%%%%%%%%%%%%%%%%%%%%%%%%%%%%%%%%%%%%%%%%%%%%%%%%%%%%%%%%%%%%%%%%

\subsection{Scaling}
\label{ssec:Hermite_scaling} 

In Chapter~\ref{chap:1d}, we defined Hermite polynomials for centred 
Gaussian random variables of variance $1$. While centering is never 
a problem, it will often be useful to allow for general variances. 
This can be achieved by a simple scaling.

\begin{definition}[Scaled Hermite polynomials]
The \emph{Hermite polynomial of degree $n$ with variance $\sigma^2$} 
is defined as 
\begin{equation}
\label{eq:Hermite_scaling} 
 H_n(x;\sigma^2) = \sigma^n H_n(x/\sigma)\;.
\end{equation} 
\end{definition}

Table~\ref{tab:Hermite_scaled} shows the expressions of the first six
scaled Hermite polynomials. 

\begin{table}
\begin{center}
\begin{tabular}{|r|l|}
\hline 
$n$ & $H_n(x;\sigma^2)$ \\
\hline 
$0$ & $1$ \\
$1$ & $x$ \\
$2$ & $x^2 - \sigma^2$ \\
$3$ & $x^3 - 3\sigma^2x$ \\
$4$ & $x^4 - 6\sigma^2x^2 + 3\sigma^4$ \\
$5$ & $x^5 - 10\sigma^2x^3 + 15\sigma^4x$ \\
\hline
\end{tabular}
\end{center}
\caption[]{List of the first Hermite polynomials of variance $\sigma^2$.}
\label{tab:Hermite_scaled} 
\end{table}

One easily checks the following properties:
\begin{itemize}
\item   scaled Hermite polynomials admit the generating function 
\begin{equation}
 G(t,x;\sigma^2) = \e^{tx - \sigma^2t^2/2}\;;
\end{equation} 
\item   scaled Hermite polynomials satisfy the same 
orthogonality~\eqref{eq:Hermite_orthogonal} as unscaled ones;
\item   scaled Hermite polynomials satisfy the recursive relations 
\begin{align}
 H_{n+1}(x;\sigma^2) &= x H_n(x;\sigma^2) - \sigma^2 \partial_x H_n(x;\sigma^2)\;, \\
 \partial_x H_n(x;\sigma^2) &= n H_{n-1}(x;\sigma^2)\;;
\end{align}
\item   scaled Hermite polynomials admit the explicit expression
\begin{equation}
\label{eq:scaled_Wick} 
 H_n(x;\sigma) = n!\sum_{k=0}^{\lfloor n/2 \rfloor}
 \frac{(-1)^k}{2^kk!(n-2k)!}\sigma^{2k}x^{n-2k}\;.
\end{equation} 
\end{itemize}

%%%%%%%%%%%%%%%%%%%%%%%%%%%%%%%%%%%%%%%%%%%%%%%%%%%%%%%%%%%%%%%%%%%%%%%%%%%%%%%%

\subsection{Binomial formula}
\label{ssec:Hermite_binomial} 

The following binomial formula allows to relate Hermite polynomials with 
different variances. 

\begin{lemma}[Binomial formula for Hermite polynomials]
\label{lem:Hermite_binomial} 
For any $x, y\in\R$, any $\sigma_1, \sigma_2 > 0$ 
and any $n\in\N_0$, one has 
\begin{equation}
 H_n(x + y;\sigma_1^2 + \sigma_2^2)
 = \sum_{m=0}^n \binom{n}{m} 
 H_m(x;\sigma_1^2) H_{n-m}(y;\sigma_2^2)\;.
\end{equation} 
\end{lemma}
\begin{proof}
Expanding the identity
\begin{equation}
 \e^{t(x+y)-(\sigma_1^2+\sigma_2^2)t^2/2}
  = \e^{tx-\sigma_1^2t^2/2}\e^{ty-\sigma_2^2t^2/2}
\end{equation} 
yields 
\begin{equation}
 \sum_{n=0}^\infty \frac{t^n}{n!} H_n(x+y;\sigma_1^2+\sigma_2^2)
= \sum_{m=0}^\infty \frac{t^m}{m!} H_m(x;\sigma_1^2)
 \sum_{k=0}^\infty \frac{t^k}{k!} H_k(y;\sigma_2^2)\;. 
\end{equation} 
Comparing coefficients of $t^n$ gives the result. 
\end{proof}

A very useful generalisation of this property, spelled out here for 
Hermite polynomials of unit variance, is the following result.

\begin{lemma}[Multinomial formula for Hermite polynomials]
\label{lem:hermite_multinomial}
Let $a\in\ell^2$ be a sequence of real numbers such that 
$\sum_{i\geqs0}a_i^2=1$. Then for any sequence $(x_i)_{i\geqs0}$ such that 
$\sum_{i\geqs0}a_ix_i$ converges, one has 
\begin{equation}
\label{eq:multinomial_Hermite} 
 H_n\biggpar{\sum_{i\geqs0}a_ix_i} 
 = \sum_{\abs{k} = n} \frac{n!}{k!} a^k \prod_{i\geqs0} H_{k_i}(x_i)\;,
\end{equation} 
where the sum runs over all $k\in\N_0^{\N_0}$ such that $\abs{k} = 
\sum_{i\geqs0} k_i = n$, 
and 
\begin{equation}
 k! := \prod_{i\geqs0} k_i!\;, \quad 
 a^k := \prod_{i\geqs0} a_i^{k_i}\;.
 \label{eq:fact_multiindex} 
\end{equation} 
\end{lemma}
\begin{proof}
First note that the condition $\abs{k} = n$ implies that $k$ has only finitely 
many nonzero indices, so that all products are well-defined. Consider now $x, 
y, a, b\in\R$ such that $a^2 + b^2 = 1$. Then Lemma~\ref{lem:Hermite_binomial} 
and the scaling property~\eqref{eq:Hermite_scaling} imply 
\begin{align}
 H_n(ax + by) 
 &= H_n(ax + by; a^2 + b^2) \\
 &= \sum_{m=0}^n \binom{n}{m} H_m(ax; a^2) H_{n-m}(by; b^2) \\
 &= \sum_{m=0}^n \binom{n}{m} a^n b^{n-m} H_m(x) H_{n-m}(y)\;.
\end{align} 
The result then follows by applying this identity repeatedly. 
\end{proof}

\begin{exercise}
Use these results to compute $H_2\Bigpar{\frac{x+y}{\sqrt{2}}}$ and 
$H_4\Bigpar{\frac{x+y}{\sqrt{2}}}$ in terms of $H_n(x)$ and $H_n(y)$. 
\end{exercise}

%%%%%%%%%%%%%%%%%%%%%%%%%%%%%%%%%%%%%%%%%%%%%%%%%%%%%%%%%%%%%%%%%%%%%%%%%%%%%%%%

\section{Wiener chaos expansion}
\label{sec:Wiener_ndim} 

%%%%%%%%%%%%%%%%%%%%%%%%%%%%%%%%%%%%%%%%%%%%%%%%%%%%%%%%%%%%%%%%%%%%%%%%%%%%%%%%

\subsection{Main result}
\label{ssec:Wiener_ndim_main} 

Consider now the case where $X_1,\dots,X_N$ are independent Gaussian random 
variables, of zero expectation and unit variance, defined on a common probability 
space $(\Omega,\cF,\fP)$. We are interested in the Hilbert space 
\begin{equation}
 \cH = L^2(\Omega,\cF,\fP)
\end{equation} 
of random variables $F = f(X_1,\dots,X_N)$ admitting a finite variance. 

In order to deal with linear combinations of the $X_i$, it turns out to be 
useful to introduce the Hilbert space $\H = \R^N$ endowed with the canonical 
inner product $\pscal{\cdot}{\cdot}_\H$. Define a linear map $W:\H\to\cH$ by  
\begin{equation}
\label{eq:def_Wh} 
 W(h) = \sum_{i=1}^N h_i X_i\;.
\end{equation} 
Then we have, for any $h, g\in \H$,  
\begin{equation}
\label{eq:Wh_covariance} 
 \expec{W(h)W(g)} 
 = \sum_{i=1}^N \sum_{j=1}^N h_i g_j \expec{X_i X_j}
 = \sum_{i=1}^N h_i g_j 
 = \pscal{h}{g}_\H\;.
\end{equation} 
This means that the map $W$ is an isometry from $\H$ onto $\cH$. 

\begin{definition}[Wiener chaos, $N$-dimensional version]
For any $n\geqs 1$, we denote by $\cH_n$ the subspace of $\cH$ 
spanned by the random variables 
\begin{equation}
 \setsuch{H_n(W(h))}{h\in \H, \norm{h}_\H = 1}\;.
\end{equation} 
For $n = 0$, $\cH_0$ is the set
of constants, which is isomorphic to $\R$. Then $\cH_n$ is called the 
\emph{homogeneous Wiener chaos of order $n$}. 
The \emph{inhomogeneous Wiener chaos of order $n$} is defined as 
\begin{equation}
 \cH_{\leqs n} = \bigoplus_{k=0}^n \cH_k\;.
\end{equation} 
\end{definition}

We can now extend Lemma~\ref{lem:Wiener_chaos} and Theorem~\ref{thm:Wiener_chaos}
to the situation at hand. These are again particular cases of~\cite[Lemma~1.1.2 and Theorem~1.1.1]{nualart2006malliavin}.

\begin{lemma}
\label{lem:Wiener_chaos_N} 
The random variables $\setsuch{\e^{W(h)}}{h\in \H}$ form a total subset of $\cH$.  
\end{lemma}
\begin{proof}
Let $Z\in\cH$ be such that $\expec{Z\e^{W(h)}} = 0$ for all $h\in \H$. 
The linearity of the map $h\mapsto W(h)$ implies that 
\begin{equation}
\label{eq:lem_Wiener01} 
 \biggexpec{Z\exp\biggset{\sum_{i=1}^m t_i W(h_i)}} = 0
\end{equation} 
for any choice of $t_1,\dots,t_m\in\R$, any $h_1,\dots,h_m\in \H$ and $m\geqs1$. 
We now fix $m$ and $h_1,\dots,h_m$, and define a signed measure $\nu$ by 
\begin{equation}
 \nu(B) = \bigexpec{Z\indicator{B}(W(h_1),\dots,W(h_m))}
\end{equation} 
for any Borel set $B$ of $\R^N$. Then~\eqref{eq:lem_Wiener01}
means that the Laplace transform of $\nu$  
is identically zero on $\R$. Therefore, this measure 
is zero, so that $\expec{Z\indicator{B}} = 0$ for any Borel set $B$. 
This implies that $Z$ is equal to $0$, and therefore the completeness 
if the system. 
\end{proof}

\begin{theorem}[Wiener chaos decomposition]
\label{thm:Wiener_chaos_N} 
The Hilbert space $\cH$ can be decomposed into the infinite orthogonal sum 
\begin{equation}
 \cH = \bigoplus_{n=0}^\infty \cH_n\;.
\end{equation} 
\end{theorem}
\begin{proof}
Let $Z\in\cH$ be orthogonal to $\cH_n$ for all $n\geqs0$. This means that 
\begin{equation}
 \expec{Z H_n(W(h))} = 0
\end{equation} 
for all $n\geqs0$ and all $h\in \H$ with $\norm{h}_\H = 1$.
By~\eqref{eq:xn_Hn}, we also have $\expec{Z W(h)^n} = 0$ for all $n\geqs0$,
and therefore $\expec{Z \e^{tW(h)}} = 0$ for all $t\in\R$ and all 
$h\in \H$ such that $\norm{h}_\H = 1$. By Lemma~\ref{lem:Wiener_chaos_N}, 
this means that $Z = 0$, which completes the proof. 
\end{proof}

%%%%%%%%%%%%%%%%%%%%%%%%%%%%%%%%%%%%%%%%%%%%%%%%%%%%%%%%%%%%%%%%%%%%%%%%%%%%%%%%

\subsection{Wiener isometry and Fock space}
\label{ssec:Wiener_Fock} 

We now construct an orthogonal basis of each Wiener chaos $\cH_n$.
For $k\in\N_0^N$, we define 
\begin{equation}
\label{eq:def_Phik} 
 \Phi_k = \prod_{i=1}^N H_{k_i}(X_i)\;.
\end{equation} 
Then Proposition~\ref{prop:Hermite_orthogonal} (orthogonality of the $H_n$) 
implies that 
\begin{equation}
 \expec{\Phi_k\Phi_\ell} 
 = \prod_{i=1}^N \expec{H_{k_i}(X_i)H_{\ell_i}(X_i)}
 = \prod_{i=1}^N k_i! \delta_{k_i\ell_i}
 = k! \delta_{k\ell}
\end{equation} 
with $k!$ as in~\eqref{eq:fact_multiindex}. 

We denote by $\H^{\otimes n}$ the $n$-fold tensor product of $\H$, and 
by $\Hsym{n}$ the subspace of symmetric tensors. Any element of 
$\smash{\H^{\otimes n}}$ can be canonically projected on $\Hsym{n}$ via
\begin{equation}
\label{eq:def_Pi} 
 \Pi\bigpar{h_1\otimes \dots \otimes h_n}
 := \frac{1}{n!} \sum_{\sigma\in\frakS_n} h_{\sigma(1)} \otimes \dots \otimes 
h_{\sigma(n)}\;,
\end{equation} 
where $\frakS_n$ denotes the set of all permutations of $\set{1,\dots,n}$. 
Let $(e_1,\dots,e_N)$ denote any orthonormal basis of $\H$ (for instance the 
canonical basis). Then for any $k\in\N_0^N$, 
\begin{equation}
\label{eq:ek_Wiener} 
 e_k := \Pi\bigotimes_{i=1}^N e_i^{\otimes k_i}
\end{equation} 
is an element of $\Hsym{\abs{k}}$, where $\abs{k} = \sum_{i=1}^N \abs{k_i}$, 
with the convention that only strictly positive $k_i$ count in the product~\eqref{eq:ek_Wiener}.
Moreover, the set $\setsuch{e_k}{\abs{k}=n}$ forms an orthogonal basis of $\Hsym{n}$, 
with 
\begin{equation}
 \pscal{e_k}{e_\ell} = \frac{k!}{n!} \delta_{k\ell}\;.
\end{equation} 
This is because among the $n!$ permutations defining $e_k$, there are 
$k!$ permutations that yield the same term. 

\begin{example}
Assume $N = 3$, and let $k = (2,1,0)$. Then $k! = 2$, $n = \abs{k} = 3$, and 
\begin{align}
 e_k &= \Pi(e_1 \otimes e_1 \otimes e_2) \\
 &= \frac13 \bigpar {e_1 \otimes e_1 \otimes e_2 + e_1 \otimes e_2 \otimes e_1 
 + e_2 \otimes e_1 \otimes e_1}\;,
\end{align}
because the $3! = 6$ permutations of $\set{1,2,3}$ have pairwise the same image. 
It follows that 
\begin{equation}
 \pscal{e_k}{e_k} = \frac13 = \frac{k!}{n!}\;.
\end{equation} 
\end{example}

\begin{definition}[Wiener isometry]
For $k\in\N_0^N$, let $n = \abs{k}$. The map 
\begin{equation}
\label{eq:def_In} 
 I_n : e_k \longmapsto \frac{1}{\sqrt{n!}}\Phi_k
\end{equation} 
is called the $n$th \emph{Wiener isometry}. 
\end{definition}

\begin{exercise}
\begin{enumerate}
\item   Show that $I_n$ is indeed an isometry between $\Hsym{n}$ and $\cH_n$. 

\item   If $N = 3$, what are the dimensions of $\cH_1$, $\cH_2$ and $\cH_3$? 
Provide orthogonal bases of these spaces. 

\item   What is the dimension of $\cH_n$ for general $N$? 

\textbf{Hint:} Use the method of stars and bars.
\end{enumerate}
\end{exercise}

% \begin{definition}[Fock space]
The space 
\begin{equation}
 \cF = \bigoplus_{n=0}^\infty \Hsym{n}
\end{equation} 
is called \emph{Fock space}. 
% \end{definition}
The Wiener isometry thus provides an isometry between Fock space and $\cH$. 
Fock space is known from quantum physics, where it describes interacting 
bosons. The space $\Hsym{n}$ describes the set of states with given number 
$n$ of particles. It is identified with the space of symmetric functions 
of $n$ variables. 

Two important particular cases of Wiener isometries are 
\begin{equation}
 I_0 = 1\;, \qquad 
 I_1(h) = W(h)\;,
\end{equation} 
where the second relation follows from $I_1(e_i) = X_i$, see~\eqref{eq:def_Wh}. 
The following lemma generalises these relations. 

\begin{lemma}
\label{lem:Wiener_tensor} 
For any $n\geqs1$ and $h\in\H$ with $\norm{h}_\H=1$, one has 
\begin{equation}
\label{eq:Wiener_tensor} 
 I_n(h^{\otimes n}) 
 = \frac{1}{\sqrt{n!}}H_n\bigpar{W(h)}\;,
\end{equation} 
independently of the basis $\set{e_i}_{1\leqs i\leqs N}$ of $\H$. 
\end{lemma}
\begin{proof}
If $h = \sum_{i=1}^N h_i e_i$, then 
\begin{equation}
 h^{\otimes n} 
 = \Pi(h^{\otimes n})
 = \sum_{1\leqs i_1, \dots, i_n\leqs N} h_{i_1}\dots h_{i_n} 
\Pi(e_{i_1}\otimes\dots\otimes e_{i_n})
 = \sum_{\abs{k} = n} \frac{n!}{k!} h^k e_k\;,
\end{equation} 
where the combinatorial factor counts the number of ways a tuple 
$(i_1,\dots,i_n)$ can be mapped to a given $k$, defined by the fact that 
$k_j$ is the number of indices equal to $j$. The result then follows from the 
definitions~\eqref{eq:def_In} of $I_n$ and~\eqref{eq:def_Phik} of $\Phi_k$ and 
the multinomial formula~\eqref{eq:multinomial_Hermite} (see 
Lemma~\ref{lem:hermite_multinomial}). 
\end{proof}

\begin{remark}
\label{rem:Wiener_isometry_scaled} 
If $h$ is not normalised, one has 
\begin{equation}
 \label{eq:Wiener_isometry_scaled}
 I_n(h^{\otimes n}) 
 = \frac{\norm{h}_\H^n}{\sqrt{n!}} H_n\biggpar{W\biggpar{\frac{h}{\norm{h}_\H}}}
 = \frac{1}{\sqrt{n!}} H_n\bigpar{W(h); \norm{h}_\H^2} 
\end{equation} 
by the scaling property~\eqref{eq:Hermite_scaling} of Hermite polynomials. 
\end{remark}

In general, we can compute $I_n$ on an element of $\H^n$ by projecting it 
on $\Hsym{n}$, decomposing it in the basis of $e_k$, and applying the 
definition~\eqref{eq:def_In} of $I_n$.

\begin{exercise}
Assume $h = \sum_{i=1}^N h_ie_i$ and $g = \sum_{i=1}^N g_ie_i$
belong to $\H$. Show that 
\begin{equation}
\label{eq:I2hg} 
 I_2(h\otimes g) = 
 \frac{1}{\sqrt{2!}} 
 \biggbrak{\sum_{i\neq j=1}^N h_ig_j H_1(X_i) H_1(X_j) 
 + \sum_{i=1}^N h_ig_i H_2(X_i)}\;.
\end{equation} 
Check that one recovers~\eqref{eq:Wiener_tensor} when $h = g$ and 
$\norm{h}_\H = \norm{g}_\H = 1$.
\end{exercise}

%%%%%%%%%%%%%%%%%%%%%%%%%%%%%%%%%%%%%%%%%%%%%%%%%%%%%%%%%%%%%%%%%%%%%%%%%%%%%%%%

\subsection{Multiplication and Wick product}
\label{ssec:Wick_product}

Now that we have defined $I_n(f)$ for general $f\in\Hsym{n}$ 
(or $\H^{\otimes n}$, using symmetrization), it is of interest to compute 
products of such quantities, in the same spirit as for the product-sum 
formulas seen in Proposition~\ref{prop:Hermite_product_sum}. 
Here it will be more convenient to use the normalisation 
\begin{equation}
 \hat I_n(f) = \sqrt{n!} I_n(f)\;.
\end{equation} 
In order not to confuse components of tensor products in $\H^n$ and components 
of elements of $\H$, we will write the latter as 
\begin{equation}
 f = \sum_{i=1}^N f(i) e_i \in \H\;.
\end{equation} 
In this way, elements of $\H^{\otimes n}$ can be written as 
\begin{equation}
 f_1 \otimes \dots \otimes f_n 
 = \sum_{i_1,\dots,i_n=1}^N f_1(i_1)\dots f_n(i_n) 
 e_{i_1}\otimes\dots\otimes e_{i_n}\
 =: \sum_{i_1,\dots,i_n=1}^N f(i_1,\dots,i_n) 
 e_{i_1}\otimes\dots\otimes e_{i_n}\;.
\end{equation} 
We can thus view $f$ as a map from $\intset{1,N}$ to $\R$. 

Let us start by computing $\hat I_1(f) \hat I_1(g)$. 
This is given by 
\begin{equation}
 \hat I_1(f) \hat I_1(g) 
 = W(f) W(g) 
 = \sum_{i,j=1}^N f(i)g(j) X_i X_j\;.
\end{equation} 
Comparing with~\eqref{eq:I2hg}, we see that 
\begin{equation}
 \hat I_1(f) \hat I_1(g) - \hat I_2(f\otimes g) 
 = \sum_{i=1}^N f(i)g(i)\bigbrak{X_i^2-H_2(X_i)}
 = \pscal{f}{g}_\H \in\cH_0\;.
\end{equation} 
We have thus obtained 
\begin{equation}
\label{eq:mult_11} 
 \hat I_1(f) \hat I_1(g) = \hat I_2(f\otimes g) 
 + \pscal{f}{g}_\H\;.
\end{equation} 
% A first generalisation of this result is as follows, 
Let $\frakS(p,n)\subset\frakS(n)$ denote the set of permutations
of $\intset{1,n}$ preserving the order of the first $p$ and the last $n-p$ elements, also called \emph{shuffles}. A first generalisation of~\eqref{eq:mult_11} 
is as follows. 

\begin{lemma}[Multiplication between the $n$th and first chaos]
\label{lem:multn1} 
Assume $f\in\H^{\otimes n}$ and $g\in\H$. Then 
\begin{equation}
 \hat I_n(f)\hat I_1(g) 
 = \hat I_{n+1}(f\otimes g) + \hat I_{n-1}(f\star_1g)\;,
\end{equation} 
where $\star_1$ denotes the \emph{contraction operation} 
\begin{equation}
\label{eq:def_star1} 
 (f\star_1g)(i_1,\dots,i_{n-1}) 
 = \sum_{\Sigma\in\frakS(1,n)}
 \sum_{j=1}^N f(\Sigma(j,i_1,\dots,i_{n-1}))g(j)\;. 
\end{equation} 
\end{lemma}
\begin{proof}
By linearity, it suffices to check the relation for basis vectors, that is, 
when one has $f = e_{i_1}\otimes\dots\otimes e_{i_n}$ and $g = e_j$. We 
distinguish between two cases.
\begin{enumerate}
\item   If $j\not\in\set{i_1,\dots,i_n}$, then $f\star_1g = 0$, because 
all components of $f$ containing one index $j$ are zero.
One easily sees that 
\begin{equation}
 \hat I_{n+1}(f\otimes e_j) = \hat I_n(f)X_j\;.
\end{equation} 
\item   If $j\in\set{i_1,\dots,i_n}$, say $j=i_1$, 
let $m$ be the number of indices equal to $i_1$. We have 
\begin{equation}
 \hat I_n(f) = H_m(X_{i_1}) P_{n-m}\;,
 \qquad 
 \hat I_1(g) = H_1(X_{i_1})\;,
\end{equation} 
where $P_{n-m}$ is a polynomial of degree $n-m$ that does not contain $X_{i_1}$. 
By the product-sum formula~\eqref{eq:H_product_sum1}, we obtain 
\begin{equation}
 H_m(X_{i_1}) H_1(X_{i_1}) 
 = \sum_{p=0}^1 \binom{m}{p}\binom{1}{p} H_{m+1-p}(X_{i_1})
 = H_{m+1}(X_{i_1}) + mH_{m-1}(X_{i_1})\;,
\end{equation} 
so that 
\begin{equation}
\label{eq:proof_mult01} 
 \hat I_n(f) \hat I_1(g)
 = \bigbrak{H_{m+1}(X_{i_1}) + mH_{m-1}(X_{i_1})} P_{n-m}\;.
\end{equation} 
On the other hand, we have
\begin{equation}
\label{eq:proof_mult02} 
 \hat I_{n+1}(f\otimes g)
 = H_{m+1}(X_i) P_{n-m}\;,
\end{equation} 
while, using $g(\ell) = \delta_{\ell j}$,  
\begin{equation}
 (f\star_1g)(i_1,\dots,i_{n-1}) 
 = \sum_{\Sigma\in\frakS(1,m)} 
 f(\Sigma(i_1,i_1,\dots,i_{n-1}))\;.
\end{equation} 
There are $m$ ways to insert the first $i_1$ among the other indices,
which all yield the same value upon applying $\hat I_{n-1}$. Therefore,
\begin{equation}
\label{eq:proof_mult03} 
 \hat I_{n-1} (f\star_1g) 
 = m H_{m-1}(X_{i_1}) P_{n-m}\;.
\end{equation} 
The sum of~\eqref{eq:proof_mult02} and~\eqref{eq:proof_mult03} is indeed 
equal to~\eqref{eq:proof_mult01}. 
\qed
\end{enumerate}
\renewcommand{\qed}{}
\end{proof}

\begin{example}
Consider the case $n=3$. Then 
\begin{equation}
 \hat I_3(f) \hat I_1(g) 
 = \hat I_4(f\otimes g) + \hat I_2(f\star_1 g)\;,
\end{equation} 
where
\begin{equation}
 (f\otimes g)(i_1,i_2,i_3,j) 
 = f(i_1,i_2,i_3) g(j)\;.
\end{equation} 
Denoting a permutation $\Sigma\in\frakS(1,3)$ by the image 
$(\Sigma(1),\Sigma(2),\Sigma(3))$ of $(1,2,3)$, we have 
\begin{equation}
\frakS(1,3) = \set{(1,2,3), (2,1,3), (2,3,1)} 
\end{equation} 
and therefore 
\begin{equation}
 (f\star_1 g)(i_1,i_2) 
 = \sum_{j=1}^N \bigbrak{f(j,i_1,i_2) + f(i_1,j,i_2) + f(i_1,i_2,j)}g(j)\;.
\end{equation} 
\end{example}

The generalisation of this to arbitrary $f$ and $g$ is as follows.

\begin{proposition}[Multiplication between $n$th and $m$th chaos]
\label{prop:mult_nm} 
Assume $f\in\H^{\otimes n}$ and $g\in\H^{\otimes m}$. Then 
\begin{equation}
\label{eq:chaos_multn} 
 \hat I_n(f)\hat I_m(g) 
 = \sum_{p=0}^{n\wedge m} \hat I_{n+m-2p}(f\star_p g)
\end{equation} 
where $\star_0 = \otimes$ and 
the contraction $\star_p$ is defined 
for $\mathbf{i} = (i_1,\dots,i_{n-p})$ 
and $\mathbf{j} = (j_1,\dots,j_{m-p})$ 
by  
\begin{equation}
\label{eq:def_starp} 
 (f\star_pg)(\mathbf{i},\mathbf{j}) 
 = \sum_{\substack{\Sigma\in\frakS(p,n)\\ \bar\Sigma\in\frakS(p,m)}}
 \sum_{\sigma\in\frakS(p)}
 \sum_{\mathbf{k}\in\intset{1,N}^p} f(\Sigma(\mathbf{k},\mathbf{i}))g(\bar\Sigma(\mathbf{k},\sigma(\mathbf{j})))\;. 
\end{equation} 
\end{proposition}

Before giving a proof of this result, we provide some intuition for the 
meaning of contractions, see Figure~\ref{fig:pairings_n}. 
We think of $f$ as a vertex with $n$ legs, representing its 
components, while $g$ is a vertex with $m$ legs. 
The contraction $\star_p$ represents all ways of 
pairing $p$ legs of $f$ with $p$ legs of $g$, where $\Sigma$ 
represents the choice of legs of $f$, $\bar\Sigma$ represents the 
choice of legs of $g$, and $\sigma$ counts all ways of pairing 
the chosen legs. As for the sum over $\ell$, it can be viewed as 
an inner product in $\H$. With this picture in mind, we can show 
the following lemma.

\begin{lemma}
\label{lem:pairing} 
Assume $f\in\H^{\otimes n}$, and $g = g_1\otimes g_2\in\H^{\otimes m}$
with $g_1\in\H^{\otimes(m-1)}$ and $g_2\in\H$. Then 
one has 
\begin{equation}
 f \star_p g 
 = \indicator{p\neq0} (f\star_{p-1} g_1)\star_1 g_2 
 + \indicator{p\neq m} (f\star_p g_1)\otimes g_2
 \qquad 
 \text{for $0\leqs p\leqs m$\;.} 
\end{equation} 
\end{lemma}
\begin{proof}
We give a graphical proof. 
The first term on the right-hand side represents pairing $p-1$ 
legs of $f$ with $p-1$ legs of $g_1$, and one leg of $f$ with one 
leg of $g_2$, see Figure~\ref{fig:pairings_n} (a). 
This is only possible if $p\geqs1$. The second term 
on the right-hand side represents pairing $p$ legs of $f$ with $p$
legs of $g_1$, and leaving $g_2$ alone, see 
Figure~\ref{fig:pairings_n} (b). This is only possible 
if $p \leqs m-1$. 
\end{proof}

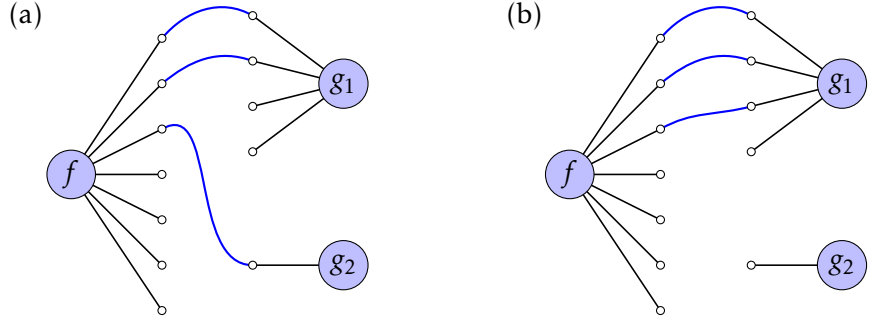
\begin{figure}
\begin{center}
\begin{tikzpicture}[>=stealth',main node/.style={circle,minimum
size=0.4cm,inner sep=2pt,fill=blue!25,draw},small node/.style={draw,circle,fill=white,minimum
size=3pt,inner sep=0pt},scale=1.2]
\draw[semithick] (0,0) -- (1,1.5);
\draw[semithick] (0,0) -- (1,1.0);
\draw[semithick] (0,0) -- (1,0.5);
\draw[semithick] (0,0) -- (1,0.0);
\draw[semithick] (0,0) -- (1,-0.5);
\draw[semithick] (0,0) -- (1,-1.0);
\draw[semithick] (0,0) -- (1,-1.5);

\draw[semithick] (3,1) -- (2,1.75);
\draw[semithick] (3,1) -- (2,1.25);
\draw[semithick] (3,1) -- (2,0.75);
\draw[semithick] (3,1) -- (2,0.25);

\draw[semithick] (3,-1) -- (2,-1);

\draw[thick,blue] (1,1.5) to[out=50,in=150] (2,1.75);
\draw[thick,blue] (1,1.0) to[out=40,in=160] (2,1.25);
\draw[thick,blue] (1,0.5) to[out=30,in=180] (2,-1);

\node[main node] (f) at (0,0) {$f$};

\node[main node] (g1) at (3,1) {$g_1$};
\node[main node] (g2) at (3,-1) {$g_2$};

\node[small node] (1) at (1,1.5) {};
\node[small node] (2) at (1,1.0) {};
\node[small node] (3) at (1,0.5) {};
\node[small node] (4) at (1,0.0) {};
\node[small node] (5) at (1,-0.5) {};
\node[small node] (6) at (1,-1.0) {};
\node[small node] (7) at (1,-1.5) {};

\node[small node] (8) at (2,1.75) {};
\node[small node] (9) at (2,1.25) {};
\node[small node] (10) at (2,0.75) {};
\node[small node] (11) at (2,0.25) {};

\node[small node] (12) at (2,-1) {};

\node[] at (-0.5,1.75) {(a)};
\end{tikzpicture}
\hspace{15mm}
\begin{tikzpicture}[>=stealth',main node/.style={circle,minimum
size=0.4cm,inner sep=2pt,fill=blue!25,draw},small node/.style={draw,circle,fill=white,minimum
size=3pt,inner sep=0pt},scale=1.2]
\draw[semithick] (0,0) -- (1,1.5);
\draw[semithick] (0,0) -- (1,1.0);
\draw[semithick] (0,0) -- (1,0.5);
\draw[semithick] (0,0) -- (1,0.0);
\draw[semithick] (0,0) -- (1,-0.5);
\draw[semithick] (0,0) -- (1,-1.0);
\draw[semithick] (0,0) -- (1,-1.5);

\draw[semithick] (3,1) -- (2,1.75);
\draw[semithick] (3,1) -- (2,1.25);
\draw[semithick] (3,1) -- (2,0.75);
\draw[semithick] (3,1) -- (2,0.25);

\draw[semithick] (3,-1) -- (2,-1);

\draw[thick,blue] (1,1.5) to[out=50,in=150] (2,1.75);
\draw[thick,blue] (1,1.0) to[out=40,in=160] (2,1.25);
\draw[thick,blue] (1,0.5) to[out=30,in=195] (2,0.75);

\node[main node] (f) at (0,0) {$f$};

\node[main node] (g1) at (3,1) {$g_1$};
\node[main node] (g2) at (3,-1) {$g_2$};

\node[small node] (1) at (1,1.5) {};
\node[small node] (2) at (1,1.0) {};
\node[small node] (3) at (1,0.5) {};
\node[small node] (4) at (1,0.0) {};
\node[small node] (5) at (1,-0.5) {};
\node[small node] (6) at (1,-1.0) {};
\node[small node] (7) at (1,-1.5) {};

\node[small node] (8) at (2,1.75) {};
\node[small node] (9) at (2,1.25) {};
\node[small node] (10) at (2,0.75) {};
\node[small node] (11) at (2,0.25) {};

\node[small node] (12) at (2,-1) {};

\node[] at (-0.5,1.75) {(b)};
\end{tikzpicture}
\end{center}
\vspace{-4mm}
\caption[]{Examples of pairings corresponding to 
$(f\star_{p-1} g_1)\star_1 g_2$ (a)
and to $(f\star_p g_1)\otimes g_2$ (b), when $n=7$, $m=5$ 
and $p=3$.}
\label{fig:pairings_n} 
\end{figure}

We can now give the proof of Proposition~\ref{prop:mult_nm}.

\begin{proof}[{\sc Proof of Proposition~\ref{prop:mult_nm}}]
We may assume $n\geqs m$. 
The proof is by induction on $m$. The base case $m=1$ is Lemma~\ref{lem:multn1}.
For the induction step, we can restrict by linearity to the case 
where $g = g_1\otimes g_2$ with $g_1\in\H^{\otimes(m-1)}$ and $g_2\in\H$. 
Assume that $g_2$ is orthogonal to all components of $g_1$. Then we claim that 
\begin{equation}
 g_1 \star_1 g_2 = 0\;.
\end{equation} 
Indeed, taking $g_2$ as a basis vector of $\H$, $g_1$ will be a linear 
combination of basis vectors different from $g_2$, so that the 
inner product in~\eqref{eq:def_star1} vanishes. It follows that 
\begin{equation}
 \hat I_{m-1}(g_1) \hat I_1(g_2) = \hat I_m(g_1\otimes g_2)\;.
\end{equation} 
Therefore, 
\begin{align}
 \hat I_n(f) \hat I_m(g)
 &= \hat I_n(f) \hat I_{m-1}(g_1) \hat I_1(g_2) \\
 &= \sum_{p=0}^{m-1} \hat I_{n+m-1-2p}(f\star_pg_1) \hat I_1(g_2) \\
 &= \sum_{p=0}^{m-1} 
 \hat I_{n+m-2p}\bigpar{(f\star_pg_1)\otimes g_2}
 + \sum_{p=0}^{m-1} \hat I_{n+m-1-2p}\bigpar{(f\star_pg_1)\star_1 g_2}\;,
\end{align}
where we have used the induction hypothesis in the second line, 
and Lemma~\ref{lem:multn1} in the third one. Making the index 
shift $p\mapsto p+1$ in the second sum and using Lemma~\ref{lem:pairing}
yields the result. 

In case it is not possible to find an orthogonal decomposition 
$g = g_1\otimes g_2$, by linearity it suffices to consider 
the case where $g = g_2^{\otimes m}$. This case can be reduced 
to an explicit computation, see Exercise~\ref{ex:multn} below. 
\end{proof}

\begin{remark}
If $f$  and $g$ are symmetric under permutations of their arguments, 
since $\frakS(p,n)$ has cardinality $\binom{n}{p}$, we have 
\begin{equation}
 (f\star_pg)(\mathbf{i},\mathbf{j}) 
 = p!\binom{n}{p} \binom{m}{p} 
 \sum_{\mathbf{k}\in\intset{1,N}^p} 
 f(\mathbf{k},\mathbf{i})g(\mathbf{k},\mathbf{j})\;. 
\end{equation} 
\end{remark}

\begin{exercise}
\label{ex:multn} 
Compute $\hat I_n(f)\hat I_m(g)$ when $f = h_1^{\otimes n}$ 
and $g = h_2^{\otimes m}$ with $h_1, h_2\in\H$. 
Argue that for the case left out in the proof of Proposition~\ref{prop:mult_nm}, 
is suffices to consider the case $h_1 = h_2$. Why is the result true 
in that case?
\end{exercise}

The leading term in the decomposition~\eqref{eq:chaos_multn} 
plays a special role, and is therefore given a name.

\begin{definition}[Wick product]
The \emph{Wick product} of $f\in\H^{\otimes n}$ and $g\in\H^{\otimes m}$
is defined by 
\begin{equation}
 \hat I_n(f) \diamond \hat I_m(g) = \hat I_{n+m}(f\otimes g)\;.
\end{equation} 
\end{definition}

Relation~\eqref{eq:chaos_multn} can thus be written
\begin{equation}
 \hat I_n(f)\hat I_m(g) 
 = \hat I_n(f) \diamond \hat I_m(g)
 + \sum_{p=1}^{n\wedge m} \hat I_{n+m-2p}(f\star_p g)\;.
\end{equation} 

%%%%%%%%%%%%%%%%%%%%%%%%%%%%%%%%%%%%%%%%%%%%%%%%%%%%%%%%%%%%%%%%%%%%%%%%%%%%%%%%

\section{Equivalence of moments}
\label{sec:equivalence_moments} 

The purpose of this section is to give a proof of the following, very 
useful result.

\begin{theorem}[Equivalence of moments]
\label{thm:equivalence_moments} 
Assume $F$ belongs to the $n$th Wiener chaos $\cH_n$. Then for any 
$p > 1$, one has 
\begin{equation}
\label{eq:equivalence_moments} 
 \expec{F^{2p}}^{1/2p} \leqs (2p-1)^{n/2} \expec{F^2}^{1/2}\;.
\end{equation} 
\end{theorem}

This result states that the variance of a random variable 
$F = f(X_1,\dots,X_N)$ controls all $L^p$ norms of $F$ for 
$p > 1$. 

We will follow a proof given in~\cite[Section~1.4]{nualart2006malliavin}, based 
on hypercontractivity of the Ornstein--Uhlenbeck semigroup. Other proofs can be found 
in~\cite[Section 4]{daPrato_Tubaro} and in the lecture 
notes~\cite[Section~7]{Hairer_Malliavin26}.

%%%%%%%%%%%%%%%%%%%%%%%%%%%%%%%%%%%%%%%%%%%%%%%%%%%%%%%%%%%%%%%%%%%%%%%%%%%%%%%%

\subsection{Ornstein--Uhlenbeck semigroup}
\label{ssec:OU} 

\begin{definition}[Ornstein--Uhlenbeck semigroup]
The Ornstein--Uhlenbeck semigroup is the one-parameter semigroup 
$\setsuch{T_t}{t\geqs0}$ of contraction operators on $\cH$ defined by 
\begin{equation}
\label{eq:def_OU_semigroup} 
 T_t(F) = \sum_{n=0}^\infty \e^{-nt}P_n F
\end{equation} 
for any $F\in\cH$, where $P_n:\cH\to\cH_n$ denotes the orthogonal 
projection on the $n$th Wiener chaos. 
\end{definition}

We have already encountered the Ornstein--Uhlenbeck process in 
Section~\ref{ssec:differential}. Consider first the one-dimensional 
case ($N = 1$). Then this process is defined as the solution of the 
stochastic differential equation in $\R$ 
\begin{equation}
\label{eq:SDE_OU} 
 \6X_t = -X_t\6t + \sqrt{2}\6W_t\;, 
 \qquad 
 X_0 = x\;,
\end{equation} 
where $(W_t)_{t\geqs0}$ denotes standard Brownian motion. By Ito's formula, 
for any $f:\R\to\R$ of class $\cC^2$, the process $Y_t = f(x_t)$ satisfies 
\begin{equation}
 \6Y_t = \bigbrak{-X_t f'(X_t) + f''(X_t)}\6t + \sqrt{2} f'(X_t)\6W_t\;.
\end{equation} 
In particular, if $f(x) = H_n(x)$, then the term in brackets is 
$-nH_n(X_t)$ (see~\eqref{eq:Hermite_eigenvalue} in 
Corollary~\ref{cor:Hermite_differential}), so that 
\begin{equation}
 \6Y_t = -nY_t\6t + \sqrt{2}f'(X_t)\6W_t\;.
\end{equation}
Integrating from $0$ to $t$, we find 
\begin{equation}
 H_n(X_t) = H_n(x) - n\int_0^t H_n(X_s)\6s 
 + \sqrt{2}\int_0^t f'(X_s)\6W_s\;.
\end{equation} 
Since the expectation of the stochastic integral is zero, it follows that  
\begin{equation}
 \expec{H_n(X_t)} = H_n(x) - n\int_0^t \expec{H_n(X_s)}\6s\;,
\end{equation} 
which implies 
\begin{equation}
 \frac{\6}{\6t} \expec{H_n(X_t)} = -n\expec{H_n(X_t)}\;, 
 \qquad 
 H_n(X_0) = H_n(x)\;.
\end{equation} 
The solution of this ordinary differential equation is simply 
\begin{equation}
 \expec{H_n(X_t)} = H_n(x)\e^{-nt}\;.
\end{equation} 
Consider now a general $f$ of the form 
\begin{equation}
 f(x) = \sum_{n\geqs0} c_n H_n(x)\;.
\end{equation} 
Then by linearity, we obtain 
\begin{equation}
 \expec{f(X_t)} = \sum_{n\geqs0} c_n H_n(x) \e^{-nt}
 = \sum_{n\geqs0} (P_n f))(x) \e^{-nt}
 = T_t(f)(x)\;,
\end{equation}
explaining why $T_t$ in~\eqref{eq:def_OU_semigroup} is called Ornstein--Uhlenbeck 
semigroup. 

Consider now the $N$-dimensional case. The $N$-dimensional Ornstein--Uhlenbeck 
is defined as the solution of the stochastic differential equation~\eqref{eq:SDE_OU} 
when $x_t\in\R^N$ and $W_t$ is $N$-dimensional Brownian motion. If $Y_t = f(X_t)$
for a twice differentiable $f:\R^2 \to \R$, Ito's formula now reads 
\begin{equation}
 \6Y_t = \bigbrak{-\pscal{\nabla f(X_t)}{X_t} + \Delta f(X_t)}\6t 
 + \sqrt{2} \pscal{\nabla f(X_t)}{\6W_t}\;.
\end{equation} 
The fact that the associated semigroup is again of the form~\eqref{eq:def_OU_semigroup}
is a consequence of the following observation. 

\begin{lemma}[Eigenfunctions of the Ornstein--Uhlenbeck generator in $\R^N$]
For $k\in\H=\R^N$ let
\begin{equation}
 f(x) = \Phi_k = \prod_{i=1}^N H_{k_i}(x_i)\;.
\end{equation} 
Then 
\begin{equation}
\label{eq:OU_eigenfunctions} 
 -\pscal{\nabla f(x)}{x} + \Delta f(x) = -\abs{k} f(x)\;.
\end{equation} 
\end{lemma}
\begin{proof}
This follows directly from the fact that 
\begin{equation}
 \frac{\partial f}{\partial x_i}
 = H_{k_i}'(x_i) \prod_{j\neq i} H_{k_j}(x_j)\;, \qquad 
 \frac{\partial f^2}{\partial x_i^2}
 = H_{k_i}''(x_i) \prod_{j\neq i} H_{k_j}(x_j)
\end{equation}
for any $i\in\intset{1,N}$. 
\end{proof}

\begin{exercise}
Check that~\eqref{eq:OU_eigenfunctions} implies that the semigroup of the 
$N$-dimensional Ornstein--Uhlenbeck process is given by~\eqref{eq:def_OU_semigroup}.
\end{exercise}

Using variation of constants, the solution of~\eqref{eq:SDE_OU} can be 
represented as 
\begin{equation}
 X_t = \e^{-t}x + \sqrt{2} \int_0^t \e^{-(t-s)}\6W_s\;,
\end{equation} 
as can be verified by applying Ito's formula. The second 
term on the right-hand side is Gaussian, centred, and by Ito's isometry, 
its variance is 
\begin{equation}
 2\int_0^t \e^{-2(t-s)} \6s = 1 - \e^{-2t}\;.
\end{equation} 
The Ornstein--Uhlenbeck process can thus be written as 
\begin{equation}
 X_t = \e^{-t}x + \sqrt{1-\e^{-2t}} X_t'\;,
\end{equation} 
where $X_t' \sim \cN(0,1)$ for any $t\geqs0$. 
This observation is the intuition behind the following result. 

\begin{proposition}[Mehler's formula]
\label{prop:Mehler} 
Let $W' = \setsuch{W'(h)}{h\in\H}$ be an independent copy of 
$W = \setsuch{W(h)}{h\in\H}$, where $W$ and $W'$ are 
defined on a product probability space 
$(\Omega\times\Omega',\cF\otimes\cF',\fP\times\fP')$. 
For $t>0$, consider the process $Z = \setsuch{Z(h)}{h\in\H}$ defined by 
\begin{equation}
  Z(h) = \e^{-t}W(h) + \sqrt{1-\e^{-2t}} W'(h)\;.
\end{equation} 
Then for any $F\in\cH$ of the form $F = f(W)$, one has 
\begin{equation}
\label{eq:Tt_Mehler} 
 T_t(F) = \E'\bigbrak{f(Z)}
\end{equation} 
where $\E'$ denotes the expectation with respect to the law $\fP'$ of $W'$. 
\end{proposition}
\begin{proof}
The process $Z(h)$ is Gaussian, centred, with covariance 
\begin{align}
 \expec{Z(h_1)Z(h_2)}
 &= \e^{-2t} \expec{W(h_1)W(h_2)} + (1-\e^{-2t}) \expec{W'(h_1)W'(h_2)} \\
\label{eq:Z_covariance} 
 &= \pscal{h_1}{h_2}_\H \\
 &= \expec{W(h_1)W(h_2)}
\end{align}
by~\eqref{eq:Wh_covariance}.
The right-hand side of~\eqref{eq:Tt_Mehler} defines a linear contraction
on any $L^p(\Omega)$ with $p>1$, because by Jensen's inequality
\begin{align}
 \expec{\abs{T_t(F)}^p}
 &= \Bigexpec{\bigabs{\E'\bigbrak{f(\e^{-t}W(h) + \sqrt{1-\e^{-2t}} W'(h))}}^p} \\
 &\leqs 
 \Bigexpec{\E'\bigbrak{\abs{f(\e^{-t}W(h) + \sqrt{1-\e^{-2t}} W'(h))}^p}}
 = \bigexpec{\abs{F}^p}\;.
\end{align}
Therefore, it is sufficient to check that~\eqref{eq:Tt_Mehler} holds for 
the generating function of Hermite polynomials 
$F = f(W) = \exp\set{W(h)-\frac12\norm{h}_\H^2}$ with any $h\in\H$. 
On one hand, we have 
\begin{align}
 F = G(1,W(h))  
 &= \sum_{n\geqs0} \frac{1}{n!} H_n(W(h);\norm{h}_H) \\
 &= \sum_{n\geqs0} \frac{1}{n!} \hat I_n(h^{\otimes n})
\end{align}
by Remark~\ref{rem:Wiener_isometry_scaled}. This implies 
\begin{equation}
\label{eq:TtF1} 
 T_t(F) = \sum_{n\geqs0} \frac{\e^{-nt}}{n!} \hat I_n(h^{\otimes n})\;.
\end{equation} 
On the other hand, we have 
\begin{equation}
 \E'\bigbrak{f(Z)} 
 = \E'\bigbrak{\e^{a+bW'(h)}}\;, \qquad 
 a = \e^{-t}W(h) - \frac12\norm{h}_\H^2\;, \qquad 
 b = \sqrt{1-\e^{-2t}}\;.
\end{equation} 
By Proposition~\ref{prop:Laplace_Gaussian_ndim} (Laplace transform) and the 
expression~\eqref{eq:Z_covariance} for the covariance of $Z$, this yields 
\begin{align}
 \E'\bigbrak{f(Z)} 
 &= \e^a \E'\bigbrak{\e^{\pscal{b}{W(h)}}} \\
 &= \exp\biggset{\e^{-t}W(h) - \frac12\norm{h}_\H^2
 + \frac12(1-\e^{-2t})\norm{h}_\H^2}\\
 &= \exp\biggset{\e^{-t}W(h) - \frac12\e^{-2t}\norm{h}_\H^2} \\
 &= G(\e^{-t}, W(h)) \\
 &= \sum_{n\geqs0} \frac{\e^{-nt}}{n!} H_n(W(h); \norm{h}_\H) \\
 &= \sum_{n\geqs0} \frac{1}{n!} \hat I_n(h^{\otimes n})\;.
\end{align}
This is equal to~\eqref{eq:TtF1}, which concludes the proof. 
\end{proof}

%%%%%%%%%%%%%%%%%%%%%%%%%%%%%%%%%%%%%%%%%%%%%%%%%%%%%%%%%%%%%%%%%%%%%%%%%%%%%%%%

\subsection{Hypercontractivity}
\label{ssec:hypercontractivity} 

We will denote by $L^p(\Omega,\cF,\fP)$ the Banach space of random variables 
$F:\Omega\to\R$ satisfying 
\begin{equation}
 \norm{F}_p = \expec{F^p}^{1/p} < \infty\;.
\end{equation} 
The following result, originally due to Nelson, says that the 
Ornstein--Uhlenbeck semigroup is \emph{hypercontractive}, meaning 
that if $t > 0$, then $T_t$ maps $L^p(\Omega,\cF,\fP)$ into 
$L^q(\Omega,\cF,\fP)$ for some $q = q(t)$ strictly larger than $p$. 
We follow essentially the proof of~\cite[Theorem~1.4.1]{nualart2006malliavin}.

\begin{theorem}[Hypercontractivity of the Ornstein--Uhlenbeck semigroup]
For $p>1$ and $t>0$, let 
\begin{equation}
\label{eq:hyper_qp} 
 q(t) = \e^{2t}(p-1) + 1 > p\;.
\end{equation} 
Then for any $F\in L^p(\Omega,\cF,\fP)$, one has 
\begin{equation}
\label{eq:hypercontractivity} 
 \norm{T_tF}_{q(t)} \leqs \norm{F}_p\;.
\end{equation} 
\end{theorem}
\begin{proof}
Let $q'$ be the H\"older conjugate of $q = q(t)$, that is, 
\begin{equation}
 \frac1q + \frac1{q'} = 1\;.
\end{equation} 
By duality, we have 
\begin{equation}
 \norm{T_tF}_q = \sup_{G\in L^{q'}} \frac{\pscal{T_tF}{G}}{\norm{G}_{q'}}
 = \sup_{G\in L^{q'}} \frac{\expec{(T_tF)G}}{\norm{G}_{q'}}\;.
\end{equation} 
We will thus have proved~\eqref{eq:hypercontractivity} if we manage 
to prove that 
\begin{equation}
 \expec{(T_tF)G} \leqs \norm{F}_p \norm{G}_{q'}
\end{equation} 
for all $G\in L^{q'}(\Omega,\cF,\fP)$. 
Since $T_t$ is non-negative, and thus $\abs{T_t F} \leqs T_t(\abs{F})$, 
we may assume that $F$ and $G$ are both non-negative. In fact, by an 
approximation argument, we may assume that $a \leqs F, G\leqs b < \infty$ 
for some $b \geqs a > 0$. 

The main idea is to use a kind of interpolation argument. 
Recall that both $F$ and $G$ are functions of the random variables 
$W(e_i) = X_i$ with $i\in\intset{1,N}$. By Mehler's 
identity~\eqref{eq:Tt_Mehler}, 
\begin{equation}
 T_t F = \E'\bigbrak{f(Z_1,\dots,Z_N)}\;,
\end{equation} 
where
\begin{equation}
 Z_i = \e^{-t} X_i + \sqrt{1-\e^{-2t}}X_i'\;,
\end{equation} 
the $X_i'$ being independent copies of the $X_i$. We can represent these 
variables as 
\begin{equation}
 X_i = \int_0^1 \6W_i(s)\;, \qquad 
 X_i' = \int_0^1 \6W_i'(s)\;,
\end{equation} 
where $(W_i(s))_{0\leqs s\leqs 1}$ and $(W_i'(s))_{0\leqs s\leqs 1}$ 
are independent Brownian motions. Then we have 
\begin{equation}
 Z_i = \int_0^1 \6 B_i(s)\;, \qquad 
 B_i(s) = \e^{-t} W_i(s) + \sqrt{1-\e^{-2t}} W_i'(s)\;.
\end{equation} 
Therefore, we can write 
\begin{equation}
 \expec{(T_tF)G} = \expec{PQ}\;,
\end{equation} 
where
\begin{equation}
 P = f(B_1,\dots,B_N)\;, \qquad 
 Q = g(W_1,\dots,W_N)\;.
\end{equation} 
By Ito's formula, $P^p$ and $Q^{q'}$ have integral representations of the form 
\begin{equation}
 P^p = \expec{P^p} + \int_0^1 \ph(s) \6B(s)\;, \qquad 
 Q^{q'} = \expec{Q^{q'}} + \int_0^1 \psi(s) \6W(s) 
\end{equation}
for some bounded, positive $\ph$ and $\psi$. 
We introduce two bounded, positive martingales 
\begin{equation}
 M(s) = \expec{P^p} + \int_0^s \ph(u) \6B(u)\;, \qquad 
 N(s) = \expec{Q^{q'}} + \int_0^s \psi(u) \6W(u)\;,
\end{equation} 
which satisfy 
\begin{align}
 M(0) &= \expec{P^p}\;,
 & N(0) &= \expec{Q^{q'}}\;, \\
 M(1) &= P^p\;, 
 & N(1) &= Q^{q'}\;.
\end{align} 
Define $f(x,y) = x^\alpha y^\beta$, with $\alpha = \frac1p$ and $\beta = \frac{1}{q'}$. 
Then $U(s) = f(M(s),N(s))$ satisfies 
\begin{equation}
 U(0) = \norm{P}_p \norm{Q}_{q'}\;, \qquad 
 U(1) = PQ\;.
\end{equation}
By Ito's formula, we have 
\begin{equation}
 \6U(s) = \frac{\partial f}{\partial x} \6M(s) 
 + \frac{\partial f}{\partial y} \6N(s)
 + \frac12\frac{\partial^2 f}{\partial x^2} \6M(s)^2 
 + \frac12\frac{\partial^2 f}{\partial y^2} \6N(s)^2 
 + \frac{\partial^2 f}{\partial x^2} \6M(s)\6N(s)\;,
\end{equation} 
where 
\begin{align}
 \6M(s)^2 &= \ph(s)^2 \6s\;, \qquad 
 \6N(s)^2 = \psi(s)^2 \6s\;, \\
 \6M(s)\6N(s) &= \ph(s)\psi(s) \6B(s)\6W'(s)
 = \ph(s)\psi(s) \e^{-s}\6s\;.
\end{align} 
Computing the partial derivatives of $f$, this leads to 
\begin{equation}
 \6U(s) = \alpha M(s)^{\alpha-1} N(s)^\beta \6M(s)
 + \beta M(s)^{\alpha} N(s)^{\beta-1} \6N(s)
 + \frac12 M(s)^\alpha N(s)^\beta A(s) \6s\;,
\end{equation} 
where
\begin{align}
 A(s) &= \alpha(\alpha-1) M(s)^{-2}\ph(s)^2 
 + \beta(\beta-1) N(s)^{-2}\psi(s)^2
 + 2\alpha\beta M(s)^{-1}N(s)^{-1} \ph(s)\psi(s) \e^{-t} \\
 &=
 \begin{pmatrix}
  \ph(s)M(s)^{-1} & \psi(s)N(s)^{-1}
 \end{pmatrix}
 \begin{pmatrix}
  \alpha(\alpha - 1) & \alpha\beta\e^{-t} \\
  \alpha\beta\e^{-t} & \beta(\beta-1)
 \end{pmatrix}
 \begin{pmatrix}
  \ph(s)M(s)^{-1} \\ \psi(s)N(s)^{-1}
 \end{pmatrix}\;.
\label{eq:hyper_matrix} 
\end{align} 
We have 
\begin{equation}
 PQ = U(1) = U(0) + \int_0^1 \6U(s) 
 = \norm{P}_p \norm{Q}_{q'}  + \int_0^1 \6U(s)\;.
\end{equation} 
Taking expectations, since $M(s)$ and $N(s)$ are martingales, we obtain 
\begin{equation}
\expec{(T_tF)G} = 
 \expec{PQ} = \norm{P}_p \norm{Q}_{q'} 
 + \frac12 \int_0^1 \expec{M(s)^\alpha N(s)^\beta A(s)} \6s\;.
\end{equation} 
The result will thus be proved if we manage to show that $A(s)\leqs0$
for all $s\in[0,1]$. 
Note that if $p > 1$, then one finds $\alpha(1-\alpha)<0$ and $\beta(1-\beta)<0$. 
It thus suffices to show that the determinant of the matrix in~\eqref{eq:hyper_matrix}
is non-negative. This is equivalent to~\eqref{eq:hyper_qp}. 
\end{proof}

It is now easy to prove equivalence of moments.

\begin{proof}[{\sc Proof of Theorem~\ref{thm:equivalence_moments}:}]
We make the change of index $q(t) \mapsto 2p$ and $p\mapsto 2$. Then~\eqref{eq:hypercontractivity}
becomes 
\begin{equation}
 \norm{T_tF}_{2p} \leqs \norm{F}_2\;,
\end{equation} 
where 
\begin{equation}
 \e^{2t} = 2p-1\;.
\end{equation} 
In particular, for $F\in\cH_n$, we have 
\begin{equation}
 \norm{T_tF}_{2p} = \e^{-nt} \norm{F}_{2p}\;.
\end{equation} 
This yields 
\begin{equation}
 \norm{F}_{2p} = \e^{nt} \norm{T_tF}_{2p}
 \leqs \e^{nt} \norm{F}_2
 = (2p-1)^{n/2} \norm{F}_2\;,
\end{equation} 
which is equivalent to~\eqref{eq:equivalence_moments}. 
\end{proof}

%%%%%%%%%%%%%%%%%%%%%%%%%%%%%%%%%%%%%%%%%%%%%%%%%%%%%%%%%%%%%%%%%%%%%%%%%%%%%%%%

\chapter{Gaussian fields}
\label{chap:fields} 

In this chapter, we extend the results of the previous chapter to the 
infinite-dimensional case, and discuss two particularly important cases 
of Gaussian fields, namely white noise and the Gaussian free field.

%%%%%%%%%%%%%%%%%%%%%%%%%%%%%%%%%%%%%%%%%%%%%%%%%%%%%%%%%%%%%%%%%%%%%%%%%%%%%%%%

\section{Isonormal Gaussian processes}
\label{sec:isonormal} 

In this chapter, we will mostly be concerned with the following set-up. 
Let $\Lambda = \T^d = (\R/\Z)^d$ be the $d$-dimensional torus, for some $d\geqs1$. Let 
\begin{equation}
 \H = L^2(\Lambda,\6x)
\end{equation} 
be the Hilbert space of square-integrable functions $h:\T^d\to\R$. 
Let $(e_i)_{i\geqs0}$ be an orthonormal basis of $\H$ -- we will typically 
choose a Fourier basis. We are then interested in random fields of the 
form 
\begin{equation}
 \phi(x) = \sum_{i\geqs0} X_i e_i(x)\;,
\end{equation} 
where the $X_i$ are independent, centred jointly Gaussian random variables of unit 
variance, defined on a common probability space $(\Omega,\cF,\fP)$. 
We will see that this is related to the following general construction.

\begin{definition}[Isonormal Gaussian process]
Let $\H$ be a separable Hilbert space. 
A stochastic process $W = \setsuch{W(h)}{h\in\H}$ defined on a complete 
probability space $(\Omega,\cF,\fP)$ is an \emph{isonormal Gaussian process} 
if $W$ is a centred Gaussian family of random variables such that 
\begin{equation}
 \expec{W(h_1)W(h_2)} = \pscal{h_1}{h_2}_\H
\end{equation} 
for all $h_1, h_2\in\H$. 
\end{definition}

Note that the map $h\mapsto W(h)$ is necessarily linear. This is because 
for any $h_1, h_2\in\H$, and any $\lambda, \mu\in\R$, 
\begin{align}
 \expec{\bigpar{W(\lambda h_1 + \mu h_2) - &\lambda W(h_1) - \mu W(h_2)}^2} \\
 ={}& \norm{\lambda h_1 + \mu h_2}_\H^2 + \lambda^2 \norm{h_1}_\H^2 + \mu^2 \norm{h_2}_\H^2 \\
 &{}- 2\lambda\pscal{\lambda h_1 + \mu h_2}{h_1}_\H
 - 2\mu\pscal{\lambda h_1 + \mu h_2}{h_2}_\H 
 + 2\lambda\mu\pscal{h_1}{h_2}_\H^2 
 = 0\;.
\end{align}
As a consequence, if each random variable is Gaussian and centred, the set 
$\setsuch{W(h)}{h\in\H}$ is automatically a family of jointly Gaussian random variables. 
The existence of the probability space $(\Omega,\cF,\fP)$ follows from 
Kolmogorov's continuity theorem. 

%%%%%%%%%%%%%%%%%%%%%%%%%%%%%%%%%%%%%%%%%%%%%%%%%%%%%%%%%%%%%%%%%%%%%%%%%%%%%%%%

\subsection{Wiener chaos expansion and Wiener isometry}
\label{ssec:Wiener_chaos_infdim} 

The definition of Wiener chaoses is the same in infinite dimension 
as in finite dimension. 

\begin{definition}[Wiener chaos, infinite-dimensional version]
For any $n\geqs 1$, we denote by $\cH_n$ the subspace of 
$\cH = L^2(\Omega,\cF,\fP)$ spanned by the random variables 
\begin{equation}
 \setsuch{H_n(W(h))}{h\in \H, \norm{h}_\H = 1}\;.
\end{equation} 
For $n = 0$, $\cH_0$ is the set
of constants, which is isomorphic to $\R$. Then $\cH_n$ is called the 
\emph{homogeneous Wiener chaos of order $n$}. 
The \emph{inhomogeneous Wiener chaos of order $n$} is defined as 
\begin{equation}
 \cH_{\leqs n} = \bigoplus_{k=0}^n \cH_k\;.
\end{equation} 
\end{definition}

We also have an analogue of Theorem~\ref{thm:Wiener_chaos_N}, with the same 
proof (see also~\cite[Theorem~1.1.1]{nualart2006malliavin}).

\begin{theorem}[Wiener chaos decomposition]
\label{thm:Wiener_chaos_inf} 
The Hilbert space $\cH$ can be decomposed into the infinite orthogonal sum 
\begin{equation}
 \cH = \bigoplus_{n=0}^\infty \cH_n\;.
\end{equation} 
\end{theorem}

The construction of the Wiener isometry proceeds in the same way as we have 
seen in Section~\ref{ssec:Wiener_Fock}, except that one now assumes that 
the multiindex $k\in\N_0^\N$ has only finitely many non-zero entries. 
In this way, 
\begin{equation}
 \Phi_k = \prod_{i\geqs0\colon k_i>0} H_{k_i}(X_i)
\end{equation} 
is well-defined, as are 
\begin{equation}
 \abs{k} = \sum_{i\geqs0\colon k_i>0} \abs{k_i} 
 \qquad\text{and}\qquad 
 k! = \prod_{i\geqs0\colon k_i>0} k_i!\;.
\end{equation} 
The Wiener isometry, with its two normalisation conventions, is again defined as 
\begin{equation}
 I_n(e_k) = \frac{1}{\sqrt{n!}}\Phi_k\;, \qquad 
 \hat I_n(e_k) = \Phi_k
\end{equation} 
for all $k\in\N_0^\N$ with $\abs{k} = n$, where 
\begin{equation}
 e_k = \Pi \bigotimes_{i\geqs0} e_i^{\otimes k_i}\;,
\end{equation} 
$\Pi$ being the symmetrisation operator~\eqref{eq:def_Pi}. Note that we have slightly 
overloaded the notation for $e$, as we use the same letter for elements of $\H$ 
and elements of Fock space. It should always be clear from the context which 
basis vector is meant.

Let us give some simple examples for clarity. 

\begin{example}
\begin{itemize}
\item   Assume that $k$ has only one nonzero entry $k_i = 1$. Then $e_k = e_i$ 
and $\hat I_1(e_k) = H_1(X_i) = X_i$. 

\item   If $k$ has two nonzero entries $k_i = k_j = 1$ with $i\neq j$, then 
$e_k = \frac12(e_i\otimes e_j + e_j\otimes e_i)$, and 
$\hat I_2(e_k) = X_iX_j$.

\item   If $k$ has one nonzero entry $k_i = 2$, then 
$e_k = e_i\otimes e_i$ and $\hat I_2(e_k) = H_2(X_i)$. 
\end{itemize}
\end{example}

The bound showing equivalence of moments, cf.\ 
Theorem~\ref{thm:equivalence_moments}, can be proved in essentially 
the same way as we did in finite dimension, so we just repeat it here. 

\begin{theorem}[Equivalence of moments]
\label{thm:equivalence_moments_infdim} 
Assume $F$ belongs to the $n$th Wiener chaos $\cH_n$. Then for any 
$p > 1$, one has 
\begin{equation}
\label{eq:equivalence_moments_infdim} 
 \expec{F^{2p}}^{1/2p} \leqs (2p-1)^{n/2} \expec{F^2}^{1/2}\;.
\end{equation} 
\end{theorem}

\begin{exercise}
\begin{itemize}
\item  Let $f, g\in\H$. Compute $\expec{\hat I_1(f)\hat I_1(g)}$. 

\item  Do the same for $\expec{\hat I_n(f)\hat I_m(g)}$, when 
$f\in\H^{\otimes n}$ and $g\in \H^{\otimes m}$. 
\end{itemize}
\end{exercise}

%%%%%%%%%%%%%%%%%%%%%%%%%%%%%%%%%%%%%%%%%%%%%%%%%%%%%%%%%%%%%%%%%%%%%%%%%%%%%%%%

\subsection{The case of $L^2(\T^d)$}
\label{ssec:gaussian_torus}

As mentioned above, we will mainly be concerned with the case $\H = L^2(\Lambda,\6x)$, 
with $\Lambda = \T^d$, endowed with a Fourier basis $(e_i)_{i\geqs0}$. Elements of 
$\H$ are thus functions $h:\Lambda\to\R$ that can be written as a Fourier series 
\begin{equation}
\label{eq:Fourier} 
 h(x) = \sum_{i\geqs0} \hat h(i) e_i(x)\;.
\end{equation} 
Recall that~\eqref{eq:Fourier} defines an isometry between $\H$ and the Hilbert 
space 
\begin{equation}
 \widehat\H = \ell^2 
 = \biggsetsuch{\hat h\in\R^{\N_0}}{\sum_{i\geqs 0}\hat h(i)^2 < \infty}\;.
\end{equation} 
Indeed, by Parseval's relation, we have 
\begin{equation}
 \norm{h}_{\H}^2 
 = \int_\Lambda h(x)^2 \6x 
 = \sum_{i\geqs 0}\hat h(i)^2 
 = \norm{\hat{h}}_{\widehat\H}^2\;.
\end{equation} 
It is thus equivalent to work in $\widehat \H$, which is essentially the 
same as what we did in the finite-dimensional case, or in $\H$. 

Elements of $\H^{\otimes n}$ can be written either as 
\begin{equation}
 h = h_1\otimes\dots\otimes h_n 
 = \sum_{i_1\geqs0,\dots,i_n\geqs0}
 \hat h_1(i_1)\dots \hat h_n(i_n) e_{i_1}\otimes\dots\otimes e_{i_n}\;,
\end{equation} 
or, equivalently, as 
\begin{equation}
\label{eq:f_tensor} 
 h(x_1,\dots,x_n) 
 = \sum_{i_1\geqs0,\dots,i_n\geqs0} \hat h(i_1,\dots,i_n) 
 e_{i_1}(x_{i_1})\dots e_{i_n}(x_{i_n})\;,
\end{equation} 
where $\hat h(i_1,\dots,i_n) = \hat h_1(i_1)\dots \hat h_n(i_n)$.
In this way, elements of $\Hsym{n}$ are represented by functions $h$ that 
are symmetric in all their arguments, that is, 
\begin{equation}
 h(\sigma(x_1,\dots,x_n)) = h(x_1,\dots,x_n) 
 \qquad 
 \forall \sigma\in\frakS_n\;.
\end{equation} 
The definition of contractions can then be rewritten in the following way.

\begin{lemma}
Let $f\in \H^{\otimes n}$ and $g\in \H^{\otimes m}$. For any $p\leqs n\wedge m$, 
all $x\in\Lambda^{n-p}$ and all $y\in\Lambda^{m-p}$, one has
\begin{equation}
 (f \star_p g)(x,y) 
 = \sum_{\substack{\Sigma\in\frakS(p,n)\\ \bar\Sigma\in\frakS(p,m)}}
 \sum_{\sigma\in\frakS(p)}
 \int_{\Lambda^p} f\bigpar{\Sigma(z,x)}g\bigpar{\bar\Sigma(z,\sigma(x))} \6z\;.
\end{equation} 
\end{lemma}
\begin{proof}
Consider the case $m = 1$. Then~\eqref{eq:def_starp} yields 
\begin{equation}
\label{eq:proof_starp01} 
 (\hat f\star_1\hat g)(i_1,\dots,i_{n-1})
 = \sum_{\Sigma\in\frakS(1,n)} \sum_{j\geqs0} 
 \hat f(\Sigma(j,i_1,\dots,i_{n-1})) \hat g(j)\;.
\end{equation} 
By~\eqref{eq:f_tensor}, we have 
\begin{equation}
 (f\star_1 g)(x_1,\dots,x_{n-1}) 
 = \sum_{i_1,\dots,i_{n-1}\geqs0} (\hat f\star_1\hat g)(i_1,\dots,i_{n-1})
 e_{i_1}(x_{i_1})\dots e_{i_{n-1}}(x_{i_{n-1}})\;.
\end{equation} 
Consider the contribution of the identity permutation to~\eqref{eq:proof_starp01}, 
which is
\begin{equation}
\label{eq:proof_starp02} 
 \sum_{j,i_1,\dots,i_{n-1}\geqs0}
 \hat f_1(j) \hat f_2(i_1)\dots \hat f_n(i_{n-1}) \hat g(j)
 e_{i_1}(x_{i_1})\dots e_{i_{n-1}}(x_{i_{n-1}})\;.
\end{equation} 
On the other hand, we have 
\begin{align}
\label{eq:proof_starp03} 
 \int_\Lambda &f(z,x_1,\dots,x_{n-1})g(z)\6z \\
 &= \int_\Lambda \sum_{j,i_1,\dots,i_{n-1},k\geqs0}  
  \hat f_1(j) \hat f_2(i_1)\dots \hat f_n(i_{n-1}) \hat g(k) 
  e_j(z) e_k(z) e_{i_1}(x_{i_1})\dots e_{i_{n-1}}(x_{i_{n-1}})
  \6z\;.
\end{align}
Orthogonality of eigenfunctions implies 
\begin{equation}
 \int_\Lambda e_j(z) e_k(z) \6z = \delta_{jk}\;.
\end{equation} 
% The element $1\otimes e_i(x)$ can be canonically identified with $e_i(x)$.
Therefore, \eqref{eq:proof_starp03} is equal to \eqref{eq:proof_starp02}.
The argument is similar for other permutations, and for $m>1$. 
\end{proof}

%%%%%%%%%%%%%%%%%%%%%%%%%%%%%%%%%%%%%%%%%%%%%%%%%%%%%%%%%%%%%%%%%%%%%%%%%%%%%%%%

\subsection{A construction of Gaussian fields}
\label{ssec:gaussian_field} 

Consider now the following construction. For $h = \sum_{i\geqs0} 
\hat h(i)e_i \in\H$, we define 
\begin{equation}
\label{eq:def_Psi} 
 \Psi(h) = \sum_{i\geqs0} \hat h(i) W(e_i)e_i
 = \sum_{i\geqs0} \hat h(i) X_i e_i\;.
\end{equation} 
This is now an $\H$-valued random variable, whose value at point $x\in\Lambda$ is  
\begin{equation}
 \Psi(h)(x)  = \sum_{i\geqs0} \hat h(i) X_i e_i(x)\;.
\end{equation} 
Note furthermore that 
\begin{align}
 \norm{\Psi(h)}_\H^2 
 &= \int_\Lambda \Psi(h)(x)^2 \6x \\
 &= \sum_{i,j\geqs0} \hat h(i) \hat h(j) X_iX_j \int_\Lambda e_i(x)e_j(x)\6x \\
 &= \sum_{i\geqs0} \hat h(i)^2 X_i^2
\end{align} 
by orthonormality of the $e_i$. Therefore, 
\begin{equation}
 \bigexpec{\norm{\Psi(h)}_\H^2} 
 = \sum_{i\geqs0} \hat h(i)^2 
 = \norm{h}_{\H}^2\;.
\end{equation} 
This shows that the map $\Psi$ is an isometry from the Hilbert space $\H$ 
to a subset of the Hilbert space $\widetilde\cH$ of $\H$-valued random 
variables that have finite variance. 

In terms of Wiener chaos, since $\Psi(h)$ is defined on the same 
probability space $(\Omega,\cF,\fP)$ as the $X_i$, if makes sense to 
decompose $\widetilde\cH$ into Wiener chaoses, by viewing them 
as random variables indexed by $x\in\Lambda$. In particular,
$\Psi(h)$ belongs to the first Wiener chaos. This allows us to work 
within the framework of separable Hilbert space. For a general construction of 
Gaussian measures on separable Banach spaces, see 
Chapter~3 of Martin Hairer's lecture notes~\cite{Hairer_LN_2009}.

%%%%%%%%%%%%%%%%%%%%%%%%%%%%%%%%%%%%%%%%%%%%%%%%%%%%%%%%%%%%%%%%%%%%%%%%%%%%%%%%

\section{Gaussian white noise}
\label{sec:white_noise} 

%%%%%%%%%%%%%%%%%%%%%%%%%%%%%%%%%%%%%%%%%%%%%%%%%%%%%%%%%%%%%%%%%%%%%%%%%%%%%%%%

\subsection{Definition and basic properties}
\label{ssec:white_noise_prop} 

Consider the case where 
\begin{equation}
 \hat h = (1,1,1,\dots)
\end{equation} 
is the vector all of whose components are equal to $1$. Then~\eqref{eq:def_Psi}
becomes
\begin{equation}
\label{eq:def_xi} 
 \xi(x) := \Psi(h)(x) = \sum_{i\geqs0} X_i e_i(x)\;.
\end{equation} 
This random variable is called \emph{white noise} on $\Lambda$, because 
it means that every Fourier mode is random with the same variance. 
The trouble is that $h$ does not belong to $\H$, since $\hat h$ is not 
square-summable. As a result, $\xi$ does not have finite variance. 

One way to try to make sense of this definition is to set, for any 
finite $N\in\N$, 
\begin{equation}
 \hat h_N = (\underbrace{1,1,1,\dots,1}_{\text{$N$ components}}, 0, 0, \dots)\;.
\end{equation} 
In this way, we obtain
\begin{equation}
 \xi_N(x) := \Psi(h_N)(x) = \sum_{i=0}^N X_i e_i(x)\;.
\end{equation} 
Since $h_N$ is in $\H$ for any finite $N$, $\xi_N$ is a random function 
with finite variance for these $N$. It is called a \emph{mollification 
with cut-off $N$} of white noise. Of course, the variance of $\xi_N$ 
diverges as $N$ goes to infinity.

Another way to make sense of $\xi$ is to view it as a random 
distribution. Let $\ph:\Lambda\to\R$ be a sufficiently regular so-called 
\emph{test function}. Then we have, at least formally,  
\begin{align}
 \pscal{\xi}{\ph} 
 &= \int_\Lambda \xi(x) \ph(x) \6x 
 = \sum_{i\geqs0} X_i \int_\Lambda e_i(x) \ph(x) \6x 
 = \sum_{i\geqs0} X_i \hat\ph(i)\;.
\end{align}
This is indeed well-defined for $\ph\in\H$. Furthermore,
for any $\ph_1,\ph_2\in\H$, we have 
\begin{align}
 \bigexpec{\pscal{\xi}{\ph_1}\pscal{\xi}{\ph_2}}
 &= \sum_{i\geqs0} \sum_{j\geqs0} \expec{X_iX_j} \hat\ph_1(i)\hat\ph_2(j) \\
 &= \sum_{i\geqs0} \hat\ph_1(i)\hat\ph_2(i) 
 = \pscal{\ph_1}{\ph_2}_\H\;.
\end{align}
This motivates the following definition.

\begin{definition}[Gaussian white noise on the torus]
\label{def:white_noise} 
Gaussian white noise on $\T^d$ is the random distribution $\xi$ on 
$(\Omega,\cF,\fP)$ such that for any smooth test function $\ph\in\H$, 
$\pscal{\xi}{\ph}$ is a centred Gaussian random variable of variance 
$\norm{\ph}_\H^2$, while the covariance is given by  
\begin{equation}
\label{eq:white_noise_cov} 
 \bigexpec{\pscal{\xi}{\ph_1}\pscal{\xi}{\ph_2}} = \pscal{\ph_1}{\ph_2}_\H
\end{equation} 
for any two smooth test functions $\ph_1,\ph_2\in\H$. 
\end{definition}

One immediate consequence of this definition is that if $\ph_1$ and 
$\ph_2$ have disjoint support (meaning that they cannot be different 
from zero at the same point), then $\pscal{\ph_1}{\ph_2}_\H = 0$, and 
therefore $\pscal{\xi}{\ph_1}$ and $\pscal{\xi}{\ph_2}$ are independent
by Proposition~\ref{prop:Gaussian_basic}. The relation~\eqref{eq:white_noise_cov}
is sometimes formally written 
\begin{equation}
 \bigexpec{\xi(x)\xi(y)} = \delta(x-y)\;,
\end{equation} 
where $\delta(x-y)$ is the Dirac distribution, that can be formally obtained 
as a limit of scaled test functions localised at $x-y$. 

Another important property of white noise is related to scaling. To this end, 
define for $\lambda\in(0,1]$ a scaling operator $\cS^\lambda$ acting on 
$\ph\in\H$ by 
\begin{equation}
\label{eq:scaling_op} 
 (\cS^\lambda \ph)(x) = \frac{1}{\lambda^d} \ph\biggpar{\frac{x}{\lambda}}\;.
\end{equation} 
Let $\xi_\lambda$ be the distribution defined by 
\begin{equation}
 \pscal{\xi_\lambda}{\ph} = \pscal{\xi}{\cS^\lambda\ph}
\end{equation} 
for any test function $\ph$. 

\begin{exercise}
Show that if $h:\Lambda\to\R$ is an integrable function, then 
$h_\lambda(x) = h(\lambda x)$. 
\end{exercise}

\begin{proposition}[Scaling of white noise]
\label{prop:white_noise_scaling}
For any $\lambda\in(0,1]$, one has equality in law 
\begin{equation}
 \xi_\lambda \eqinlaw \frac{1}{\lambda^{d/2}} \xi\;.
 \label{eq:xi_scaling} 
\end{equation} 
\end{proposition}
\begin{proof}
Both processes are Gaussian and centred. Therefore, it suffices to show that 
they have the same covariance. Given two compactly supported test functions 
$\ph_1, \ph_2$, we have 
by~\eqref{eq:white_noise_cov} 
\begin{align}
\bigexpec{\pscal{\xi_{\lambda}}{\ph_1} \pscal{\xi_{\lambda}}{\ph_2}}
&= \bigexpec{\pscal{\xi}{\cS^{\lambda}\ph_1} 
{\pscal{\xi}{\cS^{\lambda}\ph_2}}} \\
&= \int_{\Lambda} (\cS^{\lambda}\ph_1)(x) 
(\cS^{\lambda}\ph_2)(x) \6x \\
&= \frac{1}{\lambda^{2d}} \int_{\Lambda} \ph_1\biggpar{\frac{x}{\lambda}}
\ph_2\biggpar{\frac{x}{\lambda}}\6x \\
&= \frac{1}{\lambda^{d}} \pscal{\ph_1}{\ph_2}_\cH\;,
\label{eq:proof_xi_scaling} 
\end{align}
where we have not changed the domain of integration, because the 
$\ph_i$ are compactly supported. 
This is indeed the covariance of $\lambda^{-d/2}\xi$. 
\end{proof}

The scaling property~\eqref{eq:xi_scaling} can be interpreted as a form of 
self-similarity, in a statistical sense. 

%%%%%%%%%%%%%%%%%%%%%%%%%%%%%%%%%%%%%%%%%%%%%%%%%%%%%%%%%%%%%%%%%%%%%%%%%%%%%%%%

\subsection{Regularity of white noise*}
\label{ssec:white_noise_reg} 

Even though white noise is very irregular, one can find function spaces to 
which it belongs. Two important such function spaces are fractional Sobolev spaces 
and Besov--H\"older spaces of negative regularity index.

In order to define fractional Sobolev spaces, it is more convenient to 
index basis elements of $\H = L^2(\Lambda)$ by their wave number. 
This means that they are of the form 
\begin{equation}
 e_k(x) = \e^{2\pi\icx \pscal{k}{x}}\;, \qquad 
 k\in\Z^d\;,
\end{equation}
where $k$ should not be confused with the index $k$ used in the Wiener isometry. 
We then write 
\begin{equation}
 f(x) = \sum_{k\in\Z^d} \hat f(k) e_k(x)
\end{equation} 
for the Fourier series of a function $f\in\H$. 

\begin{remark}[Complex versus real Fourier series]
\label{rem:Fourier_complex} 
When one uses complex Fourier series, because of the reality condition 
$\smash{\overline{\hat f(k)}} = \hat f(-k)$, one should not take independent $X_k$ in the 
underlying probability space. Instead, they should satisfy 
\begin{equation}
 \expec{X_kX_\ell} = \delta_{k,-\ell}\;.
\end{equation} 
It is also possible to use real instead of complex Fourier series and independent 
$X_k$, but this makes some computation slightly more tedious.  
\end{remark}

Note that we have 
\begin{equation}
\label{eq:def_lambdak} 
 (\id - \Delta) e_k(x) = \lambda_k e_k(x)\;, \qquad 
 \lambda_k = 1 + (2\pi)^d \norm{k}^2\;.
\end{equation} 
The eigenvalues $\lambda_k$ of the positive operator $\id - \Delta$ play 
a natural role as weights in Sobolev spaces. 

\begin{definition}[Fractional Sobolev spaces]
\label{def:Sobolev_space} 
For $s\geqs0$, the \emph{fractional Sobolev space} $H^s(\Lambda)$ is the 
space of functions $f\in \H = L^2(\Lambda)$ such that 
\begin{equation}
\label{eq:Sobolev_norm} 
 \norm{f}_{H^s}^2 := \sum_{k\in\Z^d} \lambda_k^s \abs{\hat f(k)}^2 < \infty\;.
\end{equation} 
In particular, $H^0(\Lambda) = L^2(\Lambda)$. For $s<0$, $H^s(\Lambda)$ is 
the closure of $L^2(\Lambda)$ under the norm \eqref{eq:Sobolev_norm}.
\end{definition}

Note that in the definition~\eqref{eq:Sobolev_norm} of the fractional Sobolev 
norm, one may replace the weight $\lambda_k^s$ by $(1 + \norm{k}^2)^s$.
The resulting norm is equivalent. 

\begin{proposition}[Sobolev regularity of white noise on the torus]
\label{prop:xi_Sobolev} 
White noise $\xi$ belongs to $H^s$ for any $s < -\frac{d}{2}$, in the 
sense that 
\begin{equation}
 \bigexpec{\norm{\xi}_{H^s}^2} < \infty 
 \qquad 
 \text{for all $s < -\dfrac{d}{2}$}\;.
\end{equation} 
\end{proposition}
\begin{proof}
It follows from the Fourier representation~\eqref{eq:def_xi} of white noise 
that 
\begin{equation}
 \norm{\xi}_{H^s}^2 = \sum_{k\in\Z^d} \lambda_k^s X_k^2\;,
\end{equation} 
where the $X_k$ are independent standard Gaussians. Taking the expectation, 
we obtain 
\begin{equation}
 \bigexpec{\norm{\xi}_{H^s}^2} = \sum_{k\in\Z^d} \lambda_k^s\;.
\end{equation} 
The sum is comparable to the integral 
\begin{equation}
 \int_{\R^d} (1+\norm{y}^2)^s \6y 
 \asymp \int_1^\infty r^{2s} r^{d-1} \6r\;,
\end{equation} 
which is convergent if and only if $s < -\frac{d}{2}$. 
Here we write $x\asymp y$ to indicate that $c^{-1}x \leqs y \leqs cx$
for some $c\geqs1$. 
\end{proof}

A second scale of function spaces that are useful when working with white noise 
are so-called H\"older--Besov spaces. To define them, we first introduce a 
generalisation of the scaling operator~\eqref{eq:scaling_op} given by 
\begin{equation}
\label{eq:scaling_op_x} 
 (\cS^\lambda_x \ph)(y) = \frac{1}{\lambda^d} \ph\biggpar{\frac{y-x}{\lambda}}\;.
\end{equation} 
For $r\in\N$, we denote by $B_r$ the set of smooth test functions $\ph:\Lambda\to\R$, 
supported on a ball or radius $1$, whose partial derivatives up to order $r$ are 
bounded by $1$. 

\begin{definition}[H\"older--Besov spaces]
For $\alpha < 0$, the space $\cC^\alpha(\Lambda)$ consists in all Schwartz 
distributions $\zeta\in\cS'(\Lambda)$ such that 
\begin{equation}
\label{eq:Holder_negative} 
 \norm{\zeta}_{\cC^\alpha} 
 = \sup_{x\in\Lambda} \sup_{\ph\in B_r} \sup_{\lambda\in(0,1]} 
 \biggabs{\frac{\pscal{\zeta}{\cS^\lambda_x\ph}}{\lambda^\alpha}} < \infty\;,
\end{equation} 
where $r = \lceil -\alpha \rceil$. 
\end{definition}

This definition says that if $\zeta\in\cC^\alpha$, then 
$\pscal{\zeta}{\cS^\lambda_x\ph}$ diverges at most like $\lambda^{-\alpha}$ 
for any $x\in\Lambda$. It is thus a measure of how far the distribution is 
from admitting a finite value at $x$. 

\begin{proposition}[H\"older--Besov regularity of white noise on the torus]
\label{prop:xi_Holder} 
White noise $\xi$ belongs to $\cC^\alpha$ for any $\alpha < -\frac{d}{2}$.
\end{proposition}
\begin{proof}[{\sc Sketch of proof}]
We follow the argument outlined in~\cite[Theorem~2.7]{Chandra_Weber_LN17}.
Given $k\in\N_0$, we introduce a dyadic lattice discretisation of 
$\Lambda$ on scale $2^{-k}$, given by 
\begin{equation}
 \Lambda_k = (2^{-k}\Z)^d\cap\Lambda\;.
\end{equation} 
For any $\alpha\in\R$, one can show the existence of a test function 
$\ph$ and a constant $C$ such that 
\begin{equation}
 \norm{\xi}_{\cC^\alpha} 
 \leqs C \sup_{k\geqs0} \sup_{x\in\Lambda_k} 2^{k\alpha}
 \bigabs{\pscal{\xi}{\cS^{2^{-k}}_x\ph}}\;.
\end{equation} 
See for instance~\cite[Section12]{Caravenna_Zambotti_20}.
Bounding the suprema by sums, and taking the $p$th power, we arrive at 
\begin{equation}
 (\norm{\xi}_{\cC^\alpha})^p 
 \leqs C^p \sum_{k\geqs0} \sum_{x\in\Lambda_k} 2^{k\alpha p}
 \bigabs{\pscal{\xi}{\cS^{2^{-k}}_x\ph}}^p\;.
\label{eq:proof_Holder_xi01} 
\end{equation} 
Since $\pscal{\xi}{\cS^{2^{-k}}_x\ph}$ belongs to the first Wiener chaos, 
its $p$th power belongs to the $p$th (inhomogeneous) chaos. 
Using the scaling property~\eqref{eq:xi_scaling} and~\eqref{eq:scaling_op}, 
we get 
\begin{equation}
 \bigexpec{\pscal{\xi}{\cS^{2^{-k}}_x\ph}^2} 
 = \frac{1}{\lambda^d} \bigexpec{\pscal{\xi}{\phi}^2}
 = \frac{1}{\lambda^d} \norm{\ph}_\H^2\;.
\end{equation} 
The equivalence of moments~\eqref{eq:equivalence_moments_infdim} implies 
\begin{equation}
 \bigexpec{\pscal{\xi}{\cS^{2^{-k}}_x\ph}^p}
 \leqs \frac{C_p}{\lambda^{dp/2}} \norm{\ph}_\H^p
\end{equation} 
for some constant $C_p$ depending only on $p$. Plugging this 
into the expectation of~\eqref{eq:proof_Holder_xi01} with $\lambda = 2^{-k}$ 
and using the fact that $\Lambda_k$ has $2^{kd}$ points, we obtain 
\begin{equation}
 \bigexpec{(\norm{\xi}_{\cC^\alpha})^p}
 \leqs C'_p \sum_{k\geqs0} 2^{kd} 2^{k\alpha p} 2^{kdp/2}
 = C'_p \sum_{k\geq0} 2^{k(d + \alpha p + dp/2)}
\end{equation} 
where $C'_p$ depends only on $p$. 
If $\alpha p < -d(1 + \frac{p}{2})$, the geometric series can be summed. 
It then follows by a version of Kolmogorov's continuity theorem (see 
Theorem~\ref{thm:Kolmogorov} below) that 
there exists a version of $\xi$ with bounded $\cC^\alpha$ norm.
Since $p$ can be taken arbitrarily large, the condition reduces to 
$\alpha < -\smash{\frac{d}{2}}$.
\end{proof}

\begin{remark}[Besov spaces]
Fractional Sobolev and H\"older--Besov spaces are particular instances 
of a general class of functional spaces called \emph{Besov spaces}. 
The Besov space $\cB^\alpha_{p,q}$ is a Banach space for all $\alpha\in\R$, 
and all $p,q\in[1,\infty]$, where $\alpha$ measures regularity, and 
$p$ and $q$ measure integrability. Sobolev and H\"older--Besov spaces 
correspond to the particular cases 
\begin{equation}
 H^s = \cB^s_{2,2}
 \qquad \text{and} \qquad 
 \cC^\alpha = \cB^\alpha_{\infty,\infty}\;.
\end{equation} 
\end{remark}

%%%%%%%%%%%%%%%%%%%%%%%%%%%%%%%%%%%%%%%%%%%%%%%%%%%%%%%%%%%%%%%%%%%%%%%%%%%%%%%%

\section{The Gaussian free field}
\label{sec:GFF}

%%%%%%%%%%%%%%%%%%%%%%%%%%%%%%%%%%%%%%%%%%%%%%%%%%%%%%%%%%%%%%%%%%%%%%%%%%%%%%%%

\subsection{Definition and basic properties}
\label{ssec:GFF_prop} 

Consider now the case where $h$ is given by 
\begin{equation}
 \hat h(k) = \frac{1}{\sqrt{\lambda_k}}\;,
 \qquad 
 k \in\Z^d\;,
\end{equation} 
where $\lambda_k$ is defined in~\eqref{eq:def_lambdak}.
The associated Gaussian field is 
\begin{equation}
\label{eq:def_GFF} 
 \GFF(x) = \sum_{k\in\Z^d} \frac{X_k}{\sqrt{\lambda_k}}e_k(x)\;.
\end{equation} 
This time, we have 
\begin{equation}
\label{eq:GFF_L2} 
 \norm{h}_\H^2 = \sum_{k\in\Z^d} \frac{1}{\lambda_k}
 \asymp \int_{\R^d} \frac{1}{1+\norm{y}^2} \6y 
 \asymp \int_1^\infty \frac{r^{d-1}\6r}{r^2} 
 = \int_1^\infty \frac{\6r}{r^{3-d}}\;,
\end{equation} 
which converges if $d < 2$. We conclude that the Gaussian field~\eqref{eq:def_GFF}
has a finite variance in dimension $d = 1$, but not in higher dimension. Still,
this is somewhat better than for white noise. Its covariance is given by 
\begin{align}
 \bigexpec{\GFF(x)\GFF(y)}
 &= \sum_{k,\ell \in \Z^d}
 \frac{\expec{X_k X_\ell}}{\sqrt{\lambda_k\lambda_\ell}}
 e_k(x) e_\ell(y) \\
 &= \sum_{k\in\Z^d} \frac{e_k(x)e_{-k}(y)}{\lambda_k} \\
 &= \sum_{k\in\Z^d} \frac{e_k(x-y)}{\lambda_k}
 =: G(x-y) 
\label{eq:def_Green} 
\end{align}
by Remark~\ref{rem:Fourier_complex} and the fact that $e_k(x)e_{-k}(y) 
= e_k(x-y)$. 
Note that while $G(0)$ is defined only for $d=1$, one can show that 
$G(x)$ is defined for all $d$ if $x\neq0$. 
The function $G$ has the following property, which states that it 
can be considered as the inverse of the linear operator $(\id - \Delta)$.

\begin{lemma}
For any $g\in\H$, the function $f$ defined by 
\begin{equation}
 f(x) = \int_\Lambda G(x-y) g(y) \6y\;,
\end{equation} 
if it exists, satisfies 
\begin{equation}
 (\id - \Delta) f(x) = g(x)\;.
\end{equation} 
\end{lemma}
\begin{proof}
Assuming the integral is well-defined, we have 
\begin{align}
 (\id - \Delta) f(x) 
 &= \int_\Lambda \sum_{k\in\Z^d} \frac{(\id-\Delta)e_k(x)e_{-k}(y)}{\lambda_k}
 g(y)\6y \\
 &= \sum_{k\in\Z^d} \int_\Lambda e_k(x)e_{-k}(y) g(y) \6y \\
 &= \sum_{k\in\Z^d} \hat g(k) e_k(x) 
 = g(x)
\end{align}
by Dirichlet's theorem on Fourier series. 
\end{proof}

This property motivates the following definition. 

\begin{definition}[Gaussian free field and Green function]
\begin{itemize}
\item   The function $G$ defined by~\eqref{eq:def_Green} is called 
the \emph{Green function} associated with the operator $(\id - \Delta)$.
It is also written $G = (\id - \Delta)^{-1}$. 

\item   The Gaussian field defined by~\eqref{eq:def_GFF} is called 
the \emph{Gaussian free field (GFF) with covariance $(\id-\Delta)^{-1}$}.
\end{itemize}
\end{definition}

\begin{exercise}
Show that for any test functions $\ph_1,\ph_2\in\H$, one has 
\begin{equation}
 \bigexpec{\pscal{\GFF}{\ph_1}\pscal{\GFF}{\ph_2}}
 = \pscal{\ph_1}{(\id-\Delta)^{-1}\ph_2}_\H\;.
\end{equation} 
\end{exercise}

We can now easily adapt the proof of Proposition~\ref{prop:xi_Sobolev} to 
the GFF~\eqref{eq:def_GFF}.

\begin{proposition}[Sobolev regularity of the Gaussian free field]
\label{prop:GFF_Sobolev} 
The Gaussian free field $\GFF$ belongs to $H^s$ for any $s < 1-\frac{d}{2}$, in the 
sense that 
\begin{equation}
 \bigexpec{\norm{\GFF}_{H^s}^2} < \infty 
 \qquad 
 \text{for all $s < 1-\dfrac{d}{2}$}\;.
\end{equation} 
\end{proposition}
\begin{proof}
A similar computation as before shows that 
\begin{equation}
 \bigexpec{\norm{\GFF}_{H^s}^2}
 = \sum_{k\in\Z^d} \frac{1}{\lambda_k^{1-s}}
 \asymp \int_{\R^d} \frac{\6y}{(1+\norm{y}^2)^{1-s}} 
 \asymp \int_1^\infty \frac{\6r}{r^{3-d-2s}}\;.
\end{equation} 
This is finite if and only if $s < 1-\frac{d}{2}$. 
\end{proof}

For $d=1$, we obtain a bit more regularity than $L^2 = H^0$, that we got in 
the estimate~\eqref{eq:GFF_L2}. For dimensions $d\geqs2$, on the other hand, 
we find again that the GFF is less regular than a function. We will examine 
these two cases more closely in the next two sections. 

%%%%%%%%%%%%%%%%%%%%%%%%%%%%%%%%%%%%%%%%%%%%%%%%%%%%%%%%%%%%%%%%%%%%%%%%%%%%%%%%

\subsection{The Gaussian free field on the circle $\T^1$}
\label{ssec:GFF1} 

Proposition~\ref{prop:GFF_Sobolev} shows that the one-dimensional GFF 
is a function that has better regularity properties than being merely 
square-integrable. In fact, one can show that it enjoys some H\"older 
regularity. Recall the definition of classical H\"older spaces.

\begin{definition}[H\"older spaces of regularity $\alpha\in(0,1)$]
For $0 < \alpha < 1$, the space $\cC^\alpha(\Lambda)$ consists in all functions 
$f:\Lambda\to\R$ such that 
\begin{equation}
\label{eq:Holder_positive} 
 \norm{f}_{\cC^\alpha} 
 = \sup_{x\in\Lambda} \abs{f(x)}
 +\sup_{\substack{x,y\in\Lambda \\x\neq y}}
 \frac{\abs{f(x)-f(y)}}{\norm{x-y}^\alpha} < \infty\;.
\end{equation} 
\end{definition}

A classical result allowing the obtain H\"older regularity of a 
stochastic process is due to Kolmogorov.

\begin{theorem}[Kolmogorov continuity criterion]
\label{thm:Kolmogorov} 
Let $(\phi(x))_{x\in[0,L]}$ be a stochastic process such that 
\begin{equation}
 \bigexpec{\norm{\phi(y) - \phi(x)}^\mu} 
 \leqs C \abs{y-x}^{1+\nu}
\end{equation} 
for all $x,y\in[0,L]$, for some constants $\mu, \nu>0$. Then 
there exists a modification of the process $(\phi(x))_{x\in[0,L]}$ whose 
paths belong to $\cC^\alpha$ for all $\alpha < \frac{\nu}{\mu}$.
\end{theorem}

We then have the following result, which shows that the one-dimensional 
GFF has the same regularity as Brownian motion. Its proof can be seen 
as a much easier relative of the proof of Proposition~\ref{prop:xi_Holder}. 

\begin{proposition}[H\"older regularity of the GFF on $\T^1$]
\label{prop:GFF_Holder} 
The GFF on the circle belongs to $\cC^\alpha(\T^1)$ for all $\alpha < \frac12$.  
\end{proposition}
\begin{proof}
For $x,y\in\T^1$, we have 
\begin{align}
 \bigexpec{\bigpar{\GFF(y) - \GFF(x)}^2}
 &= \sum_{k,\ell\in\Z} \frac{\expec{X_kX_\ell}}{\sqrt{\lambda_k\lambda_\ell}}
 \bigpar{e_k(y) - e_k(x)} \bigpar{e_\ell(y) - e_\ell(x)} \\
 &= \sum_{k\in\Z} \frac{1}{\lambda_k} \abs{e_k(y) - e_k(x)}^2\;.
\end{align} 
The trigonometric identity $\abs{\e^{\icx\theta}-1}^2 = 4\sin^2\Bigpar{\frac\theta2}$
yields 
\begin{equation}
 \abs{e_k(y) - e_k(x)}^2
 = \abs{\e^{2\pi\icx k (y-x)} - 1}^2
 = 4\sin^2(2\pi k(y-x))
 \leqs 4\bigbrak{(2\pi k)^2(y-x)^2\wedge 1}\;,
\end{equation} 
where $a\wedge b = \min\set{a,b}$. 
Therefore, if $\abs{y-x}\leqs1$, one has 
\begin{equation}
 \bigexpec{\bigpar{\GFF(y) - \GFF(x)}^2 }
 \lesssim \sum_{k\in\Z} \frac{(k^2(y-x)^2)\wedge 1}{1+k^2}
 \asymp \int_1^{1/\abs{y-x}} (y-x)^2\6r + \int_{1/\abs{y-x}}^\infty \frac{\6r}{r^2}
 \asymp \abs{y-x}\;.
\end{equation} 
On the other hand, if $\abs{y-x} > 1$, we can simply bound 
$\abs{e_k(y) - e_k(x)}^2$ by $4$. We conclude that there exists a constant 
$C_0 > 0$ such that 
\begin{equation}
 \bigexpec{\bigpar{\GFF(y) - \GFF(x)}^2 } \leqs C_0 \abs{y-x}
\end{equation} 
for all $x,y\in\T^1$. 
By the equivalence of moments bound~\eqref{eq:equivalence_moments_infdim}, 
it follows that for any $p>0$, there exists a constant $C(p)$ such that 
\begin{equation}
 \bigexpec{\bigpar{\GFF(y) - \GFF(x)}^{2p}}
 \leqs C(p) \abs{x-y}^p\;.
\end{equation} 
Kolmogorov's criterion thus applies with $\mu=2p$ and $\nu = p-1$, showing that 
$\GFF(x)$ is (up to a modification) H\"older continuous with exponent 
$\frac12(1-\frac1p)$. Since $p$ can be taken arbitrarily large, the result 
follows. 
\end{proof}

\begin{remark}[Link between $\alpha > 0$ and $\alpha < 0$]
The definition~\eqref{eq:Holder_positive} of the H\"older norm for 
positive $\alpha$ looks very different from the 
definition~\eqref{eq:Holder_negative}. However, one can show 
that it is equivalent to 
\begin{equation}
\label{eq:norm_C-alpha} 
 \norm{f}_{\cC^\alpha}
 = \sup_{x\in\Lambda} \sup_{\ph\in B_0} \sup_{\lambda\in(0,1]}
 \biggabs{\frac{\pscal{f-f(x)}{\cS^\lambda_x\ph}}{\lambda^\alpha}}\;.
\end{equation} 
\end{remark}

We can also easily estimate moments of the GFF. By translation invariance, 
these do not depend on the point $x$. The second moment is given by 
\begin{equation}
\label{eq:variance_GFF1} 
 \bigexpec{\GFF(x)^2}
 = \sum_{k\in\Z} \frac{1}{\lambda_k}
 = G(0)\;,
\end{equation}
which is finite for $d=1$, cf.~\eqref{eq:GFF_L2}. Odd moments of the GFF 
are equal to zero, while even moments behave as follows.

\begin{proposition}[Moments of the GFF on $\T^1$]
For any $p>1$, there exists a constant $C(p)$ such that 
\begin{equation}
\label{eq:moments_GFF} 
 \bigexpec{\GFF(x)^{2p}} \leqs C(p) \bigexpec{\GFF(x)^2}^p\;.
\end{equation} 
\end{proposition}
\begin{proof}
A direct application of the equivalence of moments 
bound~\eqref{eq:equivalence_moments_infdim} shows that~\eqref{eq:moments_GFF}
holds with $C(p) = (2p-1)^p$. 
We can actually get a slightly better bound by a direct computation. 
Indeed, we have 
\begin{equation}
\label{eq:expec_GFF2p} 
 \bigexpec{\GFF(x)^{2p}} 
 = \sum_{k_1,\dots,k_{2p}\in\Z} \frac{\expec{X_{k_1}\dots X_{k_{2p}}}}
 {\sqrt{\lambda_{k_1}\dots\lambda_{k_{2p}}}}\;.
\end{equation} 
By Isserlis' theorem (Theorem~\ref{thm:Isserlis}), the expectation vanishes 
unless the $k_i$ are pairwise equal, in which case it has value $1$. Since 
there are $(2p-1)!!$ pairwise matchings, we get 
\begin{equation}
 \bigexpec{\GFF(x)^{2p}} 
 = (2p-1)!! \sum_{k_1,\dots,k_p\in\Z}
 \frac{1}{\lambda_{k_1}\dots\lambda_{k_p}}
 = (2p-1)!! \bigexpec{\GFF(x)^2}^p\;.
\end{equation} 
We thus obtain~\eqref{eq:moments_GFF} with $C(p) = (2p-1)!!$. 
Using~\eqref{eq:double_factorial} and Stirling's formula, we 
find that $(2p-1)!!$ behaves asymptotically like 
$\sqrt{2}(2p)^p\e^{-p}$, which is slightly better that 
$(2p-1)^p$. 
\end{proof}

%%%%%%%%%%%%%%%%%%%%%%%%%%%%%%%%%%%%%%%%%%%%%%%%%%%%%%%%%%%%%%%%%%%%%%%%%%%%%%%%

\subsection{The Gaussian free field on the torus $\T^2$}
\label{ssec:GFF2} 

We have seen that the GFF on the two-dimensional torus has infinite variance, 
and belongs only to Sobolev spaces $H^s$ with $s < 0$. By an argument similar 
to the one used in Proposition~\ref{prop:xi_Holder}, one can also show 
that it belongs to the Besov--H\"older spaces $\cC^\alpha$ with $\alpha < 0$. 
For this reason, it is not possible to define powers of $\GFF$. 

To circumvent this difficulty, we can use the idea introduced at the 
beginning of Section~\ref{ssec:white_noise_prop}, which is to work with 
a cut-off $N$. 

\begin{definition}[Truncated two-dimensional Gaussian free field]
For $N\geqs1$, let $\cK_N = \setsuch{k\in\Z^2}{\abs{k} \leqs N}$,
where $\abs{k} = \abs{k_1} + \abs{k_2}$. 
The \emph{truncated two-dimensional Gaussian free field (GFF)} with covariance 
$(\id-\Delta_N)^{-1}$ on $\Lambda$ is defined as 
\begin{equation}
 \GFFN(x) := \sum_{k\in\cK_N}
 \frac{X_k}{\sqrt{\lambda_k}} e_k(x)\;.
\end{equation} 
Here $\Delta_N$ is the restriction of $\Delta$ to the subspace $E_N$ of $\H$
spanned by Fourier basis functions $e_k$ with $\abs{k} \leqs N$. 
\end{definition}

The name is justified by the fact that the same computation as 
in~\eqref{eq:def_Green} gives 
\begin{equation}
\bigexpec{\GFFN(x)\GFFN(y)} 
 = \sum_{k\in\cK_N} \frac{1}{\lambda_k} e_k(x-y)
 =: G_N(x-y)\;,
\label{eq:cov_GFFN} 
\end{equation}
where $G_N$ is the Green function of $(\id-\Delta_N)$. 
In particular, the variance at any $x$ is given by 
\begin{equation}
C_N := \bigexpec{\GFFN(x)^2} 
 = G_N(0)
 = \sum_{k\in\cK_N} \frac{1}{\lambda_k}
 = \Tr\bigbrak{(\id-\Delta_N)^{-1}}\;.
\label{eq:def_CN} 
\end{equation}
It is not hard to see that $C_N$ diverges like $\log(N)$ as $N\to\infty$. 

\begin{exercise}
\label{exo:CN_2D} 
Show that 
$C_N = \dfrac{\log N}{2\pi} + \Order{1}$ 
as $N\to\infty$, by viewing~\eqref{eq:def_CN} as a Riemann sum. 
% 
% \noindent
% \Hint View~\eqref{eq:def_CN} as a Riemann sum and integrate 
% using polar coordinates. 
\end{exercise}

Since the variance $C_N$ does not depend on $x$, we can define for any $n\geqs1$
the quantity
\begin{equation}
 \Wick{\GFFN^n(x)} =
 \Wick{\GFFN^n(x)}_{C_N} := H_n(\GFFN(x);C_N)\;,
\end{equation} 
called \emph{$n$th Wick power} of the truncated GFF, 
where $H_n(\phi;C_N)$ denotes the scaled Hermite polynomial 
introduced in~\eqref{eq:Hermite_scaling}. For every $x\in\Lambda$,
$\Wick{\GFFN^n(x)}$ is a random variable belonging to the $n$th 
homogeneous Wiener chaos $\cH_n$. The function $\Wick{\GFF^n}$ is 
also an $\H$-valued random variable. These random variables 
are independent for different $n$, while their variance is 
bounded uniformly in the cut-off $N$, as shows the following 
very useful result.

\begin{proposition}[Uniform bound on the variance of Wick powers]
\label{prop:variance_Wick_GFF2} 
For every $n\geqs1$, 
\begin{equation}
 \sup_{N\geqs1} 
 \biggexpec{\biggpar{\int_\Lambda \Wick{\GFFN^n(x)} \6x}^2}
 < \infty\;.
\end{equation} 
\end{proposition}

To prove this result, we will need the following inequality.

\begin{lemma}[Young-type inequality]
\label{lem:Young} 
Fix integers $d > n,m > 0$ such that $n+m > d$. Then there 
exists a constant $C>0$, independent of $k$, such that 
\begin{equation}
 \sum_{\substack{k_1,k_2\in\Z^d\setminus\set{0}\\k_1+k_2 = k}}
 \frac{1}{\norm{k_1}^n\norm{k_2}^m}
 \leqs \frac{C}{\norm{k}^{n+m-d}}
\end{equation} 
for all $k\in\Z^d$. 
\end{lemma}
\begin{proof}
We may restrict the sum to $(k_1,k_2)$ such that $\norm{k_1}\geqs\norm{k_2}$, and 
multiply the end result by $2$. Since $k_1 + k_2 = k$, we cannot have both 
$\norm{k_1} < \frac12\norm{k}$ and $\norm{k_2} < \frac12\norm{k}$. 
The half sum can thus be decomposed as 
\begin{equation}
S_1 + S_2
= \sum_{\substack{k_1,k_2\in\Z^d\setminus\set{0}\\k_1+k_2 = k\\\norm{k_2}\leqs \norm{k_1} \wedge \norm{k}/2}} \frac{1}{\norm{k_1}^n\norm{k_2}^m}
 +  \sum_{\substack{k_1,k_2\in\Z^d\setminus\set{0}\\k_1+k_2 = k\\\norm{k_1}\geqs 
 \norm{k_2}> \norm{k}/2}}
 \frac{1}{\norm{k_1}^n\norm{k_2}^m}\;.
\end{equation} 
Since $\norm{k_2} \leqs \frac12\norm{k}$ implies $\norm{k_1} \geqs \frac12\norm{k}$, 
we have 
\begin{equation}
 S_1 \leqs \frac{2^n}{\norm{k}^n}
 \sum_{\substack{k_2\in\Z^d\setminus\set{0}\\ \norm{k_2}\leqs\norm{k}/2}}
 \frac{1}{\norm{k_2}^m}
 \lesssim \frac{2^n}{\norm{k}^n}\int_1^{\norm{k}/2} \frac{r^{d-1}\6r}{r^m}
 \lesssim \frac{1}{\norm{k}^{n+m-d}}\;.
\end{equation} 
As for $S_2$, it satisfies
\begin{equation}
 S_2 \leqs 
 \sum_{\substack{k_2\in\Z^d\setminus\set{0}\\ \norm{k_2}>\norm{k}/2}}
 \frac{1}{\norm{k_2}^{n+m}}
 \lesssim \int_{\norm{k}/2}^\infty \frac{r^{d-1}\6r}{r^{n+m}}
 \lesssim \frac{1}{\norm{k}^{n+m-d}}\;,
\end{equation} 
which yields the result. 
\end{proof}

\begin{proof}[{\sc Proof of Proposition~\ref{prop:variance_Wick_GFF2}}]
We have 
\begin{align}
 \biggexpec{\biggpar{\int_\Lambda \Wick{\GFFN^n(x)} \6x}^2} 
 &= \int_\Lambda \int_\Lambda \bigexpec{\Wick{\GFFN^n(x)}\Wick{\GFFN^n(y)}}
 \6x\6y \\
 &= n! \int_\Lambda \int_\Lambda \bigexpec{\GFF(x)\GFF(y)}^n \6x\6y \\
 &= n! \int_\Lambda \int_\Lambda 
 \biggpar{\sum_{k\in\cK_N}\frac{1}{\lambda_k}e_k(x-y)}^n 
 \6x\6y \\
 &= n! \sum_{k_1,\dots,k_n\in\cK_N} \frac{1}{\lambda_{k_1}\dots\lambda_{k_n}}
 \biggabs{\int_\Lambda e_{k_1}(x)\dots e_{k_n}(x)\6x}^2\\
 &= n! \sum_{\substack{k_1,\dots,k_n\in\cK_N\\k_1 + \dots + k_n = 0}} 
 \frac{1}{\lambda_{k_1}\dots\lambda_{k_n}}\;,
 \label{eq:proof_moments01} 
\end{align}
where we have used Proposition~\ref{prop:Hermite_orthogonal} to get the second 
line, and~\eqref{eq:cov_GFFN} to get the third line. 
For $k\in\Z^d$ and $n\geqs2$, let 
\begin{equation}
 S(n,k) = \sum_{\substack{k_1,\dots,k_n\in\Z^d\\k_1 + \dots + k_n = k}} 
 \frac{1}{\lambda_{k_1}\dots\lambda_{k_n}}\;.
\end{equation} 
Then Lemma~\ref{lem:Young}, together with an index shift, shows that 
$S(2,k) \leqs C(2)\lambda_k^{-1}$for some constant $C(2)$, and by induction one gets 
\begin{equation}
 S(n,k) = \sum_{k_1\in\Z^d} \frac{1}{\lambda_{k_1}} 
 S(n-1,k-k_1)
 \leqs \frac{C(n)}{\lambda_k}
\end{equation} 
for some finite $C(n)$. In particular, $S(n,0)$ is bounded. 
Since this provides an upper bound uniform in $N$ for~\eqref{eq:proof_moments01}, 
the result is proved. 
\end{proof}

By the equivalence of moments bound~\eqref{eq:equivalence_moments_infdim}, we also 
have
\begin{equation}
 \biggexpec{\biggpar{\int_\Lambda \Wick{\GFFN^n(x)} \6x}^{2p}}
 \leqs (2p-1)^{np} 
 \biggexpec{\biggpar{\int_\Lambda \Wick{\GFFN^n(x)} \6x}^2}^p\;,
\end{equation} 
which is bounded uniformly in the cut-off $N$ for all $p>1$. 
We thus conclude that all moments of Wick powers of the two-dimensional 
GFF are well-defined, when viewed as limits as $N\to\infty$ of the 
truncated GFF. Note that the same holds in dimensions $d>2$, the difference 
being that the variance $C_N$ computed in~\eqref{eq:def_CN} diverges 
like $N^{d-2}$ instead of $\log(N)$. This faster divergence causes new 
problems for non-linear fields, as we will see in Section~\ref{sec:phi43}.

%%%%%%%%%%%%%%%%%%%%%%%%%%%%%%%%%%%%%%%%%%%%%%%%%%%%%%%%%%%%%%%%%%%%%%%%%%%%%%%%

\chapter{The $\Phi^4$ model}
\label{chap:phi4} 

The $\Phi^4$ model on the $d$-dimensional torus is arguably the simplest 
example of non-linear field theory. It originated as a toy model in 
Euclidean Quantum Field Theory, and its behaviour is very different depending 
on the dimension $d$. 

%%%%%%%%%%%%%%%%%%%%%%%%%%%%%%%%%%%%%%%%%%%%%%%%%%%%%%%%%%%%%%%%%%%%%%%%%%%%%%%%

\section{Definition of the model}
\label{sec:def_phi4} 

Let $\Lambda = \T^d$ be the $d$-dimensional torus for some $d\geqs1$. 
Given constants $\alpha \geqs 0$, $m>0$ and a function $\phi: \Lambda\to\R$, 
we define its energy 
\begin{equation}
\label{eq:phi4_H} 
 \cH_{d,\alpha}(\phi) 
 = \int_{\Lambda} \biggbrak{\norm{\nabla\phi(x)}^2 + \frac{m^2}2 \phi(x)^2 
 + \alpha\phi(x)^4} \6x\;.
\end{equation} 
The name \emph{$\Phi^4$ model} is due to the term $\phi(x)^4$ in the integral. 
It is called the $\Phi^4_d$ model if we want to emphasize the value of the 
dimension. The parameter $m$ has the physical interpretation of a mass. 
We will take it equal to $1$ in what follows. 

The $\Phi^4_d$ measure is the probability measure $\mu_{d,\alpha}$ on 
$\H = L^2(\Lambda,\6x)$ (or on a suitable a space of functions $\phi:\Lambda\to\R$), 
formally defined by 
\begin{equation}
 \mu_{d,\alpha} \sim \frac{1}{\cZ_{d,\alpha}} \e^{-\cH_{d,\alpha}(\phi)} \6\phi\;,
\end{equation} 
where $\cZ_{d,\alpha}$ is the normalisation. This is called a \emph{Gibbs measure}, 
and $\cZ_{d,\alpha}$ is known in statistical physics as the \emph{partition function}. 

As such, this definition does not make sense, since there is no such thing as
Lebesgue measure on $\H$. Note however that for $\alpha = 0$, we can 
integrate by parts, to obtain 
\begin{align}
 \cH_{d,0}(\phi)
 &= \frac12 \int_{\Lambda} \bigbrak{-\Delta \phi(x) \phi(x) + \phi(x)^2} \6x \\
 &= \frac12 \pscal{\phi}{[\id-\Delta]\phi}_\H\;.
\end{align} 
In view of the expression~\eqref{eq:Gaussian_ndim} of the density of 
a finite-dimensional Gaussian measure, we can interpret 
$\mu_{d,0}(\6\phi)$ as the law of a centred Gaussian field with 
covariance $(\id-\Delta)^{-1}$ that we have studied in Section~\ref{sec:GFF}. 
This means that for a random variable $F:\H\to\R$, we can try to define 
its expectation under $\mu_{d,0}$ as 
\begin{equation}
 \expecin{\mu_{d,0}}{F} = 
 \biggexpec{F\biggpar{\sum_{k\in\Z^d}\frac{X_k}{\sqrt{\lambda_k}}e_k}}\;,
\end{equation} 
where the $X_i$ are independent, identically distributed standard 
Gaussians, and the $\lambda_k = 1 + (2\pi)^d\norm{k}^2$ are eigenvalues 
of the operator $\id - \Delta$. 

\begin{example}
\begin{itemize}
\item  Let $F(\phi) = \norm{\phi}_{L^2}^2$. Then 
\begin{equation}
\label{eq:expec_mud0} 
 \bigexpecin{\mu_{d,0}}{ \norm{\phi}_{L^2}^2}
 = \biggexpec{\sum_{k\in\Z^d}\frac{X_k^2}{\lambda_k}}
 = \sum_{k\in\Z^d}\frac{1}{\lambda_k}\;,
\end{equation} 
which we have seen is the variance of the Gaussian free field. 
This is finite for $d=1$ (cf.~\eqref{eq:variance_GFF1}), but not for $d\geqs2$. 
By translation invariance, \eqref{eq:expec_mud0} is also the expectation 
of $\phi(x)^2$ for any $x\in\Lambda$. 

\item   Let $F(\phi) = \phi(x)\phi(y)$ for $x,y\in\Lambda$. 
Then we have seen in~\eqref{eq:def_Green} that 
\begin{equation}
\label{eq:covariance_phi41} 
 \bigexpecin{\mu_{d,0}}{\phi(x)\phi(y)}
 = G(x-y)
\end{equation} 
is the Green function, which is defined for all $x,y$ if $d=1$, and 
for all $d$ if $x\neq y$. 
\end{itemize}
\end{example}

This suggests using $\mu_{d,0}$ as a reference measure, instead of the 
non-existent Lebesgue measure on $\H$, and to define expectations 
under $\mu_{d,\alpha}$ with $\alpha>0$ by 
\begin{equation}
\label{eq:expec_mudalpha} 
 \expecin{\mu_{d,\alpha}}{F} = \frac{\cZ_{d,0}}{\cZ_{d,\alpha}} 
 \biggexpecin{\mu_{d,0}}{F(\phi) 
 \exp\biggset{-\alpha\int_{\Lambda}\phi(x)^4\6x}}\;.
\end{equation} 
In particular, taking $F = 1$, we obtain 
\begin{equation}
\label{eq:Zratio} 
 \frac{\cZ_{d,\alpha}}{\cZ_{d,0}} = 
 \biggexpecin{\mu_{d,0}}{\exp\biggset{-\alpha\int_{\Lambda}\phi(x)^4\6x}}\;.
\end{equation} 
If we manage to compute this ratio, at least perturbatively for small 
$\alpha$, it gives us access to more general expectations of the 
form~\eqref{eq:expec_mudalpha}. 

One example of random variables $F$ that are physically relevant is 
$F = \phi(x_1)\dots\phi(x_n)$, for given $x_1,\dots,x_n\in\Lambda$, 
leading to so-called $n$-point functions 
\begin{equation}
G_{n,d,\alpha}(x_1,\dots,x_n) 
 = \bigexpecin{\mu_{d,\alpha}}{\phi(x_1)\dots\phi(x_n)}\;.
\end{equation} 
But one can also consider $F$ depending on an integral involving $\phi$, 
such as the energy itself. 

\begin{exercise}
Compute the $n$-point function $G_{n,d,0}(x_1,\dots,x_n)$ in the case 
$\alpha = 0$. 
\end{exercise}

In what follows, we will focus on the ratio~\eqref{eq:Zratio} of
partition functions, since this is the first step in computing expectations
of more general random variables. 

%%%%%%%%%%%%%%%%%%%%%%%%%%%%%%%%%%%%%%%%%%%%%%%%%%%%%%%%%%%%%%%%%%%%%%%%%%%%%%%%

\section{The $\Phi^4_1$ model}
\label{sec:phi41} 

In this section, we consider the $\Phi^4$ model on the one-dimensional 
torus $\Lambda = \T^1 = \T$, that is, the circle. As we have seen, all moments of the GFF are 
well-defined in this case. The ratio of partition functions \eqref{eq:Zratio} becomes 
\begin{equation}
\label{eq:Zratio_1} 
 \frac{\cZ_{1,\alpha}}{\cZ_{1,0}} = 
 \biggexpecin{\mu_{1,0}}{\exp\biggset{-\alpha\int_{\Lambda}\phi(x)^4\6x}}\;.
\end{equation} 
One way to proceed is to expand the exponential, leading to 
\begin{equation}
\label{eq:sum_phi41} 
 \frac{\cZ_{1,\alpha}}{\cZ_{1,0}} 
 \asymp \sum_{n\geqs0} \frac{(-\alpha)^n}{n!}
 \biggexpecin{\mu_{1,0}}{\biggpar{\int_{\Lambda}\phi(x)^4\6x}^n}\;,
\end{equation} 
where we use the symbol $\asymp$ because we do not know if this series 
is convergent (in fact, one can show that it is not!). 

\begin{figure}
\begin{center}
 \begin{tikzpicture}[>=stealth',main node/.style={circle,minimum
size=0.25cm,inner sep=1pt,fill=blue!25,draw},scale=1.2]
\node[main node] (1) at (0,0) {1};
\node[main node] (2) at (1,0) {2};
\node[main node] (3) at (2,0) {3};
\node[main node] (4) at (3,0) {4};
\node[main node] (5) at (0,-1) {5};
\node[main node] (6) at (1,-1) {6};
\node[main node] (7) at (2,-1) {7};
\node[main node] (8) at (3,-1) {8};
\draw[thick,violet,->] (1) --node[above]{$k_1$} (2);  
\draw[thick,violet,->] (3) --node[above]{$k_3$} (4);  
\draw[thick,violet,->] (5) --node[above]{$k_5$} (6);  
\draw[thick,violet,->] (7) --node[above]{$k_8$} (8);  

\node[main node] (1) at (6,0) {1};
\node[main node] (2) at (7,0) {2};
\node[main node] (3) at (8,0) {3};
\node[main node] (4) at (9,0) {4};
\node[main node] (5) at (6,-1) {5};
\node[main node] (6) at (7,-1) {6};
\node[main node] (7) at (8,-1) {7};
\node[main node] (8) at (9,-1) {8};
\draw[thick,violet,->] (1) --node[left]{$k_1$} (5);  
\draw[thick,violet,->] (2) --node[above]{$k_2$} (3);  
\draw[thick,violet,->] (6) --node[above]{$k_6$} (7);  
\draw[thick,violet,->] (8) --node[right]{$k_1$} (4);  

\end{tikzpicture}
\end{center}
\vspace{-4mm}
\caption[]{Two pairwise matchings contributing to the 
sum~\eqref{eq:matchings_phi41}. An oriented arrow from vertex 
$i$ to vertex $j$ means that $k_j = -k_i$.}
\label{fig:pairings_phi41} 
\end{figure}
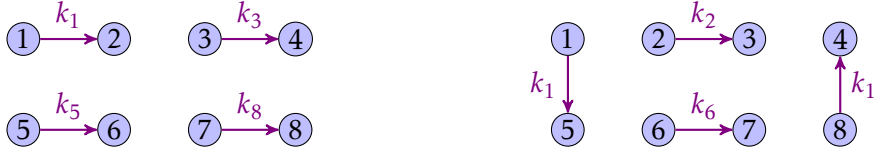

The term $n=1$ in the sum~\eqref{eq:sum_phi41} can be computed 
using~\eqref{eq:expec_GFF2p}. By Isserlis' theorem, we have 
\begin{equation}
 \bigexpecin{\mu_{1,0}}{\phi(x)^4}
 = \sum_{k_1,k_2,k_3,k_4\in\Z} 
 \frac{\expec{X_{k_1}X_{k_2}X_{k_3}X_{k_4}}}
 {\sqrt{\lambda_{k_1}\lambda_{k_2}\lambda_{k_3}\lambda_{k_4}}}
 = 3 \sum_{k_1,k_2\in\Z} \frac{1}{\lambda_{k_1}\lambda_{k_2}}
 = 3G(0)^2\;,
\end{equation} 
where the factor $3$ counts the number of pairwise matchings of 
the four indices. Therefore, we also have 
\begin{equation}
 \biggexpecin{\mu_{1,0}}{\biggpar{\int_{\Lambda}\phi(x)^4\6x}}
 = 3G(0)^2\;. 
\end{equation} 
For the term $n=2$, we find 
\begin{align}
 \biggexpecin{\mu_{1,0}}{\biggpar{\int_{\Lambda}\phi(x)^4\6x}^2}
 &= \biggexpecin{\mu_{1,0}}{\int_{\Lambda}\phi(x)^4\6x\int_{\Lambda}\phi(y)^4\6y} \\
 &= \sum_{k_1,\dots,k_8\in\Z} 
 \frac{\expec{X_{k_1}\dots X_{k_8}}}{\sqrt{\lambda_{k_1}\dots \lambda_{k_8}}} 
 \int_\Lambda \int_\Lambda e_{k_1}(x)\dots e_{k_4}(x) e_{k_5}(y)\dots e_{k_8}(y)
 \6x\6y \\
 &= \sum_{\substack{k_1,\dots,k_8\in\Z\\k_1+k_2+k_3+k_4=0\\k_5+k_6+k_7+k_8=0}} 
 \frac{\expec{X_{k_1}\dots X_{k_8}}}{\sqrt{\lambda_{k_1}\dots \lambda_{k_8}}}\;.
 \label{eq:matchings_phi41} 
\end{align}
The combinatorics is now more complicated, since we have to sum over all 
pairwise matchings of the $k_i$ that satisfy the two sum constraints. 
Figure~\ref{fig:pairings_phi41} gives two examples of such pairings.

%%%%%%%%%%%%%%%%%%%%%%%%%%%%%%%%%%%%%%%%%%%%%%%%%%%%%%%%%%%%%%%%%%%%%%%%%%%%%%%%

\subsection{Wick calculus and Feynman diagrams}
\label{ssec:phi41_wick} 

So far, we have not used the power of the Wiener chaos decomposition. One way 
to do this is to modify the energy~\eqref{eq:phi4_H} to
\begin{equation}
\label{eq:phi4_H_Wick} 
 \cH_{1,\alpha}^{\text{Wick}}(\phi) 
 = \int_{\Lambda} \biggbrak{\norm{\nabla\phi(x)}^2 + \frac12 \phi(x)^2 
 + \alpha\Wick{\phi(x)^4}} \6x\;.
\end{equation} 
This is now a different model, since we have replaced the fourth 
power $\phi(x)^4$ by the fourth Wick power 
$\Wick{\phi(x)^4} = H_4(\phi(x)^4;C)$, where $C$ should be taken equal 
to the covariance $G(0)$. One could transform this into the original 
model by adding a suitable multiple of the second Wick power 
$\Wick{\phi(x)^2}$ and a suitable constant, but we will not explore 
this further here, and work with the new energy. The ratio of partition 
functions~\eqref{eq:Zratio_1} now becomes
\begin{equation}
\label{eq:Zratio_1_Wick} 
 \frac{\cZ_{1,\alpha}}{\cZ_{1,0}} = 
 \biggexpecin{\mu_{1,0}}{\exp\biggset{-\alpha\int_{\Lambda}\Wick{\phi(x)^4}\6x}}
 \asymp \sum_{n\geqs0} \frac{(-\alpha)^n}{n!}
 \biggexpecin{\mu_{1,0}}{\biggpar{\int_{\Lambda}\Wick{\phi(x)^4}\6x}^n}\;.
\end{equation} 
The coefficient $n=1$ is simply
\begin{equation}
\label{eq:expec_phi4} 
 \biggexpecin{\mu_{1,0}}{\int_{\Lambda}\Wick{\phi(x)^4}\6x} = 0\;,
\end{equation} 
since Wick powers are centred. 
The coefficient $n=2$ is given by 
\begin{align}
 \biggexpecin{\mu_{1,0}}{\biggpar{\int_{\Lambda}\Wick{\phi(x)^4}\6x}^2}
 &= \int_{\Lambda}\int_{\Lambda} 
 \bigexpecin{\mu_{1,0}}{\Wick{\phi(x)^4}\Wick{\phi(y)^4}} \6x\6y \\
 &= 4! \int_{\Lambda}\int_{\Lambda} 
 \bigexpecin{\mu_{1,0}}{\phi(x)\phi(y)}^4 \6x\6y \\
 &= 4! \int_{\Lambda}\int_{\Lambda} G(x-y)^4 \6x\6y\;,
\label{eq:expec_phi4_squared} 
\end{align}
where we have used Proposition~\ref{prop:Hermite_orthogonal} 
and~\eqref{eq:covariance_phi41}.
Note that since $G$ is translation invariant, 
one can replace the double integral in~\eqref{eq:expec_phi4_squared} by 
a simple integral of $G(x)^4$, but this will not be important here. 
We can represent~\eqref{eq:expec_phi4_squared} graphically as 
\begin{equation}
 \biggexpecin{\mu_{1,0}}{\biggpar{\int_{\Lambda}\Wick{\phi(x)^4}\6x}^2}
 = 4!\Pi\Bigpar{\FGIV}\;,
\end{equation} 
where the two vertices of the graph indicate the two integration variables 
$x$ and $y$ in~\eqref{eq:expec_phi4_squared}, while the four edges indicate 
the four factors $G(x-y)$. The map $\Pi$ stands for evaluation of the integral. 
This is a first example of~\emph{Feynman diagram}. 

In order to generalise this computation to higher powers, we reformulate it 
in terms of the Wiener isometry. For given $x\in\Lambda$, considered 
as a parameter, define 
\begin{equation}
 h_x = \sum_{k\in\Z} \hat h_x(k) e_k\;, 
 \qquad\text{where }
 \hat h_x(k) := \frac{e_k(x)}{\sqrt{\lambda_k}}\;.
\end{equation} 
For every $x\in\Lambda$, $h_x$ is an element of $\H = L^2(\Lambda,\6x)$. 
Furthermore,
\begin{equation}
 \hat I_1(h_x) = \sum_{k\in\Z} \hat h_x(k) X_k 
 = \sum_{k\in\Z} \frac{X_k}{\sqrt{\lambda_k}} e_k(x) = \phi(x)
\end{equation} 
is the Gaussian free field. It follows that 
\begin{equation}
 \Wick{\phi(x)^4}
 = H_4(\phi(x);C) 
 = \hat I_4(h_x^{\otimes 4})\;,
\end{equation} 
so that Proposition~\ref{prop:mult_nm} yields 
\begin{align}
 \Wick{\phi(x)^4}\Wick{\phi(y)^4}
 &= \hat I_4(h_x^{\otimes 4}) \hat I_4(h_y^{\otimes 4}) \\
 &= \sum_{p=0}^4 \hat I_{8-2p} (h_x^{\otimes 4} \star_p h_y^{\otimes 4})\;.
\end{align}
Taking the expectation, we obtain
\begin{equation}
 \bigexpecin{\mu_{1,0}}{\Wick{\phi(x)^4}\Wick{\phi(y)^4}}
 = \hat I_0 (h_x^{\otimes 4} \star_4 h_y^{\otimes 4})\;,
\end{equation} 
and the definition~\eqref{eq:def_starp} of the contraction $f\star_p g$ yields 
\begin{align}
 \hat I_0 (h_x^{\otimes 4} \star_4 h_y^{\otimes 4}) 
 &= 4! \sum_{k_1,\dots,k_4\in\Z} 
 \overline{\hat h_x^{\otimes 4}(k_1,\dots,k_4)} \hat h_y^{\otimes 4}(k_1,\dots,k_4) \\
 &= 4! \biggpar{\sum_{k\in\Z} \overline{\hat h_x(k)} \hat h_y(k)}^4 \\
 &= 4! \biggpar{\sum_{k\in\Z} \frac{e_k(x-y)}{\lambda_k}}^4 \\
 &= 4! G(x-y)^4\;.
\end{align}
Here we have used complex conjugates in the inner products because we work with 
complex Fourier series. 
We thus recover~\eqref{eq:expec_phi4_squared} in a way that may seem more 
convoluted, but allows for generalisation to higher powers. 
Indeed, in the case $n = 3$ we obtain 
\begin{align}
\Wick{\phi(x)^4}\Wick{\phi(y)^4}\Wick{\phi(z)^4}
 &= \hat I_4(h_x^{\otimes 4}) \hat I_4(h_y^{\otimes 4}) \hat I_4(h_z^{\otimes 4}) \\
 &= \sum_{p=0}^4 \hat I_{8-2p} (h_x^{\otimes 4} \star_p h_y^{\otimes 4}) 
 \hat I_4(h_z^{\otimes 4}) \\
 &= \sum_{p=0}^4 \sum_{q=0}^{(8-2p)\wedge4} \hat I_{12-2p-2q} 
 \bigpar{(h_x^{\otimes 4} \star_p h_y^{\otimes 4}) \star_q h_z^{\otimes 4}}\;.
\end{align}
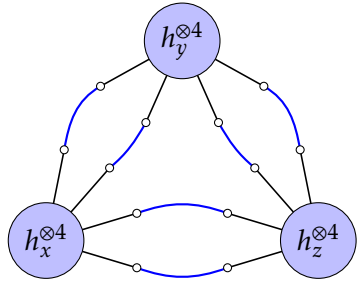
\begin{figure}
\begin{center}
\begin{tikzpicture}[>=stealth',main node/.style={circle,minimum
size=1cm,inner sep=2pt,fill=blue!25,draw},small node/.style={draw,circle,fill=white,minimum
size=3pt,inner sep=0pt},scale=1.2]
\draw[semithick] (0,0) -- (0.2,1);
\draw[semithick] (0,0) -- (0.7,0.8);
\draw[semithick] (0,0) -- (1,0.3);
\draw[semithick] (0,0) -- (1,-0.3);

\draw[semithick] (3,0) -- (2.8,1);
\draw[semithick] (3,0) -- (2.3,0.8);
\draw[semithick] (3,0) -- (2,0.3);
\draw[semithick] (3,0) -- (2,-0.3);

\draw[semithick] (1.5,2.2) -- (0.6,1.7);
\draw[semithick] (1.5,2.2) -- (1.1,1.3);
\draw[semithick] (1.5,2.2) -- (1.9,1.3);
\draw[semithick] (1.5,2.2) -- (2.4,1.7);

\draw[thick,blue] (0.2,1) to[out=80,in=-150] (0.6,1.7);
\draw[thick,blue] (0.7,0.8) to[out=40,in=-120] (1.1,1.3);
\draw[thick,blue] (2.3,0.8) to[out=140,in=-60] (1.9,1.3);
\draw[thick,blue] (2.8,1) to[out=100,in=-30] (2.4,1.7);
\draw[thick,blue] (1,0.3) to[out=20,in=160] (2,0.3);
\draw[thick,blue] (1,-0.3) to[out=-20,in=-160] (2,-0.3);

\node[main node] (x) at (0,0) {$h_x^{\otimes 4}$};
\node[main node] (y) at (1.5,2.2) {$h_y^{\otimes 4}$};
\node[main node] (z) at (3,0) {$h_z^{\otimes 4}$};

\node[small node] (1) at (0.2,1) {};
\node[small node] (2) at (0.7,0.8) {};
\node[small node] (3) at (1,0.3) {};
\node[small node] (4) at (1,-0.3) {};

\node[small node] (5) at (2.8,1) {};
\node[small node] (6) at (2.3,0.8) {};
\node[small node] (7) at (2,0.3) {};
\node[small node] (8) at (2,-0.3) {};

\node[small node] (9) at (0.6,1.7) {};
\node[small node] (10) at (1.1,1.3) {};
\node[small node] (11) at (2.4,1.7) {};
\node[small node] (12) at (1.9,1.3) {};
\end{tikzpicture}
\end{center}
\vspace{-4mm}
\caption[]{One of the pairings occurring in the 
expectation~\eqref{eq:expec_phi_xyz}.}
\label{fig:n=3} 
\end{figure}
When taking the expectation, only terms with $2p+2q = 12$ remain. There is 
actually only one option, which is to take $p = 2$ and $q = 4$, yielding 
\begin{equation}
\label{eq:expec_phi_xyz} 
 \bigexpecin{\mu_{1,0}}{\Wick{\phi(x)^4}\Wick{\phi(y)^4}\Wick{\phi(z)^4}}
 = \hat I_0\bigpar{(h_x^{\otimes 4} \star_2 h_y^{\otimes 4}) \star_4 h_z^{\otimes 4}}\;.
\end{equation} 
The contraction operations can be represented graphically, as we did in 
Section~\ref{ssec:Wick_product}. The functions $h_x^{\otimes 4}$, $h_y^{\otimes 4}$ 
and $h_z^{\otimes 4}$ are represented each by a vertex with four legs. 
The operation $\star_2$ corresponds to pairing two legs of $h_x^{\otimes 4}$
with two legs of $h_y^{\otimes 4}$, while the operation $\star_4$ 
corresponds to pairing the remaining four legs of $h_x^{\otimes 4}$ 
and $h_y^{\otimes 4}$ with the four legs of $h_z^{\otimes 4}$, see
Figure~\ref{fig:n=3}. The result is 
\begin{align}
 \bigexpecin{\mu_{1,0}}{\Wick{\phi(x)^4}\Wick{\phi(y)^4}\Wick{\phi(z)^4}}
 &= 2!\binom{4}{2}^2 4! 
 \sum_{k_1,\dots,k_6\in\Z} 
 h_x^{\otimes 4}(k_1,k_2,k_3,k_4)
 h_y^{\otimes 4}(k_1,k_2,k_5,k_6)
 h_z^{\otimes 4}(k_3,k_4,k_5,k_6) \\
 &= 1728 \biggpar{\sum_{k_1\in\Z} \frac{e_{k_1}(x-y)}{\lambda_{k_1}}}^2
 \biggpar{\sum_{k_2\in\Z} \frac{e_{k_2}(y-z)}{\lambda_{k_2}}}^2
 \biggpar{\sum_{k_3\in\Z} \frac{e_{k_3}(x-z)}{\lambda_{k_3}}}^2 \\
 &= 1728 \, G(x-y)^2 G(y-z)^2 G(x-z)^2\;.
\end{align}
It follows that 
\begin{align}
 \biggexpecin{\mu_{1,0}}{\biggpar{\int_{\Lambda}\Wick{\phi(x)^4}\6x}^3}
 &= 1728 \int_{\Lambda^3} G(x-y)^2 G(y-z)^2 G(x-y)^2 \6x\6y\6z \\
 &= 1728 \,\Pi\Bigpar{\FGVI}\;.
\end{align} 
With these examples, the pattern should have become clear. The terms in 
the expansion~\eqref{eq:Zratio_1_Wick} can be written as some combinatorial 
coefficients, times an integral of a product of Green functions. To 
formalise this, we make the following definition.

\begin{definition}[Vacuum Feynman diagram]
A \emph{vacuum diagram} is a multigraph $\Gamma = (\cV,\cE)$, meaning there can 
be multiple edges between vertices. Its \emph{valuation} is defined by 
\begin{equation}
 \Pi(\Gamma) = \int_{\Lambda^\cV} \prod_{e\in\cE} 
 G(x_{e_+} - x_{e_-}) \6x\;,
\end{equation} 
where $e_\pm$ are the vertices connected by the edge $e$.
\end{definition}

The general principle behind the above examples is as follows.

\begin{proposition}[Expansion of moments into Feynman diagrams]
For any $n\geqs2$, 
\begin{equation}
\label{eq:expansion_Feynman_graphs} 
 \biggexpecin{\mu_{1,0}}{\biggpar{\int_{\Lambda}\Wick{\phi(x)^4}\6x}^n}
 = \sum_k \Pi(\Gamma_{n,k})\;,
\end{equation} 
where the sum runs over all vacuum diagrams $\Gamma_{n,k}$ with $n$ 
vertices and $2n$ edges, obtained as perfect pairwise matchings of
$n$ vertices of arity $4$ (each vertex belongs to four edges), 
when matchings of different legs are counted as different terms. 
\end{proposition}
\begin{proof}
For any $n\geqs2$, we write 
\begin{equation}
 \prod_{i=1}^n\Wick{\phi(x_i)^4}
 = \prod_{i=1}^n\hat I_4(h_{x_i}^{\otimes 4}))\;.
\end{equation} 
Taking the expectation, we are left with the component in the zeroth 
Wiener chaos, which by a repeated application of Proposition~\ref{prop:mult_nm} 
can be represented as the sum over all pairwise matchings of $n$ vertices 
with $4$ legs each. Each pairing gives rise to a Green function, and the 
result follows by integrating over all $x_i$. 
\end{proof}

% The following exercise shows that the Green function $G = (\id - \Delta)^{-1}$ is indeed 
% bounded in dimension $d=1$. 

\begin{exercise}
\label{ex:bound_G1} 
Show that for $g$ periodic,
the unique periodic solution of $f''(x) = f(x) - g(x)$ is given by 
\begin{equation}
\begin{pmatrix}
f(x) \\ f'(x)
\end{pmatrix}
= -\bigbrak{U(-1)-{\protect \one}}^{-1} 
\int_x^{x+1} U(x-y) 
\begin{pmatrix}
 0 \\ g(y) 
\end{pmatrix}\6y\;,
\qquad 
U(x) = 
\begin{pmatrix}
 \cosh(x) & \sinh(x) \\
 \sinh(x) & \cosh(x)
\end{pmatrix}\;.
\end{equation} 
Deduce that the Green function $G = (\id-\Delta)^{-1}$ is bounded 
in dimension $d=1$. 
\end{exercise}

\begin{exercise}
\label{ex:bound_Gamma_nk} 
Give an upper bound on the number of vacuum diagrams $\Gamma_{n,k}$ 
occurring in~\eqref{eq:expansion_Feynman_graphs} for given $n$, by allowing the 
graphs to have loops (edges connecting a vertex with itself). 
Assuming this bound has the right order of magnitude, what does this 
say about the convergence of the series~\eqref{eq:Zratio_1_Wick}?
\end{exercise}

%%%%%%%%%%%%%%%%%%%%%%%%%%%%%%%%%%%%%%%%%%%%%%%%%%%%%%%%%%%%%%%%%%%%%%%%%%%%%%%%

\subsection{The linked-cluster theorem}
\label{ssec:phi41_cluster} 

The vacuum diagrams occurring in the $n$th power of the expansion need not 
be connected. For instance, the $4$th power contains terms of the form 
$\FGIV^2$, that arise from pairing the legs of two vertices between each 
other, and of the other two vertices among themselves. However, a rather 
remarkable result known in quantum field theory as 
\emph{linked-cluster 
theorem}~\cite{Brouder09, Rivasseau_09, Salmhofer_Renormalization} states 
that the \emph{logarithm} of the ratio of partition functions, that is, 
its cumulant expansion, contains only the connected graphs. More precisely, 
it is obtained by keeping only the connected diagrams in the expansion. 

\begin{theorem}[Linked-cluster theorem]
The cumulant expansion of the ratio of partition functions is given by 
\begin{equation}
\label{eq:cumulant_Feynman_graphs} 
 \log\frac{\cZ_{1,\alpha}}{\cZ_{1,0}}
 \asymp
 \sum_{n\geqs0} \frac{(-\alpha)^n}{n!}
 \sum_{k\colon \Gamma_{n,k} \text{connected}} \Pi(\Gamma_{n,k})\;.
\end{equation} 
\end{theorem}
\begin{proof}
This elegant proof is due to Dimitri Faure. We use the formalism of 
convolution algebras introduced in Section~\ref{ssec:convolution}. 
For an abstract variable $x$ (unrelated to coordinates on $\Lambda$), 
we write $\psi(x^n)$ for the coefficient of $(-\alpha)^n/n!$ in the cumulant  expansion~\eqref{eq:cumulant_Feynman_graphs}. This means that 
\begin{equation}
 \log\frac{\cZ_{1,\alpha}}{\cZ_{1,0}}
 \asymp
 \Lambda(\psi)(-\alpha)\;,
\end{equation} 
where $\Lambda$ is the map introduced in~\eqref{eq:def_Lambda}. 
By Proposition~\ref{prop:Lambda}, we have 
\begin{equation}
 \frac{\cZ_{1,\alpha}}{\cZ_{1,0}}
 \asymp
 \Lambda(\ph)(-\alpha)
 = \sum_{n\geqs0} \frac{(-\alpha)^n}{n!}\ph(x^n)\;,
\end{equation} 
where
\begin{equation}
\label{eq:Feynman_ph_xn} 
 \ph(x^n) 
 = \exp_\ast(\psi)(x^n) 
 = \sum_{k=0}^n \frac{1}{k!}
 \sum_{\substack{n_1, \dots, n_k\geqs1\\n_1+\dots+n_k=n}}
 \frac{n!}{n_1!\dots n_k!}\psi(x^{n_1})\dots\psi(x^{n_k})\;.
\end{equation} 
To deal with combinatorics, we assume that all vertices and edges of 
the diagrams are numbered. Assuming~\eqref{eq:cumulant_Feynman_graphs} 
is true, each $\psi(x^n)$ is a linear combination of connected graphs 
with $n$ vertices. Then~\eqref{eq:Feynman_ph_xn} says that the coefficient 
of $(-\alpha)^n/n!$ in the expansion~\eqref{eq:expansion_Feynman_graphs} 
is obtained by all possible disjoint unions of connected graphs 
such that the total number of vertices is $n$. The multinomial coefficient 
accounts for the choices of vertices in the subgraphs, while the factor 
$1/k!$ accounts for the fact that the order of the subgraphs is irrelevant. 
Since this yields \emph{all} pairings of $n$ vertices, the result follows 
from uniqueness of coefficients of power series. 
\end{proof}

\begin{example}
Since $\psi(x) = 0$ by~\eqref{eq:expec_phi4}, we obtain 
\begin{equation}
 \exp_\ast(\psi)(x^4)
 = \psi(x^4) + \binom{4}{2} \psi(x^2)^2\;,
\end{equation} 
since the only allowed decompositions of $4$ are $4$ and $2+2$. 
This means that the term of order $\alpha^4$ in the 
expansion~\eqref{eq:expansion_Feynman_graphs} differs from the corresponding 
term in the cumulant expansion by a term $\frac12\binom{4}{2} \FGIV^2$. 
The factor $\frac12\binom{4}{2} = 3$ is precisely the number of pairwise 
matchings of four vertices. 
\end{example}

%%%%%%%%%%%%%%%%%%%%%%%%%%%%%%%%%%%%%%%%%%%%%%%%%%%%%%%%%%%%%%%%%%%%%%%%%%%%%%%%

\subsection{Asymptotic expansion}
\label{ssec:phi41_asymp} 

So far, we have not shown that the expansion~\eqref{eq:expansion_Feynman_graphs} 
is a genuine asymptotic expansion. We do this now, by proving the following 
result.

\begin{proposition}[Asymptotic series]
\label{prop:asymptotic} 
For every $n\geqs0$ there exists a constant $M_n$ such that the ratio of partition 
functions satisfies 
\begin{equation}
\label{eq:asymptotic_expansion} 
 \biggabs{\frac{\cZ_{1,\alpha}}{\cZ_{1,0}}
 - \sum_{m=0}^n \frac{(-\alpha)^m}{m!}
 \biggexpecin{\mu_{1,0}}{\biggpar{\int_{\Lambda}\Wick{\phi(x)^4}\6x}^m}}
 \leqs M_n \alpha^{n+1}\;.
\end{equation} 
\end{proposition}

We start by showing an a-priori bound on the Laplace transform. 
To lighten notations, we will write
\begin{equation}
\label{eq:def_X} 
 \bX = \int_{\Lambda}\Wick{\phi(x)^4}\6x\;.
\end{equation} 

\begin{lemma}
\label{lem:exponential_exp} 
There exists $\alpha_0 > 0$ such that for all $\alpha\in[0,\alpha_0)$,
one has 
\begin{equation}
 0 \leqs \bigexpecin{\mu_{1,0}}{\e^{-\alpha \bX}} \leqs 1 + \Order{\alpha}\;.
\end{equation} 
\end{lemma}
\begin{proof}
In order to exploit signs, we write $\bX = \bX_0 - 6C\bY_0 + 3C^2$, where
$C=G(0)$ and 
\begin{equation}
 \bX_0 = \int_{\Lambda}\phi(x)^4\6x \geqs 0\;, \qquad 
 \bY_0 = \int_{\Lambda}\phi(x)^2\6x \geqs 0\;. 
\end{equation} 
Therefore we have 
\begin{equation}
 \bigexpecin{\mu_{1,0}}{\e^{-\alpha \bX}}
 = \e^{-3C^2\alpha} \bigexpecin{\mu_{1,0}}{\e^{-\alpha \bX_0 + 6C\alpha \bY_0}}
 = \e^{-3C^2\alpha} \frac{\widetilde \cZ(\alpha)}{\widetilde \cZ(0)}
 \bigexpecin{\tilde\mu(\alpha)}{\e^{-\alpha \bX_0}}\;,
\end{equation} 
where $\tilde\mu(\alpha)$ is a GFF with covariance 
$((1-6C\alpha)\id - \Delta)^{-1}$, and $\widetilde\cZ(\alpha)$ denotes the normalisation
of this measure. Then we have 
\begin{align}
 \log\frac{\widetilde \cZ(0)}{\widetilde \cZ(\alpha)}
 &= \frac12 \log\prod_{k\in\Z} \frac{1-6\alpha+2\pi k^2}{1+2\pi k^2} \\
 &= \frac12 \log\prod_{k\in\Z} 
 \biggpar{1 - \frac{6\alpha}{1+2\pi k^2}} \\
 &= \frac12 \sum_{k\in\Z} 
 \log\biggpar{1 - \frac{6\alpha}{1+2\pi k^2}} 
 \asymp -3\alpha \sum_{k\in\Z} \frac{1}{1+2\pi k^2}\;.
\end{align}
Since the sum converges, we obtain 
$\widetilde\cZ(\alpha) = \widetilde\cZ(0)[1+\Order{\alpha)}$, 
where $\widetilde\cZ(0) = \cZ_{1,0}$. As a consequence,
\begin{equation}
 \bigexpecin{\mu_{1,0}}{\e^{-\alpha \bX}} = 
\bigexpecin{\tilde\mu(\alpha)}{\e^{-\alpha \bX_0}}
\bigpar{1+\Order{\alpha}}\;.
\end{equation} 
Since $\bX_0$ is positive, we have 
$0\leqs \bigexpecin{\tilde\mu(\alpha)}{\e^{-\alpha \bX_0}} \leqs 1$,
which concludes the proof. 
\end{proof}

\begin{proof}[{\sc Proof of Proposition~\ref{prop:asymptotic}}]
For $n\geqs0$ and $t\in\R$, let 
\begin{equation}
 D_n(t) = \e^{-t} - \sum_{m=0}^n \frac{(-t)^m}{m!}\;.
\end{equation} 
If $n$ is even, we have 
\begin{equation}
\label{eq:bounds_Dn} 
\begin{cases}
0 \leqs D_n(t) 
\leqs \dfrac{(-t)^n}{n!} 
& \text{if $t \geqs 0$\;,} \\[10pt]
0 \leqs D_n(t)
\leqs \dfrac{(-t)^n}{n!} \e^{-t}
& \text{if $t < 0$\;.} 
\end{cases}
\end{equation}
It follows that 
\begin{equation}
 \bigexpecin{\mu_{1,0}}{\abs{D_n(\alpha\bX)}\indicator{\bX\geqs0}}
 \leqs \frac{\alpha^n}{n!} \expec{\bX^n}\;,
\end{equation} 
while 
\begin{align}
 \bigexpecin{\mu_{1,0}}{\abs{D_n(\alpha\bX)}\indicator{\bX<0}}
 &\leqs \frac{\alpha^n}{n!} \expec{\bX^n\e^{-\alpha \bX}\indicator{\bX<0}}\\
 &\leqs \frac{\alpha^n}{n!} 
 \sqrt{\expec{\bX^{2n}} \expec{\e^{-2\alpha \bX}\indicator{\bX<0}}}\;,
\end{align} 
by the Cauchy--Schwarz inequality. We know from Exercise~\ref{ex:bound_G1} 
that $\expec{\bX^{2n}}$ is finite, while Lemma~\ref{lem:exponential_exp} 
shows that $\expec{\e^{-2\alpha \bX}\indicator{\bX<0}}$ is bounded as well. 
A similar argument applies for odd $n$, with some signs reversed 
in~\eqref{eq:bounds_Dn}. 
\end{proof}

With the computations made in Section~\ref{ssec:phi41_wick}, we have thus 
obtained that the ratio of partition functions satisfies  
\begin{equation}
 \frac{\cZ_{1,\alpha}}{\cZ_{1,0}}
 = 1 + 12\alpha^2 \Pi\Bigpar{\FGIV} + 288 \alpha^3 \Pi\Bigpar{\FGVI}
 + \Order{\alpha^4}\;,
\end{equation} 
and more terms can be computed if needed. This expansion does not converge, 
however, since the number of pairwise matchings grows like $(4n-1)!!$ 
(cf.~Exercise~\ref{ex:bound_Gamma_nk}), which by~\eqref{eq:expec_X2k}
and Stirling's formula behaves like $(n!)^2$. The constant $M_n$ 
in~\eqref{eq:asymptotic_expansion} thus grows like $n!$. Such an expansion is called 
\emph{Gevrey-$1$}. The non-convergence of the expansion does not mean that 
it is useless, but it means that there is an optimal value of $n$, 
depending on $\alpha$, at which the expansion should be stopped to 
obtain the smallest possible error bound. 

\begin{exercise}
Let $r(n) = n!\alpha^n$. By extending $r$ to real arguments via $n!=\Gamma(n+1)$ and 
using Stirling's formula, estimate the minimal value of $r(n)$ for small 
$\alpha$. For what $n$ is this minimal value reached?
\end{exercise}

%%%%%%%%%%%%%%%%%%%%%%%%%%%%%%%%%%%%%%%%%%%%%%%%%%%%%%%%%%%%%%%%%%%%%%%%%%%%%%%%

\subsection{Two-point function}
\label{ssec:phi41_two_point} 

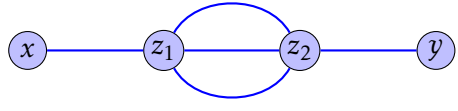
\begin{figure}
\begin{center}
\begin{tikzpicture}[>=stealth',main node/.style={circle,minimum
size=0.5cm,inner sep=1pt,fill=blue!25,draw},small node/.style={draw,circle,fill=white,minimum
size=3pt,inner sep=0pt},scale=1.2]

\node[main node] (x) at (0,0) {$x$};
\node[main node] (z1) at (1.5,0) {$z_1$};
\node[main node] (z2) at (3,0) {$z_2$};
\node[main node] (y) at (4.5,0) {$y$};

\draw[thick,blue] (x) to (z1);
\draw[thick,blue] (z1) to (z2);
\draw[thick,blue] (z1) to[out=60,in=120] (z2);
\draw[thick,blue] (z1) to[out=-60,in=-120] (z2);
\draw[thick,blue] (z2) to (y);
\end{tikzpicture}
\end{center}
\vspace{-4mm}
\caption[]{Pairing corresponding to the term $n=2$ in the expansion
of the two-point function.}
\label{fig:two_point} 
\end{figure}

Let us briefly outline how the two-point function 
\begin{equation}
 G_{2,1,\alpha}(x,y) = \expecin{\mu_{1,\alpha}}{\phi(x)\phi(y)}
\end{equation} 
can be computed by a similar procedure. By~\eqref{eq:expec_mudalpha} we have 
\begin{equation}
 G_{2,1,\alpha}(x,y) = \frac{\cZ_{1,0}}{\cZ_{1,\alpha}}
 \bigexpecin{\mu_{1,0}}{\phi(x)\phi(y)\e^{-\alpha \bX}}\;, 
\end{equation} 
where $\bX$ denotes the integral of the fourth Wick power, cf.~\eqref{eq:def_X}. 
Expanding the exponential, we get 
\begin{equation}
 \bigexpecin{\mu_{1,0}}{\phi(x)\phi(y)\e^{-\alpha \bX}}
 \asymp \sum_{n\geqs0} \frac{(-\alpha)^n}{n!}
 \bigexpecin{\mu_{1,0}}{\phi(x)\phi(y)\bX^n}\;.
\end{equation} 
We already know that the term $n=0$ is equal to $G(x-y)$. For $n=1$, we 
find 
\begin{equation}
 \bigexpecin{\mu_{1,0}}{\phi(x)\phi(y)\bX}
 = \int_\Lambda \bigexpecin{\mu_{1,0}}{\phi(x)\phi(y)\Wick{\phi(z)^4}} \6z 
= \int_\Lambda \bigexpecin{\mu_{1,0}}{\hat I_1(h_x)\hat I_1(h_y)
\hat I_4(h_z^{\otimes 4})} \6z\;.
\end{equation} 
This is equal to zero, because there is no perfect matching, leaving no 
free legs, of two vertices with one leg each and one vertex with four legs. 
For $n=2$, we obtain 
\begin{equation}
 \bigexpecin{\mu_{1,0}}{\phi(x)\phi(y)\bX^2}
= \int_\Lambda \int_\Lambda
\bigexpecin{\mu_{1,0}}{\hat I_1(h_x)\hat I_1(h_y)
\hat I_4(h_{z_1}^{\otimes 4})\hat I_4(h_{z_2}^{\otimes 4})} \6z_1\6z_2\;.
\end{equation} 
This is now different from zero, because there exist perfect matchings, as 
show in Figure~\ref{fig:two_point}. We thus conclude that the 
expectation of $\phi(x)\phi(y)\e^{-\alpha \bX}$ can be computed in a 
similar way as for the ratio of partition functions, except that it now 
involves Feynman diagrams with two free legs, labeld $x$ and $y$.

\begin{remark}
\label{rem:BFS} 
There exists a slightly different representation of the two-point function, 
that avoids having to divide by the ratio of partition functions, based on the Schwinger--Dyson equations~\eqref{eq:ibp}~\cite{Brydges_Frohlich_Sokal_CMP83_RW}. 
\end{remark}

%%%%%%%%%%%%%%%%%%%%%%%%%%%%%%%%%%%%%%%%%%%%%%%%%%%%%%%%%%%%%%%%%%%%%%%%%%%%%%%%

\section{The $\Phi^4_2$ model}
\label{sec:phi42} 

In this section, we consider the $\Phi^4$ model on the two-dimensional torus
$\Lambda = \T^2$. We have seen in Section~\ref{ssec:GFF_prop} that the GFF 
in dimension $2$ has infinite variance. In fact, one can show that the Green 
function behaves like
\begin{equation}
\label{eq:Green_GFF2} 
 G(x) \asymp \bigabs{\log(\norm{x})}
 = \log(\norm{x}^{-1})\;.
\end{equation} 
This can be seen by computing the Green function $G_{\R^2}$ of the full plane 
$\R^2$ using polar coordinates, and then periodising it via 
\begin{equation}
 G_{\T^2} (x) = \sum_{k\in\Z^2} G_{\R^2}(x - k)\;. 
\end{equation} 
We know however that the truncated two-dimensional GFF has finite moments. 
This suggests considering the modified energy 
\begin{equation}
\label{eq:phi42_H_Wick} 
 \cH_{2,\alpha,N}^{\text{Wick}}(\phi_N) 
 = \int_{\Lambda} \biggbrak{\norm{\nabla\phi_N(x)}^2 + \frac12 \phi_N(x)^2 
 + \alpha\Wick{\phi_N(x)^4}_{C_N}} \6x\;,
\end{equation} 
defined for the truncated field 
\begin{equation}
 \phi_N(x) = \sum_{k\in\cK_N} \frac{X_k}{\sqrt{\lambda_k}} e_k(x)\;,
\end{equation} 
which has variance 
\begin{equation}
\label{eq:CN_phi42} 
 C_N = \sum_{k\in\cK_N} \frac{1}{\lambda_k}
 \asymp \log(N)\;.
\end{equation} 
We recall that $\cK_N = \setsuch{k\in\Z^2}{\abs{k} \leqs N}$, 
where $\abs{k} = \abs{k_1} + \abs{k_2}$. 

The situation is now similar to the one we have encountered in dimension $1$, 
except for the important difference that the constant $C_N$ occurring 
in~\eqref{eq:phi42_H_Wick} depends on the cut-off $N$. The model thus changes 
with $N$. This is an instance of what is called \emph{renormalisation} in 
quantum field theory, and $C_N$ is known as a \emph{counterterm}. 

The computations from Section~\ref{ssec:phi41_two_point} can now be 
repeated in the same way, and result in an expansion of the form 
\begin{equation}
\label{eq:Zratio_2_Wick} 
 \frac{\cZ_{2,\alpha,N}}{\cZ_{2,0,N}}  
 \asymp \sum_{n\geqs0} \frac{(-\alpha)^n}{n!}
 \biggexpecin{\mu_{2,0,N}}{\biggpar{\int_{\Lambda}\Wick{\phi_N(x)^4}_{C_N}\6x}^n}
\end{equation} 
of the ratio of partition functions, where the expectations of powers of 
the fourth Wick power are given by a sum of valuations of the same Feynman 
vacuum diagrams as in~\eqref{eq:expansion_Feynman_graphs}. The only difference 
is that the diagrams involve the Green function satisfying~\eqref{eq:Green_GFF2}. 
It is not immediately obvous that the diagrams all have a finite valuation, 
but we will show in Section~\ref{ssec:phi43_Hepp} below that this is indeed 
the case, uniformly in the cut-off $N$. 

The linked-cluster theorem is also true in this case. However, the proof of 
the a priori bound on the Laplace transform given in 
Lemma~\ref{lem:exponential_exp} does not work here, because of the diverging 
constant $C_N$ in the energy~\eqref{eq:phi42_H_Wick}. Fortunately, there 
is an alternative proof of that bound, due to Nelson.

%%%%%%%%%%%%%%%%%%%%%%%%%%%%%%%%%%%%%%%%%%%%%%%%%%%%%%%%%%%%%%%%%%%%%%%%%%%%%%%%

\subsection{Nelson's estimate}
\label{ssec:phi42_Nelson} 

In order to bound the Laplace transform of the fourth Wick power $\bX$, 
we first derive the following consequence of the equivalence of moments 
bound~\eqref{eq:equivalence_moments_infdim}.

\begin{lemma}
\label{lem:Nelson} 
Fix two cut-offs $M > N \geqs 1$. Then for any $p > 1$ and $n\geqs2$, 
there exists a constant $K_n$ depending only on $n$ such that 
\begin{equation}
 \biggexpecin{\mu_{2,0,N}}{
 \biggpar{\int_\Lambda \Wick{\phi_M(x)^n}_{C_M}\6x - 
 \int_\Lambda \Wick{\phi_N(x)^n}_{C_N}\6x}^{2p}}^{1/(2p)} 
 \leqs K_n(2p-1)^{n/2} \frac{(\log N)^{n-2}}{N}\;.
\end{equation} 
\end{lemma}
\begin{proof}
By the same computation as in the proof of Proposition~\ref{prop:variance_Wick_GFF2},
see~\eqref{eq:proof_moments01}, we have 
\begin{equation}
 \biggexpec{\int_\Lambda \Wick{\phi_M^n(x)}_{C_M} \6x
 \int_\Lambda \Wick{\phi_N^n(x)}_{C_N} \6x} 
 = n! \sum_{\substack{k_1,\dots,k_n\in\cK_N\\k_1 + \dots + k_n = 0}} 
 \frac{1}{\lambda_{k_1}\dots\lambda_{k_n}}\;.
\end{equation}
It follows by expanding the square that 
\begin{equation}
 \biggexpec{\biggpar{\int_\Lambda \Wick{\phi_M^n(x)}_{C_M} \6x 
 - \int_\Lambda \Wick{\phi_N^n(x)}_{C_N} \6x}^2}
 = n! \sum_{\substack{k_1,\dots,k_n\in\cK_M \setminus \cK_N\\k_1 + \dots + k_n = 0}} 
 \frac{1}{\lambda_{k_1}\dots\lambda_{k_n}}\;.
%  - n! \sum_{\substack{k_1,\dots,k_n\in\cK_N\\k_1 + \dots + k_n = 0}} 
%  \frac{1}{\lambda_{k_1}\dots\lambda_{k_n}}\;.
\end{equation}
By a similar argument as in Lemma~\ref{lem:Young}, one can show by induction on $n$ that 
\begin{equation}
 \sum_{\substack{k_1,\dots,k_n\in\cK_M \setminus \cK_N\\k_1 + \dots + k_n = k}} 
 \frac{1}{\lambda_{k_1}\dots\lambda_{k_n}}
 \leqs K_n \min\Biggset{\frac{(\log N)^{n-2}}{N^2}, \frac{(\log \norm{k})^{n-2}}{\norm{k}^2}}\;,
\end{equation} 
where $K_n$ does not depend on $N$ or $M$. Therefore, the result follows by taking $k=0$
in the above bound and using equivalence of moments 
(Theorem~\eqref{eq:equivalence_moments_infdim}). 
\end{proof}

\begin{proposition}[Nelson's estimate]
\label{prop:Nelson}
For any $\alpha \geqs 0$, 
there exists a constant $K > 0$, independent of $N$, such that for all $N\in\N$, one has 
\begin{equation}
 0 \leqs \biggexpecin{\mu_{2,0,N}}
 {\exp\biggset{-\alpha \int_\Lambda \Wick{\phi_N(x)^4}\6x}} 
 \leqs K\;.
\end{equation} 
\end{proposition}
\begin{proof}
By the definition of scaled Hermite polynomials, 
\begin{equation}
 H_4(\phi_N(x);C_N) = 
 \bigpar{\phi_N(x)^2 - 3C_N^2}^2 - 6C_N^2\;,
\end{equation} 
which shows that 
\begin{equation}
 \bX_N := \int_\Lambda \Wick{\phi_N(x)^4}_{C_N}\6x
 \geqs -6C_N^2 =: -D_N\;.
\end{equation} 
Since $\bigexpecin{\mu_{2,0,N}}{\e^{-\alpha\bX_N}\indicator{\bX_N \geqs 0}}
\leqs \bigprobin{\mu_{2,0,N}}{\bX_N \geqs 0} \leqs 1$, it is sufficient to bound
\begin{align}
 \bigexpecin{\mu_{2,0,N}}{\e^{-\alpha\bX_N}\indicator{\bX_N < 0}}
 &= \bigprobin{\mu_{2,0,N}}{\bX < 0} 
 + \alpha\int_0^\infty \e^{\alpha t} \bigprobin{\mu_{2,0,N}}{-\bX_N > t} \6t \\
 &\leqs \e^\alpha + \int_1^\infty \e^{\alpha t} \bigprobin{\mu_{2,0,N}}{-\bX_N > t} \6t\;.
\label{eq:integral_proof_Nelson} 
\end{align}
If $t \geqs D_N$, then $\bigprobin{\mu_{2,0,N}}{-\bX_N > t} = 0$. Otherwise 
we have, for any $M\in\N$ and any choice of even $p(t)$, 
\begin{align}
 \bigprobin{\mu_{2,0,N}}{-\bX_N > t}
 &\leqs \bigprobin{\mu_{2,0,N}}{\bX_M-\bX_N > t - D_M} \\
 &\leqs \bigprobin{\mu_{2,0,N}}{\abs{\bX_M-\bX_N}^{2p(t)} > \abs{t - D_M}^{2p(t)}}\;.
\end{align}
We now apply this inequality with $M=M(t)$ satisfying 
\begin{equation}
 t - D_{M(t)} \geqs 1\;,
\end{equation} 
which by~\eqref{eq:CN_phi42} is possible with $\log(M(t))$ of order $t^{1/2}$, and 
implies $M(t) < N$. By Markov's inequality and Lemma~\ref{lem:Nelson} with $n = 4$, 
\begin{align}
 \bigprobin{\mu_{2,0,N}}{-\bX_N > t}
 &\leqs \bigexpecin{\mu_{2,0,N}}{\abs{\bX_M(t)-\bX_N}^{p(t)}} \\
 &\leqs K_4 \bigpar{p(t)-1}^{2p(t)}
 \frac{(\log N)^{4p(t)}}{N^{2p(t)}}\\
 &\leqs K(\eta) \frac{(p(t)-1)^{2p(t)}}{M(t)^{2(1-\eta)p(t)}}
\end{align}
with a finite $K(\eta)$ for any $\eta > 0$. 
Choosing $p(t)$ of order $t^\beta$ for $\beta > \frac12$, we obtain 
\begin{equation}
 \log \Bigpar{\e^{\alpha t}\bigprobin{\mu_{2,0,N}}{-\bX_N > t}}
 \leqs \alpha t + c_1\beta t^\beta \log(t) - c_2(1-\eta) t^{\beta+1/2}
\end{equation} 
for constants $c_1, c_2 > 0$. Since the term $-c_2(1-\eta) t^{\beta+1/2}$ 
dominates for large $t$, this leads to a convergent integral 
in~\eqref{eq:integral_proof_Nelson}. 
\end{proof}

As a consequence, we obtain the following analogue of Proposition~\ref{prop:asymptotic}, 
by the same proof. 

\begin{corollary}[Asymptotic series]
\label{cor:asymptotic} 
For every $n\geqs0$ and $N\geqs1$, there exists a constant $M_n$, independent of $N$,
such that the ratio of partition functions satisfies 
\begin{equation}
\label{eq:asymptotic_expansion_2d} 
 \biggabs{\frac{\cZ_{2,\alpha,N}}{\cZ_{2,0,N}}
 - \sum_{m=0}^n \frac{(-\alpha)^m}{m!}
 \biggexpecin{\mu_{2,0,N}}{\biggpar{\int_{\Lambda}\Wick{\phi_N(x)^4}\6x}^m}}
 \leqs M_n \alpha^{n+1}\;.
\end{equation} 
\end{corollary}

As a consequence, we obtain again an expansion of the form 
\begin{equation}
\label{eq:Z_ratio_2d} 
 \frac{\cZ_{2,\alpha,N}}{\cZ_{2,0,N}}
 = 1 + 12\alpha^2 \Pi_N\Bigpar{\FGIV} + 288 \alpha^3 \Pi_N\Bigpar{\FGVI}
 + \Order{\alpha^4}
\end{equation} 
for the ration of partition functions, and a similar expansion holds for 
the two-point function and other random variables. The main difference with 
the one-dimensional case, besides the fact that the model itself depends 
on $C_N$, is that the valuations of Feynman diagrams depend 
on $N$, via the cut-off Green function $G_N$, given by~\eqref{eq:cov_GFFN}. 
One can however show that these valuations all converge to a finite 
limit as $N\to\infty$. 

\begin{remark}
A non-perturbative proof of the existence of the ratio of partition 
functions~\eqref{eq:Z_ratio_2d}, based on a Girsanov formula, has been obtained 
by Barashkov and Gubinelli~\cite{Barashkov_Gubinelli_18}. 
\end{remark}

%%%%%%%%%%%%%%%%%%%%%%%%%%%%%%%%%%%%%%%%%%%%%%%%%%%%%%%%%%%%%%%%%%%%%%%%%%%%%%%%

\section{The $\Phi^4_3$ model*}
\label{sec:phi43} 

We consider now the $\Phi^4$ model on the three-dimensional torus 
$\Lambda = \T^3$. One can show that the Green function now behaves as 
\begin{equation}
 G(x) \asymp \frac{1}{\norm{x}}\;,
\end{equation} 
while the variance of the Gaussian free field satisfies 
\begin{equation}
 C_N = G_N(0) = \sum_{k\in\cK_N} \frac{1}{\lambda_k}
\asymp N
\end{equation} 
with $\cK_N = \setsuch{k\in\Z^3}{\abs{k}\leqs N}$. One also checks that 
the truncated Green function satisfies  
\begin{equation}
\label{eq:Green_GFF3_N} 
 G_N(x) = \sum_{k\in\cK_N} \frac{1}{\lambda_k} e_k(x) 
 \asymp \frac{1}{\norm{x} + N^{-1}}\;.
\end{equation} 
A natural guess would be that the energy~\eqref{eq:phi42_H_Wick}, transposed 
to the three-dimensional setting, would still lead to well-defined asymptotic 
expansions of expectations. Unfortunately, this is not the case. The actual 
result is as follows. 

\begin{theorem}[Renormalisation of the $\Phi^4_3$ model]
\label{thm:phi43} 
Define the energy by 
\begin{equation}
\label{eq:phi43_H_Wick} 
 \cH_{3,\alpha,N}^{\Phi^4}(\phi_N) 
 = \int_{\Lambda} \biggbrak{\norm{\nabla\phi_N(x)}^2 
 + \frac12 \phi_N(x)^2 
 + \alpha\Wick{\phi_N(x)^4}_{C_N}
 + \beta_N(\alpha)\Wick{\phi_N(x)^2}_{C_N} + \gamma_N(\alpha)} \6x\;,
\end{equation} 
where the additional counterterms are 
\begin{align}
 \beta_N(\alpha) &= -48\alpha^2 \Pi_N\bigpar{\FGIII}\;, \\
 \gamma_N(\alpha) &= 12\alpha^2 \Pi_N\Bigpar{\FGIV}
 - 288\alpha^3 \Pi_N\Bigpar{\FGVI}\;.
\end{align}
Then the $n$-point functions admit asymptotic expansions in $\alpha$, 
all of whose terms are uniformly bounded in the cut-off $N$. 
\end{theorem}

The counterterm $\beta_N(\alpha)$ is called \emph{mass renormalisation}, 
while the counterterm $\gamma_N(\alpha)$ is called \emph{energy renormalisation}.
The latter is not crucial for most computations, since it will cancel out 
when taking ratios of partition functions. These counterterms were not 
needed in dimension $2$, because their value actually converges to finite 
limits as $N\to\infty$. However, as we shall see in Section~\ref{ssec:phi43_Hepp}, 
they diverge in dimension $3$, either like $N$ or like $\log N$. This is 
a symptom of the fact that the $\Phi^4_3$ measure is not absolutely 
continuous with respect to the three-dimensional Gaussian free field. 

\begin{exercise}
Determine how the values $\Pi_N\bigpar{\FGIII}$ 
and $\Pi_N\Bigpar{\FGIV}$ scale with $N$ as $N\to\infty$, 
for $d = 3$, by using~\eqref{eq:Green_GFF3_N} and spherical coordinates. 
Compare with the case $d = 2$. 
\end{exercise}

Theorem~\ref{thm:phi43} is an important result in Euclidean Quantum Field theory, 
which has a long history. The earliest works by Glimm and Jaffe and by Feldman 
approached the problem via a detailed combinatorial analysis of Feynman 
diagrams~\cite{Glimm_Jaffe_68,Glimm_Jaffe_73,Feldman74,Glimm_Jaffe_81}. The 
works~\cite{BCGNOPS78,BCGNOPS_80} introduced the idea of using a 
renormalisation group approach, consisting in a decomposition of the covariance 
of the GFF into scales, which then allows to integrate successively over 
one scale after the other. This method was further perfected 
in~\cite{Brydges_Dimock_Hurd_CMP_95}, using polymers to control error terms, an 
approach based on ideas from Statistical Physics~\cite{Gruber_Kunz_71}. 

In another direction, the approach provided 
in~\cite{Brydges_Frohlich_Sokal_CPM83,Brydges_Frohlich_Sokal_CMP83_RW} allows 
to bound correlation functions without having to compute the partition 
function explicitly, by using the Schwinger--Dyson equations (see also 
Remark~\ref{rem:BFS}). This involves the derivation of so-called skeleton 
inequalities, which were obtained up to third order 
in~\cite{Brydges_Frohlich_Sokal_CPM83}, and later extended to all orders 
in~\cite{Bovier_Felder_CMP84}. A relatively compact 
derivation of bounds on the partition function 
based on a Girsanov formula was recently 
obtained in~\cite{Barashkov_Gubinelli_18}. 

We are not going to present a full proof of Theorem~\ref{thm:phi43} here, 
as all its versions are quite technical. However, in the next sections, we 
shall outline some key ideas of a proof.

%%%%%%%%%%%%%%%%%%%%%%%%%%%%%%%%%%%%%%%%%%%%%%%%%%%%%%%%%%%%%%%%%%%%%%%%%%%%%%%%

\subsection{Hepp sectors and subdivergences*}
\label{ssec:phi43_Hepp} 

In this section, we provide a simple way to determine whether 
the value of a vacuum diagram diverges as $N\to\infty$, or not. 
Given a diagram $\Gamma = (\cV,\cE)$, define its degree by 
\begin{equation}
\label{eq:def_degree} 
 \deg(\Gamma) := d(\abs{\cV} - 1) - (d-2)\abs{\cE}\;.
\end{equation}
We will call a diagram \emph{divergent} if $\deg(\Gamma)\leqs 0$. 
Here we have defined the degree for a general dimension $d$, though the main 
focus of this section is the case $d=3$. 

\begin{exercise}
\label{ex:degree} 
Compute the degrees of the diagrams occurring in Theorem~\ref{thm:phi43}
for $d=2$ and $d=3$. 
\end{exercise}

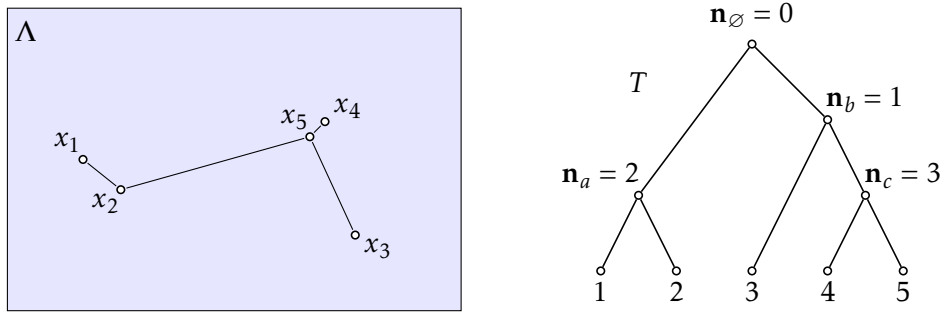
\begin{figure}[t]
\begin{center}
\begin{tikzpicture}[>=stealth',main 
node/.style={draw,circle,fill=white,minimum size=1pt,inner sep=1pt},
line/.style={draw,shorten <=0.5pt,shorten >=0.5pt}]

% grid to help positioning
%\draw[help lines] (-5,-4) grid (5,1);

\draw[thin,fill=blue!10] (0,0) rectangle (6,4); 

\node[main node,semithick,fill=white,
label={[xshift=-0.2cm,yshift=-0.1cm]$x_1$}] (z1) at (1,2) {};

\node[main node,semithick,fill=white,
label={[xshift=-0.2cm,yshift=-0.5cm]$x_2$}] (z2) at (1.5,1.6) {};

\node[main node,semithick,fill=white,
label={[xshift=0.3cm,yshift=-0.5cm]$x_3$}] (z3) at (4.6,1) {};

\node[main node,semithick,fill=white,
label={[xshift=0.3cm,yshift=-0.15cm]$x_4$}] (z4) at (4.2,2.5) {};

\node[main node,semithick,fill=white,
label={[xshift=-0.2cm,yshift=-0.1cm]$x_5$}] (z5) at (4,2.3) {};

%spanning tree
\draw[line] (z1) -- (z2);
\draw[line] (z5) -- (z4);
\draw[line] (z2) -- (z5);
\draw[line] (z3) -- (z5);

\node[] at (0.25,3.7) {$\Lambda$};

\end{tikzpicture}
\hspace{10mm}
\begin{tikzpicture}[>=stealth',main 
node/.style={draw,circle,fill=white,minimum size=1pt,inner sep=1pt}]

% grid to help positioning
%\draw[help lines] (-5,-4) grid (5,1);

% leaves
\node[main node,semithick,fill=white,
label={[xshift=0cm,yshift=-0.6cm]$1$}] (z1) at (0,-3) {};

\node[main node,semithick,fill=white,
label={[xshift=0cm,yshift=-0.6cm]$2$}] (z2) at (1,-3) {};

\node[main node,semithick,fill=white,
label={[xshift=0cm,yshift=-0.6cm]$3$}] (z3) at (2,-3) {};

\node[main node,semithick,fill=white,
label={[xshift=0cm,yshift=-0.6cm]$4$}] (z4) at (3,-3) {};

\node[main node,semithick,fill=white,
label={[xshift=0cm,yshift=-0.6cm]$5$}] (z5) at (4,-3) {};

%inner nodes
\node[main node,semithick,
label={[xshift=0cm,yshift=0cm]$\bno = 0$}] (O) at (2,0) {};

\node[main node,semithick,
label={[xshift=0.5cm,yshift=-0.1cm]$\bn_b = 1$}] (A) at (3,-1) {};

\node[main node,semithick,
label={[xshift=-0.5cm,yshift=-0.1cm]$\bn_a = 2$}] (B) at (0.5,-2) {};

\node[main node,semithick,
label={[xshift=0.5cm,yshift=-0.1cm]$\bn_c = 3$}] (C) at (3.5,-2) {};

% tree
\draw[semithick] (z1) -- (B) -- (z2);
\draw[semithick] (z4) -- (C) -- (z5);
\draw[semithick] (C) -- (A) -- (z3);
\draw[semithick] (B) -- (O) -- (A);

\node[] at (0.5,-0.5) {$T$};

\end{tikzpicture}
\end{center}
\vspace{-3mm}
 \caption[]{A point configuration $x\in\Lambda^5$, its minimal spanning tree 
(left), and the associated Hepp tree $\bT = (T,\bn)$ (right). The inner 
nodes of $T$ are labelled $\varnothing, a, b, c$. For instance, $\bn_{1\wedge 
2}=\bn_a=2$, so that $x_1$ and $x_2$ are at a distance of order $2^{-2}$, while 
$\bn_{3\wedge 5}=\bn_b=1$, so that $x_3$ and $x_5$ are at a distance of order 
$2^{-1}$.} 
\label{fig:Heppsector}
\end{figure}

The following theorem was obtained by Weinberg~\cite{Weinberg60}.

\begin{theorem}[Criterion for non-divergence]
\label{thm:Weinberg} 
Assume $G_N(x) \asymp (\norm{x} + N^{-1})^{d-2}$. 
If\/ $\Gamma$ satisfies $\deg(\bar\Gamma) > 0$ for any subgraph $\bar\Gamma$ 
of\/ $\Gamma$, then $\Pi_N(\Gamma)$ is bounded uniformly in $N$.  
\end{theorem}

This result can be proved quite easily, using an idea introduced by 
Hepp~\cite{Hepp66}.

\begin{definition}[Hepp sector]
\index{Hepp sector}%
Fix a constant $C>1$. Let $T$ be a binary tree with $m = \abs{\cV}$ leaves, and 
let\/ $\bn$ be a node decoration on the vertices of\/ $T$, which is 
non-decreasing on any path from the root of\/ $T$ to a leaf of\/ $T$. We write 
$\bT = (T,\bn)$ and define a subset of $\Lambda^m$ by 
\begin{equation}
 D_{\bT} := 
 \Bigsetsuch{x\in\Lambda^m}{C^{-1}2^{-\bn_{i\wedge j}} \leqs \norm{x_i-x_j} 
\leqs C2^{-\bn_{i\wedge j}} \; \forall i,j\in \set{1,\dots,m}}\;,
\end{equation} 
where $i\wedge j$ denotes the last common ancestor of $i$ and $j$ in the tree 
$T$. Then $D_{\bT}$ is called the \emph{Hepp sector} associated with $\bT$. 
\end{definition}

Given a point configuration $x=(x_1,\dots,x_m)\in\Lambda^m$, one can associate 
with it a Hepp sector $D_{\bT}$ in the following way (\figref{fig:Heppsector}). 
One starts by finding a minimal spanning tree of $x_1, \dots, x_m$. Pairs of 
points which are closest are children of a common node in $T$, which is 
labelled by the power of $2$ which is closest to the length of the edge in the spanning 
tree. The construction is then iterated until all leaves of $T$ are connected 
to the root. 

\begin{proof}[{\sc Proof of Theorem~\ref{thm:Weinberg}}]
One can check that for sufficiently large $C$, $\Lambda^m$ is covered by the 
union of all Hepp sectors $D_{\bT}$. The bound on $G_N(x)$ implies 
\begin{equation}
\label{eq:Hepp_integral} 
 \abs{\Pi_N(\Gamma)}
 \lesssim \sum_{T,\bn}
 \int_{D_{T,\bn}} \prod_{e\in\cE} 
\frac{1}{\bigpar{\norm{x_{e_+}-x_{e_-}} + N^{-1}}^{d-2}}\6x\;.
\end{equation} 
It follows from the definition of Hepp sectors that the term 
$\norm{x_{e_+}-x_{e_-}}$ in \eqref{eq:Hepp_integral} is bounded below by 
$\smash{C^{-1}2^{(d-2)\bn_{\eup}}}$, where $\smash{\eup := e_+\wedge e_-}$ 
denotes the last common ancestor of the two vertices incident to $e$. We 
thus obtain that, uniformly in the cut-off $N$, one has 
\begin{equation}
 \abs{\Pi_N(\Gamma)} \lesssim \sum_{T,\bn} 
 \prod_{e\in\cE} 2^{(d-2)\bn_{\eup}} \int_{D_{T,\bn}} \6x\;.
\end{equation} 
The volume of the Hepp sector, given by the integral over $D_{T,\bn}$, is 
easily seen to have order $\prod_{v\in T} 2^{-d\bn_v}$, where the product runs 
over all inner nodes of $T$. It follows that 
\begin{equation}
\label{eq:bound_PiN_Gamma} 
 \abs{\Pi_N(\Gamma)} \lesssim \sum_{T,\bn} 
 \prod_{v\in T} 2^{-\eta_v\bn_v}\;,
 \quad\text{where}\quad 
 \eta_v := d - (d-2)\sum_{e\in\cE} \indicator{\eup}(v)\;.
\end{equation}
Let $\geqs$ be the partial order on inner vertices of $T$ given by descendancy: 
$w\geqs v$ if and only if the unique path from the root of $T$ to $w$ contains 
$v$. We claim that 
\begin{equation}
 \sum_{w\geqs v} \eta_w > 0
\end{equation} 
holds uniformly in $v\in T$. Indeed, this expression is the degree of a 
subgraph of $\Gamma$, which is positive by assumption. Using this observation, 
it is not difficult to show that $\abs{\Pi_N(\Gamma)}$ is uniformly bounded, by 
induction starting from the leaves of $T$. 
\end{proof}

\begin{example}
Consider the case of the tree depicted in \figref{fig:Heppsector}, with the 
inner vertices of $T$ denoted $\varnothing, a, b, c$. We have $\eta_a > 0$, 
$\eta_c > 0$, $\eta_b + \eta_c > 0$ and $\eta_\varnothing + \eta_a + \eta_b + 
\eta_c > 0$ (this last sum being the degree of $\Gamma$). The sum over all node 
decorations is given by 
\begin{equation}
 \sum_{\bno \geqs 0} 2^{-\eta_\varnothing}
 \sum_{\bn_a \geqs \bno} 2^{- \eta_a \bn_a}
 \sum_{\bn_b \geqs \bno} 2^{- \eta_b \bn_b}
 \sum_{\bn_c \geqs \bn_b} 2^{-\eta_c\bn_c}\;.
\end{equation} 
Performing first the sums over $\bn_a$ and $\bn_c$, then the sum over $\bn_b$, 
and finally the sum over $\bno$ yields indeed a finite quantity. 
Since there are finally many binary trees with $5$ leaves, the result follows. 
\end{example}

The examples seen in Exercice~\ref{ex:degree} may suggest that the large-$N$ 
behaviour of $\Pi_N(\Gamma)$ is governed by the degree of $\Gamma$, in the 
sense that $\abs{\Pi_N(\Gamma)}$ is bounded uniformly in $N$ if 
$\deg(\Gamma) > 0$, diverges like $\log(N)$ if $\deg(\Gamma) = 0$, 
and diverges like $N^{-\deg(\Gamma)}$ if $\deg(\Gamma) < 0$. 
Unfortunately, the reality is a bit more complex. For example, we have 
\begin{equation}
\label{eq:subdivergence} 
 \deg\Bigpar{\FGIIIplus} = 2d - 5(d-2) = 10 - 3d
\end{equation} 
which is strictly positive for $d\leqs3$. However, the diagram contains the 
subgraph $\FGIII$, which is divergent in dimension $3$. As a result, one 
can show that the diagram is divergent as well. This is an instance 
of a \emph{subdivergence}, which is a serious source of complication for 
the analysis. Note that this does not contradict Theorem~\ref{thm:Weinberg}. 

%%%%%%%%%%%%%%%%%%%%%%%%%%%%%%%%%%%%%%%%%%%%%%%%%%%%%%%%%%%%%%%%%%%%%%%%%%%%%%%%

\subsection{BPHZ renormalisation*}
\label{ssec:phi43_BPHZ} 

BPHZ renormalisation, named after Bogoliubov, Parasiuk, Hepp, and 
Zimmermann~\cite{Bogoliubov56,Hepp66,Zimmermann69} is a procedure 
allowing to systematically analyse the divergent behavior of diagrams 
that are either divergent, or contain sub-divergences, and to 
determine the counterterms needed to obtain well-defined asymptotic 
expansions. The key result is the following. 

\begin{theorem}[BPHZ renormalisation]
\label{thm:BPHZ} 
There exists a linear map $\cA$, acting on Feynman diagrams, such that 
\begin{equation}
\label{eq:BPHZ} 
 \Pi_N(\cA(\Gamma)) \asymp 
 \begin{cases}
  N^{-\deg(\Gamma)} & \text{if $\deg(\Gamma)<0$\;,} \\
  \log(N)^\zeta  & \text{if $\deg(\Gamma)=0$\;,}
 \end{cases}
\end{equation}
for a finite integer $\zeta$, while $\Pi_N(\cA(\Gamma))$ is bounded 
uniformly in $N$ if $\deg(\Gamma) > 0$. 
\end{theorem}

For a modern exposition, see~\cite{Hairer_BPHZ}. Slightly sharper bounds
have been obtained for a different model in~\cite{BB19}. While the proof 
is quite involved, the basic mechanism can be understood in a simple example.

\begin{example}
Consider again the diagram in~\eqref{eq:subdivergence}. In that case, 
\begin{equation}
 \cA\Bigpar{\FGIIIplus}
 = - \FGIIIplus + \FGIII \cdot \FGII\;.
\end{equation} 
This means that while 
\begin{equation}
\label{eq:PiN_FGIIIplus} 
 \Pi_N\biggpar{\FGIIIplus} 
 = \iiint_{\Lambda^3} 
 G_N(x_2-x_1)^3 G_N(x_3-x_2) G_N(x_3-x_1)\6x_1\6x_2\6x_3\;,
\end{equation} 
one has 
\begin{equation}
\label{eq:PiN_AFGIIIplus} 
 \Pi_N\biggpar{\cA \FGIIIplus} 
 = -\iiint_{\Lambda^3} 
 G_N(x_2-x_1)^3 G_N(x_3-x_2) \bigbrak{G_N(x_3-x_2) - G_N(x_3-x_1)}\6x_1\6x_2\6x_3\;.
\end{equation} 
The crucial observation is that by Taylor's formula, 
\begin{equation}
\label{eq:BPHZ_Taylor} 
 \abs{G_N(x_3-x_2) - G_N(x_3-x_1)} 
 \lesssim \abs{(x_2 - x_1) \cdot \nabla G_N(x_3-x_1)}
 \lesssim \frac{\norm{x_2-x_1}}{(\norm{x_3-x_1} + N^{-1})^{2}}\;.
\end{equation} 
If $\norm{x_2 - x_1} \ll \norm{x_3-x_1}$, this is smaller than the 
contribution of $G_N(x_3-x_1)$ to~\eqref{eq:PiN_FGIIIplus}. 
This gain is enough to make the integral~\eqref{eq:PiN_AFGIIIplus}
convergent. 
\end{example}

The linear map $\cA$ in~\eqref{eq:BPHZ} has an algebraic meaning: it is actually 
the so-called \emph{antipode} 
of the \emph{Connes--Kreimer extraction--contraction Hopf 
algebra}~\cite{Connes_Kreimer_2000a,Connes_Kreimer_2001}. To explain its 
construction, we introduce the following sets: 
\begin{itemize}
\item   $\F$ is the set of all \emph{connected} multigraphs whose vertices 
have arity $2$, $3$ or $4$ (that   is, $2$, $3$ or $4$ edges meet at each vertex);
\item   $\Fminus \subset \F$ is the subset of all divergent multigraphs in~$\F$;  
\item   $\cF$ is the algebra generated by~$\F$ with respect to the disjoint union 
product~$\cdot$ -- note that $\cF$ also contains \emph{non-connected} Feynman diagrams;
\item   $\cFminus \subset \cF$ is the subalgebra of~$\cF$ generated by~$\Fminus$; 
in particular, for any~$\Gamma \in \cFminus$, all connected components are divergent. 
\end{itemize}
We also denote by $\spanF$, $\spanFminus$, $\spancF$, $\spancFminus$ the linear spans, 
respectively, of $\F$, $\Fminus$, $\cF$ and $\cFminus$. 
The neutral element for multiplication is the empty graph, which we denote by $\unit$. 
Note that the valuation $\Pi_N$ is multiplicative, meaning that 
\begin{equation}
 \Pi_N(\Gamma_1\cdot\Gamma_2) = \Pi_N(\Gamma_1)\Pi_N(\Gamma_2)
 \qquad \text{for all $\Gamma_1,\Gamma_2\in\cF$\;.}
\end{equation} 

\begin{definition}[Connes--Kreimer extraction-contraction coproduct]
The \emph{Connes--Kreimer extraction-contraction 
co\-product} $\DeltaCK: \spanF \to \spancFminus \otimes \spanF$ is defined by 
\begin{equation}
 \label{eq:CK_coproduct} 
 \DeltaCK(\Gamma) = \Gamma\otimes\unit + \unit\otimes\Gamma 
+ \sum_{\substack{\unit\neq\overline\Gamma\subsetneq\Gamma, \overline\Gamma \in \Fminus}} 
\overline\Gamma\otimes(\Gamma/\overline\Gamma)\;, 
\end{equation} 
where the sum ranges over all \emph{divergent} subgraphs $\overline\Gamma$, and 
$\Gamma/\overline\Gamma$ denotes the graph obtained by replacing $\overline\Gamma$
by a single vertex. The subgraphs have to be \emph{full}, in the sense that if an edge
$e$ belongs to $\overline\Gamma$, all edges connecting the same vertices also belong 
to $\overline\Gamma$. 
\end{definition}

The coproduct can be extended multiplicatively to a map 
$\DeltaCK: \spancF \to \spancFminus \otimes \spancF$.

\begin{example}
We have 
\begin{equation}
 \DeltaCK \Bigpar{\FGIIIplus} 
 = \FGIIIplus \otimes \unit 
 + \unit \otimes \FGIIIplus 
 + \FGIII \otimes \FGII\;.
\end{equation} 
\end{example}

 We endow $\spanF$ with two more linear maps. A \emph{counit} $\unit^\star: \spanF \to\R$, given by projection on the unit $\unit$, and an antipode, defined as follows. 
 
\begin{definition}[Antipode]
The \emph{antipode} $\cA:\spanF\to\spancF$ is defined inductively by $\cA(\unit)=\unit$ and 
\begin{align}
 \cA(\Gamma) &= -\Gamma - 
 \sum_{\substack{\unit\neq\overline\Gamma\subsetneq\Gamma, \overline\Gamma \in \Fminus}} 
\cA(\overline\Gamma)\cdot(\Gamma/\overline\Gamma)\\
&= - \Gamma - \cM(\cA\otimes\id) \DeltaCKred(\Gamma)\;.
\label{eq:antipode} 
\end{align} 
Here $\DeltaCKred = \DeltaCK - \Gamma\otimes\unit - \unit\otimes\Gamma$ 
denotes the \emph{reduced coproduct}, and the map $\cM:\spancF\otimes \spancF\to\spancF$ denotes multiplication, defined 
by $\cM(\Gamma_1\otimes\Gamma_2) = \Gamma_1 \cdot \Gamma_2$. 
\end{definition}

Both the counit $\unit^\star$ and the antipode $\cA$ can be extended multiplicatively 
to the whole algebra $\cF$. The space $(\cF,\cdot,\DeltaCK,\unit,\unit^\star,\cA)$ constructed 
in this way is a Hopf algebra, called \emph{Connes--Kreimer extraction-contraction Hopf algebra}. 
This means in particular that we have 
\begin{align}
 (\id\otimes\Delta)\Delta\Gamma &= (\Delta\otimes\id)\Delta\Gamma\;, \\
 \cM(\cA\otimes\id)\Delta\Gamma &= \cM(\id\otimes\cA)\Delta\Gamma
 = \unit^\star(\Gamma)\unit
\end{align}
for all $\Gamma\in\cF$. We have already encountered the first property, 
called~\emph{co-associativity}, in the case of polynomials, see Remark~\ref{rem:Hopf}. 

To define BPHZ renormalisation, we first introduce the \emph{twisted antipode},  
defined as 
\begin{equation}
 \tilde \cA(\Gamma) = \cA(\Gamma) 1_{\deg\Gamma\leqs0}\;.
\end{equation} 
Note that if $\deg(\Gamma)\leqs 0$, then one has 
\begin{equation}
\label{eq:twisted_antipode} 
 \tilde \cA(\Gamma) = - \Gamma - \cM(\tilde\cA\otimes\id) \DeltaCKred(\Gamma)\;,
\end{equation} 
because $\DeltaCKred$ produces only divergent terms on the left of the tensor product. 

A \emph{character} on $\spancF$ is a linear map $g:\spancF\to\R$ which is multiplicative 
in the sense that we have $g(\Gamma_1\cdot\Gamma_2) = g(\Gamma_1)g(\Gamma_2)$ for all $\Gamma_1, \Gamma_2\in\spancF$. 
With any character $g$, one can associate a linear map 
$M^g$ defined by 
\begin{equation}
 M^g(\Gamma) = (g\otimes\id)\DeltaCK(\Gamma)\;,
\end{equation} 
and the set of these maps is known to form a group. 
The \emph{BPHZ character} is the linear map $g^{\text{BPHZ}}: \spancF\to\R$ given by 
\begin{equation}
 g^{\text{BPHZ}}(\Gamma) = \Pi_N \tilde\cA(\Gamma)\;.
\end{equation} 
The fact that $g^{\text{BPHZ}}$ is indeed a character follows from multiplicativity 
of $\cA$ and $\Pi_N$. The map $\smash{M^{g^{\text{BPHZ}}}}$ is called 
\emph{BPHZ renormalisation map}. 
It defines a \emph{renormalised valuation} given by 
\begin{equation}
\label{eq:PiNBPHZ} 
 \PiNBPHZ(\Gamma) 
 = \Pi_N M^{g^{\text{BPHZ}}}(\Gamma)
 = (g^{\text{BPHZ}}\otimes\Pi_N)\DeltaCK(\Gamma)
 = (\Pi_N\tilde \cA \otimes \Pi_N)\DeltaCK(\Gamma)\;.
\end{equation} 
The interest of this construction is the following result. 

\begin{lemma}
\label{lem:BPHZ} 
The BPHZ renormalised valuation satisfies 
\begin{equation}
 \PiNBPHZ(\Gamma) = 
 \begin{cases}
  0 & \text{if\/ $\deg\Gamma\leqs0$\;,}\\
  -\Pi_N\cA(\Gamma) & \text{if\/ $\deg\Gamma > 0$\;.}
 \end{cases}
\end{equation} 
\end{lemma}
\begin{proof}
In the case $\deg\Gamma\leqs0$, using~\eqref{eq:twisted_antipode} we get 
\begin{align}
 \PiNBPHZ(\Gamma) 
 ={}& (\Pi_N\otimes\Pi_N)(\tilde\cA\otimes\id)
 \bigbrak{\Gamma\otimes\unit + \unit\otimes\Gamma + \DeltaCKred(\Gamma)}\\
 ={}& (\Pi_N\otimes\Pi_N)\bigbrak{\tilde\cA(\Gamma)\otimes\unit + \unit\otimes\Gamma 
 + (\tilde\cA\otimes\id)\DeltaCKred(\Gamma)} \\
 ={}& (\Pi_N\otimes\Pi_N)
 \bigbrak{-\Gamma\otimes\unit 
 - \cM(\tilde\cA\otimes\id) \DeltaCKred(\Gamma)\otimes\unit 
 + \unit\otimes\Gamma + (\tilde\cA\otimes\id)\DeltaCKred(\Gamma)}\;,
\end{align}
which vanishes by multiplicativity of $\Pi_N$. 
In the case $\deg\Gamma > 0$, using $\tilde\cA(\Gamma) = 0$ in the second line of the 
above computation, we obtain 
\begin{align}
 \PiNBPHZ(\Gamma)
 &= (\Pi_N\otimes\Pi_N)\bigbrak{\unit\otimes\Gamma 
 + (\tilde\cA\otimes\id)\DeltaCKred(\Gamma)} \\
 &= \Pi_N(\Gamma) + (\Pi_N\tilde\cA \otimes \Pi_N)\DeltaCKred(\Gamma) \\
 &= \Pi_N(\Gamma) + \Pi_N \cM(\tilde\cA \otimes \id)\DeltaCKred(\Gamma)
\end{align}
again by multiplicativity of $\Pi_N$. This is equal to $-\Pi_N\cA(\Gamma)$
by~\eqref{eq:antipode}. 
\end{proof}

If follows from Theorem~\ref{thm:BPHZ} that the renormalized valuation $\PiNBPHZ$
is bounded uniformly in the cut-off $N$ for any $\Gamma\in\spancF$. 

%%%%%%%%%%%%%%%%%%%%%%%%%%%%%%%%%%%%%%%%%%%%%%%%%%%%%%%%%%%%%%%%%%%%%%%%%%%%%%%%

\subsection{Wick map*}
\label{ssec:phi43_Wick} 

We now indicate how the above results allow to prove Theorem~\ref{thm:phi43}.
We again focus on the problem of computing the ratio of partition 
functions. Writing as before 
\begin{equation}
 \bX = \int_{\Lambda}\Wick{\phi(x)^4}_{C_N}\6x\;,
 \qquad 
 \bY = \int_{\Lambda}\Wick{\phi(x)^2}_{C_N}\6x\;,
\end{equation} 
this ratio is given by 
\begin{equation}
\label{eq:Zratio_3} 
 \frac{\cZ_{3,\alpha,N}}{\cZ_{3,0,N}}  
 = \bigexpecin{\mu_{3,0,N}}{\e^{-\alpha\bX-\beta\bY-\gamma}}
 = \e^{-\gamma}\bigexpecin{\mu_{3,0,N}}{\e^{-\alpha\bX-\beta\bY}}\;,
\end{equation} 
where $\beta = \beta_N(\alpha)$ and $\gamma = \gamma_N(\alpha)$. 
A naive approach would be to expand first the exponential, and then 
use Newton's binomial formula to get
\begin{align}
 \bigexpecin{\mu_{3,0,N}}{\e^{-\alpha\bX-\beta\bY}} 
 &\asymp 
 \sum_{n\geqs0} \frac{1}{n!}
 \bigexpecin{\mu_{3,0,N}}{(-\alpha\bX-\beta\bY)^n} \\
 &= \sum_{n\geqs0} 
 \sum_{k=0}^n \frac{(-\alpha)^k}{k!} 
 \frac{(-\beta)^{n-k}}{(n-k)!} 
 \bigexpecin{\mu_{3,0,N}}{\bX^k \bY^{n-k}}\;.
\end{align} 
The first terms of this expansion are 
\begin{align}
 \bigexpecin{\mu_{3,0,N}}{\e^{-\alpha\bX-\beta\bY}}
 \asymp  
 1 &{}+ 4! \alpha^2\, \Pi_N\bigpar{\FGIV} + 2! \beta^2\, \Pi_N\bigpar{\FGII} \\
 &{}-\binom{4}{2}^3 2^3 \alpha^3 \Pi_N\bigpar{\FGVI} 
 - 3 (4^2\cdot 2\cdot 3!)\alpha^2\beta\,\Pi_N\bigpar{\FGIIIplus} \\
 &{}- 3\cdot  4! \alpha\beta^2\, \Pi_N\bigpar{\FGIItwice}
 - 8\beta^3 \,\Pi_N\bigpar{\FGtriangle} + \dots\;.
 \label{eq:expansion_eXY3} 
\end{align}
The problem is that some Feynman diagrams are divergent, while others 
are not, and $\beta$, being of order $\alpha^2\log(N)$ is itself 
divergent. Therefore, the whole asymptotic series can have non-divergent 
coefficients only if there are many cancellations between divergent 
terms. The combinatorics of this is very hard to keep track of. 

An alternative reformulation is as follows. We introduce a 
linear map $\cP:\R[\bX,\bY] \to \spanF$ associating with a monomial 
$\bX^n\bY^m$ all possible diagrams obtained by perfect pairwise matchings 
of $n$ vertices with $4$ legs each and $m$ vertices with $2$ legs each, 
and projecting on connected diagrams. In this way, by the linked-cluster 
theorem, 
\begin{equation}
\label{eq:log_expec_mu3} 
\log \bigexpecin{\mu_{3,0,N}}{\e^{-\alpha\bX-\beta\bY}}
 = \Pi_N \circ \cP(\e^{-\alpha\bX-\beta\bY})\;.
\end{equation} 
The following theorem, inspired by ideas in~\cite{EFPTZ18,bruned2025renormalisingfeynmandiagramsmultiindices}, 
has been proved in~\cite{BKT}. 

\begin{theorem}[Commutative diagram]
\label{thm:commutative} 
There exist linear maps $\W:\R[\bX]\to\R[\bX,\bY]$, 
and $\Theta_\tF: \spanF\to\spanF$ such that the following 
diagram commutes:
\begin{equation}  
\label{comdiag:C[X]G}
\begin{tikzcd}[column sep=large, row sep=large]
\R[\bX]
\arrow[r, "\cP"] 
\arrow[d, "\W"']
& \spanF
\arrow[d, "(\Pi_N\tilde{\cA}\otimes\id)\DeltaCK + \Theta_\tF"] 
\\
\displaystyle
\R[\bX,\bY]
\arrow[r, "\cP"] 
& \spanF
\end{tikzcd}
\end{equation}
The map $\W$, called \emph{Wick map}, satisfies 
\begin{equation}
 \W(\bX^n) = H_n(\bX;-\beta\bY) \quad\forall n\geqs2
 \qquad\text{and}\qquad
 \W(\e^{-\alpha\bX}) = \e^{-\alpha\bX-\beta\bY}\;,
\end{equation} 
where $H_n$ is the $n$th scaled Hermite polynomial, 
while the map $\Theta_\tF$ satisfies 
\begin{equation}
 (\Pi_N\Theta_\tF\circ\cP)(\e^{-\alpha \bX})
 = - (\Pi_N\tilde\cA \circ \cP)(\e^{-\alpha \bX})
 = \gamma\;.
 \label{eq:Theta} 
\end{equation} 
\end{theorem}

The main idea of this result is that the complicated map 
$(\Pi_N\tilde{\cA}\otimes\id)\DeltaCK$, acting on Feynman diagrams, can be 
replaced by the much simpler map $\W$, acting on polynomials. The Wick map 
is of the form 
\begin{equation}
 \W = (\exp_\ast(-\kappa)\otimes\id)\Delta
\end{equation} 
that we have encountered in~\eqref{eq:Wick_expo}, where 
\begin{equation}
 \kappa(x^n) = 
 \begin{cases}
  \beta \bY & \text{if $n=2$\;,}\\
  0 & \text{otherwise\;.}
 \end{cases}
\end{equation} 
The intuition behind this is that the only subdivergences are graphs 
containing one or several \lq\lq bubbles\rq\rq\ $\FGIII$ as subgraphs. 
The effect of the map $(\Pi_N\tilde{\cA}\otimes\id)\DeltaCK$ is to 
extract bubbles and replace them by $\beta$ times a vertex of arity 
$2$, which can be seen as replacing $\bX^2$ by $\beta\bY$. This is also 
compatible with the combinatorial interpretation of Hermite polynomials 
we have seen in Section~\ref{ssec:combinatorics}. 
As for the relation $\W(\e^{-\alpha\bX}) = \e^{-\alpha\bX-\beta\bY}$,
it is a consequence of Proposition~\ref{prop:Wick_map}. 
One should note that we are working here with a \lq\lq second 
level Wick renormalisation\rq\rq, the first level being associated 
with using Wick powers in the energy, and the second level taking
care of the remaining subdivergences.

Together with the definition~\eqref{eq:PiNBPHZ} of the BPHZ valuation, 
Theorem~\ref{thm:commutative} implies that the following diagram commutes: 
\begin{equation}  
\label{comdiag:BPHZ}
\begin{tikzcd}[column sep=large, row sep=large]
\e^{-\alpha \bX}
\arrow[r, "\cP", mapsto] 
\arrow[d, "\W"', mapsto]
& \cP(\e^{-\alpha \bX})
\arrow[d, "(\Pi_N\tilde{\cA}\otimes\id)\DeltaCK + \Theta_\tF", mapsto] 
\arrow[rrd, bend left, "\PiNBPHZ + \Pi_N\Theta_\tF"]
\\
\e^{-\alpha \bX - \beta \bY}
\arrow[r, "\cP", mapsto] 
& \cP(\e^{-\alpha \bX - \beta \bY})
\arrow[rr, "\Pi_N", mapsto] 
&& \R
\end{tikzcd}
\end{equation}
This has the following consequence. On the one hand, \eqref{eq:Zratio_3} 
and~\eqref{eq:log_expec_mu3} show that 
\begin{equation}
\label{eq:logZ_sum} 
 \log\frac{\cZ_{3,\alpha,N}}{\cZ_{3,0,N}}
 = \Pi_N\circ\cP(\e^{-\alpha \bX-\beta \bY}) -\gamma\;.
%  = \PiNBPHZ\circ\cP(\e^{-\alpha \bX})\;.
\end{equation} 
On the other hand, by commutativity and~\eqref{eq:Theta}, we have 
\begin{align}
\Pi_N\circ\cP(\e^{-\alpha \bX-\beta \bY})
&= \bigpar{\PiNBPHZ + \Pi_N\Theta_\tF}\circ\cP(\e^{-\alpha\bX}) \\
&= \PiNBPHZ \circ \cP(\e^{-\alpha\bX}) + \gamma\;.
\end{align}
It follows that 
\begin{equation}
\label{eq:log_Zratio3} 
 \log\frac{\cZ_{3,\alpha,N}}{\cZ_{3,0,N}} 
 = \PiNBPHZ \circ \cP(\e^{-\alpha\bX}) 
 \asymp \sum_{n\geqs1} \frac{(-\alpha)^n}{n!} \PiNBPHZ\circ \cP(\bX^n)\;.
\end{equation} 
Now Lemma~\ref{lem:BPHZ} shows that $\PiNBPHZ\circ \cP(\bX^n)$ is different 
from $0$ only if $\cP(\bX^n)$ has strictly positive degree, in which case 
it is equal to $-\Pi_N\cA\circ\cP(\bX^n)$. By Theorem~\ref{thm:BPHZ}, 
these terms are bounded uniformly in $N$. This completes the proof 
of Theorem~\ref{thm:phi43} in the case of the ratio of partition functions. 

In fact, all Feynman diagrams in $\cP(\bX^n)$ have $n$ vertices and $2n$ 
edges, as they result from pairing $4n$ legs, or half-edges. Therefore,
\begin{equation}
 \deg(\cP(\bX^n)) = 3(n-1)-2n = n - 3\;,
\end{equation} 
which is strictly positive for all $n\geqs4$. It follows that the 
sum~\eqref{eq:log_Zratio3} starts at $n=4$. On the other hand, the terms 
of order $\alpha^2$ and $\alpha^3$ in~\eqref{eq:expansion_eXY3} are 
compensated by $\gamma$.

%%%%%%%%%%%%%%%%%%%%%%%%%%%%%%%%%%%%%%%%%%%%%%%%%%%%%%%%%%%%%%%%%%%%%%%%%%%%%%%%

\section{The $\Phi^4_{4-\eps}$ model*}
\label{sec:phi44} 

Now that the $\Phi^4_d$ model is understood in dimensions $1$, $2$ and $3$, 
one may wonder what happens in higher dimensions. In fact, Fr\"ohlich has 
shown~\cite{Froehlich_1982} that the model in trivial for any $d > 4$, 
while Aizenmann and Duminil--Copin have shown~\cite{Aizenmann-Duminil-Copin} 
that it is trivial for $d = 4$ as well, for any reasonable renormalisation procedure. 
Here \emph{trivial} means that in the limit $N\to\infty$, the $n$-point functions 
are the same as for the Gaussian model, as given by Isserlis' theorem. 

This leaves the question of what happens in dimensions $d\in(3,4)$. Non-integer 
dimensions in that interval can be interpreted as working on the three-dimensional torus, 
but changing the behaviour of the Green function to
\begin{equation}
 G(x) \asymp \frac{1}{\norm{x}^{d-2}}\;.
\end{equation} 
The degree of $\cP(\bX^n)$, computed via~\eqref{eq:def_degree}, becomes 
\begin{equation}
 \deg\cP(\bX^n) = 4n - (n+1)d\;.
\end{equation} 
This means that more diagrams in $\cP(\bX^n)$ can become divergent when 
$d$ increases. In fact, we have 
\begin{equation}
 \deg \cP(X^n) \leqs 0
 \qquad \Leftrightarrow \qquad 
 d \geqs d^*_{\textup{e}}(n) 
 := 4 - \frac{4}{n+1}\;.
\end{equation} 
Note that these threshold accumulate at $d = 4$ as $n\to\infty$. 
Furthermore, one can check that new subdivergences appear whenever $d$ 
crosses one of the thresholds
\begin{equation}
 d^*_{\textup{m}}(n) 
 = d^*_{\textup{e}}(2n-1)
 = 4 - \frac{2}{n}\;.
\end{equation} 
Table~\ref{tab:subdivergence} shows the first few of these new subdivergences. 
We introduce the notations 
\begin{equation}
  n^*_{\textup{e}}(d) = \biggl\lfloor \frac{d}{4-d} \biggr\rfloor
 \qquad \text{and} \qquad 
 n^*_{\textup{m}}(d) = \biggl\lfloor \frac{2}{4-d} \biggr\rfloor
\end{equation} 
for the inverse thresholds of $d^*_{\textup{m}}(n)$ and $d^*_{\textup{e}}(n)$.

\begin{table}[t]
  \centering
\begin{tabular}{|c|c|c|c|}
\hline
Graphs & Degree & Critical $d$ & Minimal $n$ \\
\hline
\vrule height 16pt depth 8pt width 0pt
$\FGIII$ &  $6 - 2d$ & $3 = d^*_{\text{m}}(2)$ & $4$ \\
\hline
\vrule height 20pt depth 8pt width 0pt
$\FGTwoTwoOne$ $\FGIIIplus$ &  $10 - 3d$ & $\frac{10}3 = d^*_{\text{m}}(3)$ & $5$ \\ 
% \vrule height 16pt depth 8pt width 0pt
% $\FGIIIplus$ &   &  &  \\ 
\hline
\vrule height 20pt depth 12pt width 0pt
$\FGQuadA$ $\FGQuadB$ $\FGQuadD$ $\FGQuadC$ &  $14 - 4d$ & $\frac72 = d^*_{\text{m}}(4)$ & $6$ \\ 
% \vrule height 16pt depth 12pt width 0pt
% $\FGQuadC$ &   &  &  \\ 
\hline
\end{tabular}
\caption{List of the first divergent subdiagrams of the $\Phi^4_d$ model, with their degree, 
the value of $d$ for which they become divergent, and the minimal value of $n$ such that they 
occur in $\cP(\bX^n)$.}
\label{tab:subdivergence} 
\end{table}

In order to deal with these subdivergences, we have to allow for more general 
replacement rules than $\bX^2 \mapsto \beta\bY$. Therefore, the relevant Wick map 
$\W$ no longer involves Hermite polynomials, but more general so-called 
\emph{Bell polynomials}. 
 
%%%%%%%%%%%%%%%%%%%%%%%%%%%%%%%%%%%%%%%%%%%%%%%%%%%%%%%%%%%%%%%%%%%%%%%%%%%%%%%%

\subsection{Bell polynomials*}
\label{ssec:phi44_Bell} 

Bell polynomials can be defined by starting with cumulants 
\begin{equation}
\kappa(x^n) = 
\begin{cases}
0 & \text{if $n=1$\;,}\\
y_n & \text{otherwise\;,}
\end{cases}
\end{equation} 
where the $y_n$ are considered for now as parameters. 
The associated Wick map is
\begin{equation}
 \W(t,x) = \e^{tx - K(t)}
 = \exp\biggset{tx - \sum_{n\geqs2}y_n\frac{t^n}{n!}}\;,
\end{equation} 
where $K(t) = \Lambda(\kappa)(t)$, see~\eqref{eq:def_Lambda}. 

\begin{definition}[Bell polynomials]
The Wick map $\W(t,x)$ is the generating function of Bell polynomials, 
in the sense that 
\begin{equation}
 \W(t,x) = \sum_{n\geqs0} B_n(x, -y_2, \dots, -y_n) \frac{t^n}{n!}\;.
\end{equation} 
$B_n$ is called the \emph{$n$th complete Bell polynomial}. It can be decomposed 
as 
\begin{equation}
 B_n(x, -y_2, \dots, -y_n)
 = \sum_{k=1}^n B_{n,k}(x, -y_2, \dots, -y_{n-k+1})\;,
\end{equation} 
where each $B_{n,k}$ is the homogeneous part of degree $k$ of $B_n$, and 
is called \emph{incomplete Bell polynomial} of order $(n,k)$. 
\end{definition}

Using the fact (see Section~\ref{ssec:convolution}) that
\begin{equation}
 B_n(x, -y_2, \dots, -y_n)
= \W(x^n) 
= \sum_{k=0}^n \binom{n}{k} \mu^{-1}(x^k)x^{n-k}
\end{equation} 
where $\mu^{-1} = \exp_\ast(-\kappa)$, one finds the explicit expression 
\begin{equation}
 B_n(x, -y_2, \dots, -y_n) 
 = n! \sum_{k=0}^n \sum_{p=1}^k \frac{(-1)^p}{p!}
 \sum_{\substack{n_1,\dots,n_p\geqs2\\n_1 + \dots + n_p=k}}
 \frac{y_{n_1}}{n_1!} \dots \frac{y_{n_p}}{n_p!}
 \frac{x^{n-k}}{(n-k)!}\;.
\end{equation} 
The incomplete Bell polynomial $B_{n,k}$ is simply the $k$th term in this sum.
In particular, comparing with the explicit expression~\eqref{eq:scaled_Wick}
of scaled Wick polynomials, we recover 
\begin{equation}
 B_n(x,-y_2,0,\dots,0) = H_n(x;y_2)\;.
\end{equation} 
Bell polynomials also have a simple combinatorial interpretation: the coefficients
of $B_{n,k}$ count the number of partitions of a set of cardinality $n$ into 
$k$ subsets, whose sizes are encoded into the monomial. 

\begin{example}
One has 
\begin{equation}
 B_{5,3}(x,y_2,y_3) = 15xy_2^2 + 10x^2y_3\;.
\end{equation} 
% $B_{5,3}(x,y_2,y_3) = 15xy_2^2 + 10x^2y_3$. 
Therefore, there are $15$ 
ways of partitioning a set of $5$ elements into $3$ subsets of sizes $1$, $2$ 
and $2$, and $10$ ways of partitioning it into $3$ subsets of sizes $1$, $1$ 
and $3$. The polynomial $B_{5,3}(x,y_2,y_3)$ is also obtained by applying 
the substitutions $x^2\mapsto y_2$ and $x^3\mapsto y_3$ to the monomial
$x^5$ in all possible ways, and keeping only terms of degree $3$. 
\end{example}

%%%%%%%%%%%%%%%%%%%%%%%%%%%%%%%%%%%%%%%%%%%%%%%%%%%%%%%%%%%%%%%%%%%%%%%%%%%%%%%%

\subsection{Wick map*}
\label{ssec:phi44_Wick} 

% We introduce the notations 
% \begin{equation}
%   n^*_{\textup{e}}(d) = \biggl\lfloor \frac{d}{4-d} \biggr\rfloor
%  \qquad \text{and} \qquad 
%  n^*_{\textup{m}}(d) = \biggl\lfloor \frac{2}{4-d} \biggr\rfloor
% \end{equation} 
% for the inverse thresholds of $d^*_{\textup{m}}(n)$ and $d^*_{\textup{e}}(n)$.
Theorem~\ref{thm:commutative} admits the following generalisation, also 
shown in~\cite{BKT}. 

\begin{theorem}[Commutative diagram]
\label{thm:commutative_4eps} 
There exist linear maps $\W:\R[\bX]\to\R[\bX,\bY]$
and $\Theta_\tF: \spanF\to\spanF$,
and constants $\sigma_n(N)$, diverging like $N^{2-(4-d)n}$, 
such that the diagram~\eqref{comdiag:C[X]G} commutes.
% \begin{equation}  
% \label{comdiag:C[X]G2}
% \begin{tikzcd}[column sep=large, row sep=large]
% \R[\bX]
% \arrow[r, "\cP"] 
% \arrow[d, "\W"']
% & \spanF
% \arrow[d, "(\Pi_N\tilde{\cA}\otimes\id)\DeltaCK + \Theta_\tF"] 
% \\
% \displaystyle
% \R[\bX,\bY]
% \arrow[r, "\cP"] 
% & \spanF
% \end{tikzcd}
% \end{equation}
The map Wick map $\W$ satisfies 
\begin{equation}
 \W(\e^{-\alpha\bX}) = \e^{-\alpha\bX-\beta\bY}
\end{equation}where 
\begin{equation}
 \beta = \beta_{N,d}(\alpha) 
 = \sum_{n=2}^{n^*_{\textup{m}}(d)} \frac{(-\alpha)^n}{n!}\sigma_n(N)\;.
\end{equation} 
Furthermore, 
\begin{equation}
 \W(\bX^n) = B_n(\bX,-\sigma_2(N)\bY,\dots,-\sigma_n(N)\bY)
\end{equation} 
holds for any $n\geqs2$, where $B_n$ is the $n$th complete Bell polynomial, 
while the map $\Theta_\tF$ satisfies 
\begin{equation}
 (\Pi_N\Theta_\tF\circ\cP)(\e^{-\alpha \bX})
 = - (\Pi_N\tilde\cA \circ \cP)(\e^{-\alpha \bX})
 = \gamma\;.
 \label{eq:Theta2} 
\end{equation} 
\end{theorem}

The counterterms $\sigma_n(N)$ can again be written explicitly in terms of 
valuations of divergent Feynman diagrams, the first of which are shown 
in Table~\ref{tab:subdivergence}. For given $n$, these diagrams 
have $n$ vertices, which are either two vertices of arity $3$ and $n-2$ 
vertices of arity $4$, or one vertex of arity $2$, and $n-1$ vertices of 
arity $4$. Using the same arguments as in Section~\ref{ssec:phi43_Wick}, 
we arrive at the following result. 

\begin{corollary}[Renormalisation of the $\Phi^4_d$ model with $d\in(3,4)$]
Let 
\begin{equation}
\label{eq:gamma} 
\gamma =  \gamma_{N,d}(\alpha) = 
 - \sum_{n=2}^{n^*_{\textup{e}}(d)}
 \frac{(-\alpha)^n}{n!} \Pi_N\tilde\cA(\cP(\bX^n))\;. 
\end{equation} 
Then the $\Phi^4_d$ model with $d\in(3,4)$, mass counterterm $\beta$ 
and energy counterterm $\gamma$ satisfies 
\begin{equation}
\label{eq:log_Zratio4} 
 \log\frac{\cZ_{d,\alpha,N}}{\cZ_{d,0,N}} 
%  = \PiNBPHZ \circ \cP(\e^{-\alpha\bX}) 
 \asymp -\sum_{n\geqs n^*_{\textup{e}}(d)} \frac{(-\alpha)^n}{n!} 
 \Pi_N\cA(\cP(\bX^n))\;,
\end{equation} 
where the coefficients are bounded uniformly in the cut-off $N$. 
\end{corollary}

%\listoftheorems[ignoreall,show=theorem]

% \bibliographystyle{plain}
\cleardoublepage
\phantomsection
\bibliographystyle{alpha}
\addcontentsline{toc}{chapter}{Bibliography}
\bibliography{Lammi}

@Preamble{
"\def\cprime{$'$} "
}

@Article{AF,
  Title                    = {Topological properties of linearly coupled expanding map lattices},
  Author                   = {Afraimovich, Valentin and Fernandez, Bastien},
  Journal                  = {Nonlinearity},
  Year                     = {2000},
  Number                   = {4},
  Pages                    = {973--993},
  Volume                   = {13},

  Fjournal                 = {Nonlinearity}
}

@Article{Barashkov_Gubinelli_18,
  author    = {Barashkov, Nikolay and Gubinelli, Massimiliano},
  title     = {A variational method for $\Phi^4_3$},
  journal   = {Duke Mathematical Journal},
  year      = {2020},
  volume    = {169},
  number    = {17},
  pages     = {3339 -- 3415},
  note      = {\texttt{arXiv:1805.1081}},
  doi       = {10.1215/00127094-2020-0029},
  publisher = {Duke University Press},
  url       = {https://doi.org/10.1215/00127094-2020-0029},
}

@Article{BCGNOPS_80,
  Title                    = {Ultraviolet stability in {E}uclidean scalar field theories},
  Author                   = {Benfatto, G. and Cassandro, M. and Gallavotti, G. and Nicol\`o, F. and Olivieri, E. and Presutti, E. and Scacciatelli, E.},
  Journal                  = {Comm. Math. Phys.},
  Year                     = {1980},
  Number                   = {2},
  Pages                    = {95--130},
  Volume                   = {71},

  Fjournal                 = {Communications in Mathematical Physics},
  ISSN                     = {0010-3616},
  Mrclass                  = {81E10 (60G20)},
  Mrnumber                 = {560344},
  Mrreviewer               = {Francesco Guerra},
  Url                      = {http://projecteuclid.org/euclid.cmp/1103907452}
}

@Article{BCGNOPS78,
  Title                    = {Some probabilistic techniques in field theory},
  Author                   = {Benfatto, G. and Cassandro, M. and Gallavotti, G. and Nicol\`o,
 F. and Olivieri, E. and Presutti, E. and Scacciatelli, E.},
  Journal                  = {Comm. Math. Phys.},
  Year                     = {1978},
  Number                   = {2},
  Pages                    = {143--166},
  Volume                   = {59},

  Fjournal                 = {Communications in Mathematical Physics},
  ISSN                     = {0010-3616},
  Mrclass                  = {81E15 (60G20 60H15 60J45 82A05)},
  Mrnumber                 = {491611},
  Mrreviewer               = {Sergio Albeverio},
  Url                      = {http://projecteuclid.org/euclid.cmp/1103901608}
}

@Book{Bogoliubov56,
  Title                    = {On a new form of adiabatic perturbation theory in the problem of particle interaction with a quantum field},
  Author                   = {Bogoliubov, N. N.},
  Publisher                = {Translated by Morris D. Friedman, 572 California St., Newtonville 60, Mass.},
  Year                     = {1956}
}

@Article{Bovier_Felder_CMP84,
  Title                    = {Skeleton inequalities and the asymptotic nature of perturbation theory for {$\Phi ^{4}$}-theories in two and three dimensions},
  Author                   = {Bovier, Anton and Felder, Giovanni},
  Journal                  = {Comm. Math. Phys.},
  Year                     = {1984},
  Number                   = {2},
  Pages                    = {259--275},
  Volume                   = {93},

  Coden                    = {CMPHAY},
  Fjournal                 = {Communications in Mathematical Physics},
  ISSN                     = {0010-3616},
  Mrclass                  = {81E25 (81E08 81E15 82A68)},
  Mrnumber                 = {742195},
  Mrreviewer               = {Alan D. Sokal},
  Url                      = {http://projecteuclid.org/euclid.cmp/1103941056}
}

@Article{Brydges_Dimock_Hurd_CMP_95,
  Title                    = {The short distance behavior of {$(\phi^4)_3$}},
  Author                   = {Brydges, D. and Dimock, J. and Hurd, T. R.},
  Journal                  = {Comm. Math. Phys.},
  Year                     = {1995},
  Number                   = {1},
  Pages                    = {143--186},
  Volume                   = {172},

  Fjournal                 = {Communications in Mathematical Physics},
  ISSN                     = {0010-3616},
  Mrclass                  = {81T08 (81T17)},
  Mrnumber                 = {1346375},
  Mrreviewer               = {Domingos H. U. Marchetti},
  Url                      = {http://projecteuclid.org/euclid.cmp/1104273962}
}

@Article{Brydges_Frohlich_Sokal_CMP83_RW,
  Title                    = {The random-walk representation of classical spin systems and correlation inequalities. {II}. {T}he skeleton inequalities},
  Author                   = {Brydges, David C. and Fr{\"o}hlich, J{\"u}rg and Sokal, Alan D.},
  Journal                  = {Comm. Math. Phys.},
  Year                     = {1983},
  Number                   = {1},
  Pages                    = {117--139},
  Volume                   = {91},

  Coden                    = {CMPHAY},
  Fjournal                 = {Communications in Mathematical Physics},
  ISSN                     = {0010-3616},
  Mrclass                  = {81E08 (81E25 82A67)},
  Mrnumber                 = {719815},
  Mrreviewer               = {Gerhard C. Hegerfeldt},
  Url                      = {http://projecteuclid.org/euclid.cmp/1103940478}
}

@Article{Brydges_Frohlich_Sokal_CPM83,
  Title                    = {A new proof of the existence and nontriviality of the continuum {$\varphi ^{4}_{2}$} and {$\varphi ^{4}_{3}$} quantum field theories},
  Author                   = {Brydges, David C. and Fr{\"o}hlich, J{\"u}rg and Sokal, Alan D.},
  Journal                  = {Comm. Math. Phys.},
  Year                     = {1983},
  Number                   = {2},
  Pages                    = {141--186},
  Volume                   = {91},

  Coden                    = {CMPHAY},
  Fjournal                 = {Communications in Mathematical Physics},
  ISSN                     = {0010-3616},
  Mrclass                  = {81E08},
  Mrnumber                 = {723546},
  Mrreviewer               = {Gerhard C. Hegerfeldt},
  Url                      = {http://projecteuclid.org/euclid.cmp/1103940528}
}

@Article{Chandra_Weber_LN17,
  Title                    = {Stochastic {PDE}s, regularity structures, and interacting
 particle systems},
  Author                   = {Chandra, Ajay and Weber, Hendrik},
  Journal                  = {Ann. Fac. Sci. Toulouse Math. (6)},
  Year                     = {2017},
  Number                   = {4},
  Pages                    = {847--909},
  Volume                   = {26},

  Doi                      = {10.5802/afst.1555},
  Fjournal                 = {Annales de la Facult\'{e} des Sciences de Toulouse. Math\'{e}matiques.
 S\'{e}rie 6},
  ISSN                     = {0240-2963},
  Mrclass                  = {35R60 (35B65 35K59 60H15)},
  Mrnumber                 = {3746645},
  Mrreviewer               = {Alp O. Eden},
  Url                      = {https://doi.org/10.5802/afst.1555}
}

@Article{Feldman74,
  Title                    = {The {$\lambda \varphi ^{4}_{3}$} field theory in a finite
 volume},
  Author                   = {Feldman, Joel},
  Journal                  = {Comm. Math. Phys.},
  Year                     = {1974},
  Pages                    = {93--120},
  Volume                   = {37},

  Fjournal                 = {Communications in Mathematical Physics},
  ISSN                     = {0010-3616},
  Mrclass                  = {81.46},
  Mrnumber                 = {0384003},
  Mrreviewer               = {S. Aramaki},
  Url                      = {http://projecteuclid.org/euclid.cmp/1103859849}
}

@Book{Glimm_Jaffe_81,
  Title                    = {Quantum physics. A functional integral point of view},
  Author                   = {Glimm, James and Jaffe, Arthur},
  Publisher                = {Springer-Verlag, New York-Berlin},
  Year                     = {1981},

  ISBN                     = {0-387-90562-6},
  Mrclass                  = {81-02 (35R15 60K35 81C35 81E10 82-02)},
  Mrnumber                 = {628000},
  Mrreviewer               = {Paul Federbush},
  Pages                    = {xx+417}
}

@Article{Glimm_Jaffe_73,
  Title                    = {Positivity of the {$\phi ^{4}_{3}$} {H}amiltonian},
  Author                   = {Glimm, James and Jaffe, Arthur},
  Journal                  = {Fortschr. Physik},
  Year                     = {1973},
  Pages                    = {327--376},
  Volume                   = {21},

  Fjournal                 = {Fortschritte der Physik. Progress of Physics},
  ISSN                     = {0015-8208},
  Mrclass                  = {81.47},
  Mrnumber                 = {0408581},
  Mrreviewer               = {R. Hoegh-Krohn}
}

@Article{Glimm_Jaffe_68,
  Title                    = {A {$\lambda \phi^{4}$} quantum field without cutoffs. {I}},
  Author                   = {Glimm, James and Jaffe, Arthur},
  Journal                  = {Phys. Rev. (2)},
  Year                     = {1968},
  Pages                    = {1945--1951},
  Volume                   = {176},

  Mrclass                  = {81.47},
  Mrnumber                 = {0247845},
  Mrreviewer               = {K. Veseli\'{c}}
}

@Article{Gruber_Kunz_71,
  Title                    = {General properties of polymer systems},
  Author                   = {Gruber, C. and Kunz, H.},
  Journal                  = {Comm. Math. Phys.},
  Year                     = {1971},
  Pages                    = {133--161},
  Volume                   = {22},

  Fjournal                 = {Communications in Mathematical Physics},
  ISSN                     = {0010-3616},
  Mrclass                  = {82.62},
  Mrnumber                 = {321473},
  Mrreviewer               = {H. S. Green},
  Url                      = {http://projecteuclid.org/euclid.cmp/1103857447}
}

@InProceedings{Hairer_BPHZ,
  Title                    = {An Analyst's Take on the {BPHZ} Theorem},
  Author                   = {Hairer, Martin},
  Booktitle                = {Comput. Combin. Dyn. Stoch. Control},
  Year                     = {2018},

  Address                  = {Cham},
  Pages                    = {429--476},
  Publisher                = {Springer International Publishing},
  ISBN                     = {978-3-030-01593-0}
}

@TechReport{Hairer_LN_2009,
  author      = {Hairer, Martin},
  title       = {An introduction to stochastic {PDE}s},
  institution = {University of Warwick},
  year        = {2009},
  type        = {Lecture notes},
  note        = {{\tt http://arxiv.org/abs/0907.4178}},
}

@Book{IJ,
  Title                    = {Elementary Stability and Bifurcation Theory},
  Author                   = {Iooss, G\'erard and Joseph, Daniel D.},
  Publisher                = {Springer-Verlag},
  Year                     = {1990},

  Address                  = {New York},
  Edition                  = {Second}
}

@Book{nualart2006malliavin,
  Title                    = {The Malliavin calculus and related topics},
  Author                   = {Nualart, David},
  Publisher                = {Springer},
  Year                     = {2006},
  Volume                   = {1995}
}

@Article{daPrato_Tubaro,
  author    = {Da Prato, Giuseppe and Tubaro, Luciano},
  title     = {Wick powers in stochastic {PDE}s: an introduction},
  year      = {2007},
  publisher = {University of Trento},
}

@TechReport{Hairer_Malliavin26,
  author      = {Hairer, Martin},
  title       = {Advanced Stochastic Calculus},
  institution = {EPFL and Imperial College London},
  year        = {2026},
  type        = {Lecture notes},
}

@Article{Caravenna_Zambotti_20,
  author    = {Caravenna, Francesco and Zambotti, Lorenzo},
  title     = {{H}airer's reconstruction theorem without regularity structures},
  journal   = {EMS Surveys in Mathematical Science},
  year      = {2021},
  volume    = {7},
  number    = {2},
  pages     = {207--251},
  note      = {\texttt{arXiv:2005.09287}},
  url       = {https://arxiv.org/abs/2005.09287},
}

@Article{Weinberg60,
  author     = {Weinberg, Steven},
  title      = {High-energy behavior in quantum field-theory},
  journal    = {Phys. Rev. (2)},
  year       = {1960},
  volume     = {118},
  pages      = {838--849},
  issn       = {0031-899X},
  fjournal   = {Physical Review. Series II},
  mrclass    = {81.00},
  mrnumber   = {116953},
  mrreviewer = {O. Hara},
}

@Article{Hepp66,
  author    = {Hepp, K.},
  title     = {Proof of the {B}ogoliubov--{P}arasiuk theorem on renormalization},
  journal   = {Comm. Math. Phys.},
  year      = {1966},
  volume    = {2},
  number    = {4},
  pages     = {301--326},
}

@Article{Zimmermann69,
  author    = {Zimmermann, W.},
  title     = {Convergence of {B}ogoliubov's method of renormalization in momentum space},
  journal   = {Comm. Math. Phys.},
  year      = {1969},
  volume    = {15},
  pages     = {208--234},
  issn      = {0010-3616},
  fjournal  = {Communications in Mathematical Physics},
  mrclass   = {81.28},
  mrnumber  = {255162},
  url       = {http://projecteuclid.org/euclid.cmp/1103841945},
}

@Article{BB19,
  author  = {Berglund, Nils and Bruned, Yvain},
  title   = {{BPHZ} renormalisation and vanishing subcriticality asymptotics of the fractional $\Phi^3_d$ model},
  journal = {Stochastics and Partial Differential Equations: Analysis and Computations},
  year    = {2025},
  volume  = {13},
  pages   = {243-307},
  comment = {arXiv:1907.13028},
}

@Article{Connes_Kreimer_2000a,
  author    = {Connes, Alain and Kreimer, Dirk},
  title     = {Renormalization in quantum field theory and the {R}iemann-{H}ilbert problem. {I}. {T}he {H}opf algebra structure of graphs and the main theorem},
  journal   = {Comm. Math. Phys.},
  year      = {2000},
  volume    = {210},
  number    = {1},
  pages     = {249--273},
  issn      = {0010-3616},
  doi       = {10.1007/s002200050779},
  fjournal  = {Communications in Mathematical Physics},
  mrclass   = {81T15 (16W30 34M50 81T18)},
  mrnumber  = {1748177},
  url       = {https://doi.org/10.1007/s002200050779},
}

@Article{Connes_Kreimer_2001,
  author    = {Connes, Alain and Kreimer, Dirk},
  title     = {Renormalization in quantum field theory and the {R}iemann-{H}ilbert problem. {II}. {T}he {$\beta$}-function, diffeomorphisms and the renormalization group},
  journal   = {Comm. Math. Phys.},
  year      = {2001},
  volume    = {216},
  number    = {1},
  pages     = {215--241},
  issn      = {0010-3616},
  doi       = {10.1007/PL00005547},
  fjournal  = {Communications in Mathematical Physics},
  mrclass   = {81T15 (16W30 34M50 81T18)},
  mrnumber  = {1810779},
  url       = {https://doi.org/10.1007/PL00005547},
}

@Unpublished{BKT,
  author = {Berglund, Nils and Klose, Tom and Tapia, Nikolas},
  title  = {Perturbative renormalisation of the $\Phi^4_{4-\eps}$ model via generalized {W}ick maps},
  note   = {Preprint, \texttt{arXiv:2507.03820}},
  year   = {2025},
  url    = {https://arxiv.org/abs/2507.03820},
}

@Article{Aizenmann-Duminil-Copin,
  author    = {Michael {Aizenman} and Hugo {Duminil-Copin}},
  title     = {{Marginal triviality of the scaling limits of critical 4D Ising and \(\phi_4^4\) models}},
  journal   = {{Ann. Math. (2)}},
  year      = {2021},
  volume    = {194},
  number    = {1},
  pages     = {163--235},
  issn      = {0003-486X},
  doi       = {10.4007/annals.2021.194.1.3},
  fjournal  = {{Annals of Mathematics. Second Series}},
  language  = {English},
  msc2010   = {60K35 82B20 82B27 60G60},
  publisher = {Princeton University, Mathematics Department, Princeton, NJ},
}

@Article{Froehlich_1982,
  author  = {J{\"u}rg Fr{\"o}hlich},
  title   = {{On the Triviality of $\lambda \phi^4_d$ Theories and the Approach to the Critical Point in $d > 4$ Dimensions}},
  journal = {Nuclear Physics B},
  year    = {1982},
  volume  = {200},
  number  = {2},
  pages   = {281--296},
  doi     = {10.1016/0550-3213(82)90087-6},
}

@Article{Rivasseau_09,
  author    = {Rivasseau, V.},
  title     = {Constructive field theory in zero dimension},
  journal   = {Adv. Math. Phys.},
  year      = {2009},
  volume    = {2009},
  pages     = {12},
  issn      = {1687-9120},
  note      = {Id/No 180159},
  doi       = {10.1155/2009/180159},
  fjournal  = {Advances in Mathematical Physics},
  language  = {English},
  zbl       = {1201.81085},
  zbmath    = {5816333},
}

@Book{Salmhofer_Renormalization,
  title     = {{Renormalization. An Introduction.}},
  publisher = {Springer},
  year      = {1999},
  author    = {Salmhofer, Manfred},
  address   = {Berlin, Heidelberg, New York},
  edition   = {1st},
}

@Article{Brouder09,
  author    = {Brouder, Christian},
  title     = {Quantum field theory meets {Hopf} algebra},
  journal   = {Math. Nachr.},
  year      = {2009},
  volume    = {282},
  number    = {12},
  pages     = {1664--1690},
  issn      = {0025-584X},
  doi       = {10.1002/mana.200610828},
  fjournal  = {Mathematische Nachrichten},
  language  = {English},
  zbl       = {1182.81056},
  zbmath    = {5654789},
}

@Article{EFPTZ18,
  author    = {Ebrahimi-Fard, Kurusch and Patras, Fr{\'e}d{\'e}ric and Tapia, Nikolas and Zambotti, Lorenzo},
  title     = {Hopf-algebraic deformations of products and {Wick} polynomials},
  journal   = {Int. Math. Res. Not.},
  year      = {2020},
  volume    = {2020},
  number    = {24},
  pages     = {10064--10099},
  issn      = {1073-7928},
  doi       = {10.1093/imrn/rny269},
  fjournal  = {IMRN. International Mathematics Research Notices},
  language  = {English},
  zbl       = {1476.16033},
  zbmath    = {7323440},
}

@Article{bruned2025renormalisingfeynmandiagramsmultiindices,
  author        = {Yvain Bruned and Yingtong Hou},
  title         = {Renormalising {F}eynman diagrams with multi-indices},
  journal       = {Preprint, arXiv/2501.08151},
  year          = {2025},
  archiveprefix = {arXiv},
  eprint        = {2501.08151},
  primaryclass  = {math-ph},
  url           = {https://arxiv.org/abs/2501.08151},
}

@Book{Peccati_Taqqu_book,
  title      = {Wiener chaos: moments, cumulants and diagrams},
  publisher  = {Springer, Milan; Bocconi University Press, Milan},
  year       = {2011},
  author     = {Peccati, Giovanni and Taqqu, Murad S.},
  volume     = {1},
  series     = {Bocconi \& Springer Series},
  isbn       = {978-88-470-1678-1},
  note       = {A survey with computer implementation, Supplementary material available online},
  doi        = {10.1007/978-88-470-1679-8},
  mrclass    = {60H05 (05D40 60-08 60G15 60H30 65Cxx)},
  mrnumber   = {2791919},
  mrreviewer = {Sergey V. Lototsky},
  pages      = {xiv+274},
  url        = {https://doi.org/10.1007/978-88-470-1679-8},
}

@Book{Janson_book_08,
  title     = {Gaussian {Hilbert} spaces},
  publisher = {Cambridge: Cambridge University Press},
  year      = {2008},
  author    = {Janson, Svante},
  volume    = {129},
  series    = {Camb. Tracts Math.},
  edition   = {Reprint of the 1997 hardback ed.},
  isbn      = {978-0-521-05720-2},
  doi       = {10.1017/CBO9780511526169},
  fseries   = {Cambridge Tracts in Mathematics},
  issn      = {0950-6284},
  language  = {English},
  zbl       = {1143.60005},
  zbmath    = {5321086},
}

@Book{Sanz_Sole_book,
  title     = {Malliavin calculus with applications to stochastic partial differential equations},
  publisher = {Boca Raton, FL: CRC Press; Lausanne: EPFL Press},
  year      = {2005},
  author    = {Sanz-Sol{\'e}, Marta},
  isbn      = {0-8493-4030-6; 2-940222-06-1; 978-1-4398-1894-7},
  doi       = {10.1201/9781439818947},
  language  = {English},
  zbl       = {1098.60050},
  zbmath    = {2200693},
}

@Book{NB_SPDE_book,
  title     = {An introduction to singular stochastic {PDEs}. {Allen}-{Cahn} equations, metastability, and regularity structures},
  publisher = {Berlin: European Mathematical Society (EMS)},
  year      = {2022},
  author    = {Berglund, Nils},
  series    = {EMS Ser. Lect. Math.},
  isbn      = {978-3-98547-014-3; 978-3-98547-514-8},
  doi       = {10.4171/ELM/34},
  fseries   = {EMS Series of Lectures in Mathematics},
  issn      = {2523-5176},
  language  = {English},
  zbl       = {1511.35001},
  zbmath    = {7533266},
}

@TechReport{Tubaro_Zanella25,
  author      = {Tubaro, Luciano and Zanella, Margherita},
  title       = {An introduction to {M}alliavin calculus},
  institution = {Trento, Milano.},
  year        = {2025},
  note        = {\texttt{arxiv:2502.07941}},
}

\vfill

\bigskip\bigskip\noindent
{\small
Nils Berglund \\
Institut Denis Poisson (IDP) \\ 
Universite d'Orleans, Universite de Tours, CNRS -- UMR 7013 \\
B\^atiment de Mathematiques, B.P. 6759\\
45067~Orleans Cedex 2, France \\
{\it E-mail address: }
{\tt nils.berglund@univ-orleans.fr} \\
{\tt https://www.idpoisson.fr/berglund} 

\end{document}